\nonstopmode
\pdfoutput=1
\documentclass[twoside]{article}
 \usepackage[usenames]{color}
 \definecolor{white}{rgb}{1,1,1}

 \usepackage{amsmath}
 \usepackage{amscd}
 \usepackage{amsfonts}
 \usepackage{amssymb}
 \usepackage{amsthm}
 \usepackage{euscript}
 \usepackage{graphicx}
 \usepackage{epsfig}
 \usepackage{hyperref}

\newtheorem{theorem}{Theorem}
\newtheorem*{thm_mf}{Theorem (Mario's formula)}
\newtheorem{corollary}[theorem]{Corollary}
\newtheorem{proposition}[theorem]{Proposition}
\newtheorem{lemma}[theorem]{Lemma}

\theoremstyle{definition}

\newtheorem*{remark}{Remark}
\newtheorem*{remarks}{Remarks}
\newtheorem*{summary}{Summary}
\newtheorem*{notation}{Notation}
\newtheorem*{conclusion}{Conclusion}

\newcommand{\gamz}{\gamma_0}
\newcommand{\gam}{\gamma_{12}}
\newcommand{\gamp}{\gamma'_{12}}

\newcommand{\kermat}{\mathbf{ K}}
\newcommand{\pari}{\partial_i}
\newcommand{\parj}{\partial_j}

\newcommand{\idmatrix}{\mathbb I}
\newcommand{\invg}{Q}
\newcommand{\vastb}{v^\ast(\overline{q})}
\newcommand{\vasth}{v^\ast(\hat{q})}

\newcommand{\gul}[2]{g^{#1}_{\textcolor{white}{#1},#2\ }}

\begin{document}
\begin{center}
{\large\bf Sectional Curvature in terms of the Cometric,
with 
\\ \vspace{.2cm}
\mbox{Applications to the Riemannian Manifolds of Landmarks}}
\par\vspace{1cm}
\hspace*{-.8cm}
\begin{tabular}{ccc}
Mario Micheli
& Peter W.~Michor
& David Mumford \\
Department of Mathematics
& Fakult\"at f\"ur Mathematik
& Div.~of Applied Mathematics \\
Univ.~of California, Los Angeles
& Universit\"at Wien 
& Brown University \\
520 Portola Plaza
& Nordbergstrasse 15
& 182 George Street \\
Los Angeles, CA 90095, USA
& A-1090 Wien, Austria
& Providence, RI 02012, USA\\
\tt micheli@math.ucla.edu
&\tt Peter.Michor@univie.ac.at
&\tt David\_Mumford@brown.edu\\
\end{tabular}
\end{center}
\par\vspace{.3cm}
\vspace*{.1in}
\par\noindent {\bf Acknowledgements}: MM was supported by ONR grant N00014-09-1-0256, 
PWM was supported by FWF-project 21030, 
DM was supported by 
NSF grant DMS-0704213,
and all authors were supported by 
NSF grant DMS-0456253 (Focused Research Group: 
The geometry, mechanics, and statistics of the infinite dimensional
shape manifolds).
MM would like to thank
Andrea Bertozzi of UCLA for her continuous advice and
and support.
\begin{abstract}
This paper deals with the computation of sectional curvature
for the manifolds of~$N$ landmarks (or feature points)
in~$D$ dimensions, endowed
with the Riemannian metric induced by the group action of 
diffeomorphisms. The inverse of the 
metric tensor for these manifolds (i.e.~the cometric),
when written in coordinates, is such that each of its
elements depends on at most~$2D$ of the~$ND$ coordinates.
This makes the matrices of partial derivatives of the cometric
very sparse in nature, thus suggesting solving the highly non-trivial problem of developing a formula
that expresses sectional curvature in terms
of the cometric and its first and second
partial derivatives (we call this Mario's formula). We apply such formula to the manifolds of landmarks and in particular
we fully explore the case of 
geodesics on which only two points have non-zero momenta and 
compute the sectional curvatures of 2-planes 
spanned by the tangents to such geodesics. 
The latter example gives insight to the geometry of the full manifolds
of landmarks.
\end{abstract}
\section{Introduction}
In the past few years there has been a growing interest, in diverse scientific communities, in modeling \em shape spaces \em as Riemannian manifolds. The study of shapes and their similarities is in fact central in computer vision and related fields (e.g.~for object recognition, target detection and tracking, classification of biometric data, and automated   
medical diagnostics), in that it allows one to recognize and classify objects from their representation. In particular, a \em distance \em function between shapes should express the meaning of \em similarity \em between them for the application that one has in mind. One of the most mathematically sound and tractable methods for defining a distance on a manifold is to measure infinitesimal distance by a Riemannian structure and global distance by the corresponding lengths of geodesics. 
\par
Among the several ways of endowing a shape manifold with a Riemannian structure (see, for example,~\cite{kendall:metrics, klassen, kushnarev:1, michormumford:1, sharonmumford:2, 
sundaramoorthi:1}), one of the most natural is inducing it through the action of the infinite-dimensional Lie group of diffeomorphisms of the manifold ambient to the shapes being studied. You start by putting a right-invariant metric on this diffeomorphism group, as described in~\cite{miller:1}. Then fixing a base point on the shape manifold, one gets a surjective map from the group of diffeomorphisms to the shape manifold. The right-invariance of the metric ``upstairs'' implies that we get a quotient metric on the shape manifold for which this map is a {\em submersion} (see below). This approach can be used to define a metric on very many shape spaces, such as the manifolds of curves~\cite{glaunes:3, michormumford:3}, surfaces~\cite{zhang_s}, scalar images~\cite{beg:2}, vector fields~\cite{cao:2}, diffusion tensor images~\cite{cao:1}, measures~\cite{glaunes.phd,glaunes:2}, and labeled landmarks (or ``feature points'')~\cite{glaunes:1,joshi}. The actual geometry of these Riemannian manifolds has remained almost completely unknown until very recently, when certain fundamental questions about their curvature have started being addressed~\cite{michormumford:1, michormumford:3, michormumford:4}. 
\par
Among all shape manifolds, the simplest case of the manifold of landmarks in Euclidean space plays a central role. This is defined as
$$ 
\mathcal{L}^N(\mathbb{R}^D) := \Big\{(P^1,\ldots,P^N)\,\big| \, 
P^a\in \mathbb{R}^D,\; a=1,\ldots,N \Big\}.
$$  
(typically we consider landmarks $P^a$, $a=1,\ldots,N$ that do not coincide pairwise).
It is finite-dimensional, albeit with high dimension $n=N D$, where~$N$ is the number of landmarks and $D$ is the dimension of the ambient space in which they live (e.g.~$D=2$ for the plane). Therefore its metric tensor may 
be written, in any set of coordinates, as a finite-dimensional matrix. This space is important in the study of all other shape manifolds because of a simple property of submersions: for any submersive map $f:X \rightarrow Y$, all geodesics on $Y$ lift to geodesics on $X$ and give you, in fact, all geodesics on $X$ which at one and hence all points are perpendicular to the fiber of $f$ (so called ``horizontal'' geodesics). This means that geodesics on the space of landmarks lift to geodesics on the diffeomorphism group and then project down to geodesics on all other shape manifolds associated to the same underlying ambient space $\mathbb{R}^D$. Thus geodesics of curves, surfaces, etc.~in $\mathbb{R}^D$ can be derived from geodesics of landmark points. Technically, these are the geodesics on these shape manifolds whose momentum has finite support. This efficient way of constructing geodesics on many shape manifolds has been exploited in much recent work, e.g.~\cite{allassonniere:07,durrleman:11,sommer:11}.
\par
What sort of metrics arise from submersions? Mathematically, the key point is that the inverse of the metric tensor, the inner product on the cotangent space hence called the co-metric, behaves simply in a submersion. Namely, for a submersion $f:X \rightarrow Y$, the co-metric on $Y$ is simply the restriction of the co-metric on $X$ to the pull-back 1-forms. Therefore, for the space of landmarks the cometric has a simple structure. In our case, we will see that each of its elements depends only on at most~$2D$ of the~$N\!D$ coordinates. Hence the matrices obtained by taking first and second partial derivatives of the cometric have a very \em sparse \em structure --- that is, most of their entries are zero. This suggests that for the purpose of calculating curvature (rather than following the ``classical'' path of computing first and second partial derivatives of the metric tensor itself, the Christoffel symbols, et cetera) it would be convenient to write sectional curvature in terms of the inverse of the metric tensor and its derivatives. We have solved the highly non-trivial problem of developing a formula (that we call ``Mario's formula'') precisely for this purpose: for a given pair of cotangent vectors this formula expresses the corresponding sectional curvature as a function of the cometric and its first and second partial derivatives except for one term which requires the metric (but not its derivatives). This formula is closely connected to O'Neill's formula which, for any submersion as above, connects the curvatures of $X$ and $Y$. Subtracting Mario's formula on $X$ and $Y$ gives O'Neill's as a corollary.
\par
This paper deals with the problem of computing geodesics and sectional curvature for landmark spaces, and is based on results from the thesis of the first author \cite{micheli.phd}. The paper is organized as follows. We first give a few more details about the manifold of landmarks, and describe the metric induced by the action of the Lie group of diffeomorphisms. We then give a proof for the general formula expressing sectional curvature in terms of the cometric. This formula is used in the following section to compute the sectional curvature for the manifold of labeled landmarks. In the last section, we analyze the case of geodesics on which only two points have non-zero momenta and the sectional curvatures of 2-planes made up of the tangents to such geodesics. In this case, both the geodesics and the curvature are much simpler and give insight into the geometry of the full landmark space.

\section{Riemannian Manifolds of Landmarks}
\label{Mot_Land}
In this section we briefly summarize how the shape space of landmarks 
can be given the structure of a Riemannian manifold. We refer the reader to~\cite{miller:1, Younes10} for the general framework on how to endow \em generic \em shape manifolds with a Riemannian metric via the action of Lie groups of diffeomorphisms. 

\subsection{Mathematical preliminaries}
\label{se:prel}
We will first define a \em distance \em function~$d:\mathcal{L}^N(\mathbb{R}^D) \times\mathcal{L}^N(\mathbb{R}^D) \rightarrow\mathbb{R}^+$ on landmark space
which will then turn out to be the geodesic distance with respect to a Riemannian metric. Let~$\mathcal{Q}$ be the set of differentiable landmark {\it paths\/}, that is: 
$$
\mathcal{Q} := \Big\{ q = (q^1,\ldots,q^N): [0,1]\rightarrow 
\mathcal{L}^N(\mathbb{R}^D) \, \Big| \, q^a\in 
C^1\big([0,1],\mathbb{R}^D\big), a=1,\ldots,N \,\Big\}.
$$
Following \cite[Chapters 9, 12, 13]{Younes10}, a Hilbert space~$\big(V,\langle\quad,\quad\rangle_V\big)$ of vector fields on Euclidean space (which we consider as functions $\mathbb{R}^D\rightarrow \mathbb{R}^D$) is said to be \em admissible \em if (i)~$V$~is continuously  embedded in the space of $C^1$-mappings on $\mathbb R^D\to \mathbb R^D$ which are bounded together with their derivatives, (ii)~$V$ is large enough: For any positive integer~$M$, if 
$x_1,\ldots, x_M\in \mathbb{R}^D$  and 
$\alpha_1,\ldots, \alpha_M\in \mathbb{R}^D$ are such that, for all~$u\in V$, 
$\sum_{a=1}^M\big\langle\alpha_a,u(x_a) \big\rangle_{\mathbb{R}^D}=0$,
then $\alpha_1=\ldots=\alpha_M=0$.
\par
The space $(V,\langle \quad,\quad \rangle_V)$ admits a \em reproducing kernel\em:
that is, for each $\alpha,x\in \mathbb R^D$ there exists $K_x^\alpha\in V$ with 
$\langle K_x^\alpha, f \rangle_V
=
\langle\alpha,f(x)\rangle_{\mathbb R^D}$ 
for all $f\in V$. Further, 
$
\langle K_y^\beta,K_x^\alpha \rangle_V
=
\langle \beta,K_x^\alpha(y) \rangle_{\mathbb R^D} 
= 
\langle \alpha,K_y^\beta(x) \rangle_{\mathbb R^D} 
$ 
which is a  bilinear form in 
$(\alpha,\beta)\in (\mathbb R^D)^2$, thus given by a $D\times D$
matrix 
$K(x,y)$;
the symmetry of the inner product implies that~$K(y,x)=K(x,y)^T$
(where~$^T$ indicates the transpose).
In this paper we shall assume that $K(x,y)$ is a multiple of the identity and is translation invariant: we then write $K(x,y)$ simply as $K(x-y)
\,\mathbb{I}_D$ (where~$\mathbb{I}_D$ is the~$D\times D$ identity matrix); 
the 
scalar reproducing kernel~$K:\mathbb{R}^D\rightarrow \mathbb{R}$
must be \em symmetric\em\/, and \em positive definite \em (see~\cite[\S9.1]{Younes10} for details). 
\par
There are other very natural admissible norms on vector fields~$v$ whose kernels are not multiples of the identity, e.g.~one can add a multiple of $\operatorname{div}(v)^2$ to any norm and then $K$ will intertwine different components of $v$. The most natural examples of the norms we will consider are given by inner products   
\begin{equation}
\label{eq:Lnorm}
\langle u, v\rangle_V = \langle u, v\rangle_L
:= \int_{\mathbb{R}^D} \big\langle Lu(x), v(x)\big\rangle_{\mathbb{R}^D}\ dx,
\end{equation}
where~$L$ is a self-adjoint elliptic scalar differential operator of order 
greater than $ D+2$ with constant coefficients which is applied separately to each of the scalar components of the vector field~$u=(u^1,\ldots,u^D)$. By the Sobolev embedding theorem then $V$ consists of $C^1$-functions on $\mathbb R^D$ which are bounded together with their derivatives. If $K$ is a scalar fundamental solution (or Green's function~\cite{evans}) so that $L(K)(x) = \delta(x)$, then the reproducing kernel is given by $K_x^\alpha = K(\quad -x)\,\alpha$.
A possible choice of the operator is $L=(1-A^2\Delta)^k$ (where $A\in\mathbb{R}$ is a scaling factor,  $k\in \mathbb{N}$ and~$\Delta$ is the Laplacian operator), with~$k>\frac{D}{2}+1$, in which case~\eqref{eq:Lnorm} becomes the Sobolev norm:
\begin{equation}
\label{eq:SoNorm}
\|u\|_L^2 = \int_{\mathbb{R}^D} \sum_{\ell=1}^D \sum_{m=0}^k 
\binom{k}{m} A^{2m} \sum_{|\alpha|=m} \big| D^\alpha u^\ell \big|^2 \ dx,
\end{equation}
When $L=(1-A^2\Delta)^k$  the scalar kernel~$K$ has 
the form~$K(x-y)=\gamma\big(\|x-y\|_{\mathbb{R}^D}\big)$, with:
\begin{equation}\label{Besselkernel}
\gamma(\varrho) = 
\frac{1}{2^{k+\frac{D}{2}-1}\pi^{\frac{D}{2}}\,\Gamma(k)
\,A^D}
\Big(\frac{\varrho}{A}\Big)^{k-\frac{D}{2}}\,
K_{k-\frac{D}{2}}\!\big(\frac{\varrho}{A}\big), \quad\varrho> 0 ,
\end{equation}
where $K_\nu$ (with~$\nu=k-\frac{D}{2}$)  is a modified Bessel function~\cite{abramowitz}
of order~$\nu$ (\em not \em to be confused with the symbol~$K$ we use for the kernel of~$V$).
\par
In summary, the scalar kernels that we consider in this paper will \em always \em have the properties:
\par\vspace{.1cm}
\noindent{\bf (K1)} 
$K$ is \em positive definite\em\/;
\par
\noindent{\bf (K2)} 
$K$ is \em symmetric\em\/, i.e.~$K(x)=K(-x)$, $x\in\mathbb{R}^D$.
\par\vspace{.1cm}
\noindent In addition, 
in certain sections we will  introduce 
the following simplifying
assumptions: 
\par\vspace{.1cm}
\noindent{\bf (K3)} $K$ is \em twice continuously differentiable\em\/, 
$K\in C^{2}(\mathbb{R}^D)$;
\par
\noindent{\bf (K4)} $K$ is \em rotationally invariant\em\/, 
i.e.~$K(x) = \gamma(\|x\|_{\mathbb{R}^D})$, $x\in\mathbb{R}^D$,
for some $\gamma \in C^2
\big([0,\infty)\big)$.
\par\vspace{.1cm}
\noindent
Note that if (K4) holds then $\gamma(0)\geq|\gamma(\rho)|$ for all $\rho\geq0$ by (K1) and (K2). 
Also, the bell-shaped Bessel kernels of the type~\eqref{Besselkernel}
satisfy all of the above when~$k>\frac{D}{2}+1$.
\par
Now fix any admissible Hilbert space of vector fields. The space~$L^p([0,1],V)$
is the set of functions~$v:[0,1]\rightarrow V$ such that: 
$$\|v\|_{L^p([0,1],V)}:=
\Big( \int_0^1 \|v(t,\;\;\;) \|^p_V\ dt \Big)^{\frac{1}{p}} <\infty. $$ 
The space $L^2([0,1],V)$ is a subset of $L^1([0,1],V)$
and is in fact a Hilbert space with inner product
$\langle  u,v\rangle  _{L^2([0,1],V)} := \int_0^1 \langle  u,v\rangle_V\ dt $.
It is well known from the theory of ordinary differential 
equations~\cite{chicone} 
that for any $v\in L^1([0,1],V)$, the $D$-dimensional non-autonomous dynamical system $\dot{z} = v_t(z)$, with initial condition $z(t_0) = x$, has a unique solution of the type~$z(t)=\psi(t,t_0,x)$. 
Let~$\varphi_{st}^v(x):=\psi(t,s,x)$; fixing~$t=1$ and $s=0$ we get
$\varphi^v:=\varphi_{01}^v$, which is the \em diffeomorphism generated by\em~$v$. For an admissible Hilbert space we will call the set
$$
\mathcal{G}_V :=  \big\{ \varphi^v: v\in L^1\big([0,1],V\big) \big\}
$$
the \em group of diffeomorphisms \em generated by~$V$; by \cite[Chapter 12]{Younes10} it is a metric space and a topological group. But, in the language of manifolds, $\mathcal{G}_V$ is {\bf not} an infinite-dimensional Lie group~\cite{michor:cs}. $V$ is not a Lie algebra, but is the completion of the Lie algebra of $C^\infty$-vector fields with compact support with respect to $\|\quad\|_V$.  

\subsection{Definition of the distance function}
For velocity vector fields~$v\in L^2([0,1],V)$ and landmark trajectories
$q\in\mathcal{Q}$ define the energy
\begin{align}
E_\lambda[v,q]\equiv E[v,q] := \int_0^1 \Big(\big\| v(t,\quad)\big\|_V^2
+ \lambda \sum_{a=1}^{N} \Big\| \frac{dq^a}{dt}(t) -v\big(t,q^a(t) \big) 
\Big\|_{\mathbb{R}^D}^2\Big) \ dt ,
\label{Evq}
\end{align}
where~$\lambda\in (0,\infty]$ is a fixed \em smoothing parameter \em (soon to be described). We claim that  a distance function~$d$ on~$\mathcal{L}^N(\mathbb{R}^D)$  between two  landmark sets (or shapes)
$I=(x^1,x^2,\ldots,x^N)$ and  $I'=(y^1,y^2,\ldots,y^N)$ can be defined  as 
\begin{equation}
\label{lan_dist}
d(I,I') := \inf_{v,q}\Big\{\sqrt{E[v,q]}: v\in L^2\big([0,1],V\big),\,  
	q\in\mathcal{Q}\mbox{ with } q(0)=I, \, q(1)=I'\Big\};
\end{equation}
in the next subsection we will argue that the above function is in fact a \em geodesic distance \em with respect to a Riemannian metric. We treat the minimization of~\eqref{Evq} as our starting point; 
it is the ``energy of a metamorphosis'' as formulated
in~\cite[Chapter~13]{Younes10}.
%
\par
The above infimum is computed over all differentiable landmark paths~$q\in \mathcal{Q}$ that satisfy the boundary conditions
($q^a(0)=x^a$ and $q^a(1)=y^a$, $a=1,\ldots,N$), and vector fields
$v\in L^2([0,1],V)$. The resulting landmark trajectories
$\{q^a(t),t\in[0,1]\}_{a=1,\ldots,N}$ follow the minimizing velocity field
more or less exactly, depending on the value of the smoothing
parameter~$\lambda\in(0,\infty]$; it is a weight between the first term, 
that measures the smoothness of the vector field that generates the diffeomorphism,  and the second term, that measures how closely the landmark trajectories actually follow the vector field. 
\par
The \em exact matching problem \em is the following: given two sets of 
landmarks 
$I=(x^1,x^2,\ldots,x^N)$ and  $I'=(y^1,y^2,\ldots,y^N)$~with 
$x^a\not=x^b$ and $y^a\not=y^b$ for any~$a\not=b$, minimize the energy
$$
E_\infty[v] := \int_0^1 \|v(t,\quad)\|_V^2 \, dt
$$
among all~$v\in L^2([0,1],V)$ such that~$\varphi^v(x^a)=y^a$, $a=1,\ldots,N$. In this case
the landmark trajectories are defined as the solutions to the ordinary differential 
equations~$\dot{q}^a=v(t,q^a)$, $a=1,\ldots,N$. Note that this is equivalent to solving~(\ref{Evq})
for~$\lambda=\infty$, since such equations
are obtained by setting the integrands of the second term in the right-hand 
side of~\eqref{Evq} equal to zero. When~$\lambda<\infty$ in~(\ref{Evq}) we have \em regularized \em matching,
i.e.~the landmark trajectories ``almost'' satisfy such set of ordinary differential equations; this allows for the time varying vector field to be smoother. For this reason the second term in
the right-hand side of~\eqref{Evq} is often referred to as {\it smoothing term\/}; by allowing smoother vector fields the distance~$d$ is made tolerant to small diffeomorphisms and therefore more robust to object variations due to noise in the data.
%
%
%
%
%

\subsection{Minimizing velocity fields and Riemannian formulation}
By manipulating expression~\eqref{Evq} we will now show that it is equivalent to the energy of a path $q\in\mathcal{Q}$ with respect to a Riemannian metric. 
\begin{notation}
Consider a landmark~$q=(q^1,\ldots,q^N)$ in~$\mathcal{L}^N(\mathbb{R}^D)$.
The~$D$ scalar components in Euclidean coordinates of the~$N$ landmark trajectories $q^a=(q^{a1},\ldots,q^{aD})$, $a=1,\ldots,N$ can be ordered either into an~$N\times D$ matrix or in a tall concatenated column vector. We shall always use indices $a,b,c,\ldots \in \{1,\dots,N\}$ as {\it landmark indices}, and $i,j,k,\ldots \in\{1,\dots, D\}$ as {\it space coordinates} in $\mathbb R^D$. We will associate to each of the~$N$ landmarks~$q^a\in\mathbb{R}^{1\times D}$ a \em momentum\em~$p_a\in\mathbb{R}^{1\times D}$ (defined in the next proposition) which we will write, in coordinates, as~$p_a=(p_{a1},\ldots,p_{aD})$, for each~$a=1,\ldots,N$. The components of momenta can also be ordered into an~$N\times D$ matrix or in a long row vector. We chose \em superscript \em indices for landmark coordinates and \em subscript \em indices for momenta.
\par
For a given set of landmarks~$(q^1,\ldots,q^N)\in\mathcal{L}^N(\mathbb{R}^D)$
we will define the symmetric~$N\times N$ matrix
$\kermat(q):= \bigl(K(q^a-q^b)\bigr)_{a,b=1,\dots,N}$.
The matrix~$\kermat(q)$ is positive definite by property (K1) of the kernel. 
\end{notation}
\begin{proposition}
\label{pr:vstar}
For a \em fixed \em landmark path~$\overline{q}=\big\{\overline{q}^a:[0,1]\rightarrow 
\mathbb{R}^D\big\}_{a=1}^N\in\mathcal{Q}$
there exists a unique minimizer with respect to~$v\in L^2([0,1],V)$ of the energy~$E[v,\overline{q}]$, namely:
\begin{equation}
\label{vstar1}
v^\ast(t,x) := \sum_{a=1}^N p_a(t) \, K\big(x-\overline{q}^a(t)\big), \qquad t\in[0,1],  \; x\in\mathbb{R}^D,
\end{equation}
where the components of the momenta are given by:
\begin{equation}
\label{Expr_Mom}
p_{ai}(t) = \sum_{b=1}^N
\Big( \kermat\big(\overline{q}(t)\big) + \frac{\mathbb I_N}{\lambda} \Big)^{-1}_{ab}  \cdot 
	\frac{d}{dt}\overline{q}^{bi}(t), \qquad t\in[0,1], 
\end{equation}
$a=1,\ldots,N$, $i=1,\ldots,D$ (here~$\mathbb I_N$ indicates the~$N\times N$ identity matrix).
\end{proposition}
\begin{remark}
What the above proposition essentially says is that the vector field of minimum energy that transports the~$N$ landmarks along fixed trajectories is, at any point of time, the linear combination of~$N$ lumps of velocity, each centered at a landmark point. The directions and amplitudes of the summands are determined precisely by the momenta.
\end{remark}
\begin{proof}[Proof of Proposition~\ref{pr:vstar}]
Using property (ii) of the admissible Hilbert space $V$, 
\cite[Lemma 9.5]{Younes10} shows that
for given $q=(q^1,\dots,q^N)\in \mathcal L^N(\mathbb R^D)$ we have 
the orthogonal decomposition
\begin{equation}\label{decompositionV}
V = \big\{v\in V: v(q^a)=0, a=1,\dots N\big\} 
\oplus \big\{v=\textstyle\sum_{a=1}^N \alpha_a K(\quad-q^a): \alpha_a\in \mathbb R^D\big\}. 
\end{equation}
Thus the minimizer must have the form
\begin{equation}
\label{v_generic}
v(t,x) = \sum_{a=1}^N \alpha_a(t) \, K\big(x-\overline{q}^a(t)\big), \qquad t\in[0,1],  \; x\in\mathbb{R}^D,
\end{equation}
for some coefficients
$\alpha_a\in C([0,1],\mathbb{R}^{D})$, $a=1,\ldots,N$,
to be computed. 
For velocities of the type~\eqref{v_generic} the
energy~\eqref{Evq} can be rewritten as
\begin{equation}
\label{Evq_aux}
E[v,\overline{q}] = \int_0^1 \sum_{i=1}^D\sum_{a,b=1}^N \Big\{ \alpha_{ai} K(\overline q^a-\overline q^b) \alpha_{bi} 
+ \lambda \,\big| \alpha_{ai}K(\overline q^a-\overline q^b)  - \dot{\overline q}^{bi} \big|^2 \Big\} \ dt.
\end{equation}
Setting the first variation of (\ref{Evq_aux}) with respect to 
coefficients $\alpha_{ai}$ to zero
yields the momenta~\eqref{Expr_Mom}. 
\end{proof}
It is convenient, at this point, to introduce the  $ND\times ND$, block-diagonal matrix
\begin{equation}
\label{def:met_t}
g(q):=
\begin{pmatrix}
\big(\kermat(q)+\frac{\idmatrix_N}{\lambda}\big)^{-1} & 0 & \cdots & 0 \\
0 & \big(\kermat(q)+\frac{\idmatrix_N}{\lambda}\big)^{-1} & \cdots & 0 \\
\vdots & \vdots & \ddots & \vdots \\
0 & 0 & \cdots & \big(\kermat(q)+\frac{\idmatrix_N}{\lambda}\big)^{-1}
\end{pmatrix},
\end{equation}
where the
$N\times N$ block~$\big(\kermat(q)+\frac{\mathbb{I}_N}{\lambda}
\big)^{-1}$ is repeated~$D$ times; the choice of symbol~$g$ is justified by the fact that~\eqref{def:met_t} is, as we shall see soon, precisely the Riemannian metric tensor with which we are endowing the manifold of landmarks,
written in coordinates.

Thus for a fixed path~$\overline{q}\in\mathcal{Q}$ the minimizer of~$E[v,\overline{q}]$ with respect to $v\in L^2([0,1],V)$ is given by~\eqref{vstar1};
since it depends on~$\overline{q}$ we will write it, with an abuse of notation, as~$\vastb$. We can define
\begin{equation}
\label{Etil}
\widetilde{E}[\overline q]  := E[ v^\ast(\overline q),\overline q ],
\end{equation}
which depends \em only \em on  the arbitrary path~$\overline q\in\mathcal{Q}$. The energy~\eqref{Etil} is ``equivalent'' to the energy~$E[v,q]$, in that:
\par\vspace{.1cm}
\noindent{\bf (a)} if $(\hat{v},\hat{q})$ minimizes~$E[v,q]$ then $\hat{q}$
minimizes~$\widetilde{E}[q]$, 
and~$E[\hat{v},\hat{q}]=\widetilde{E}[\hat{q}]$;
\par\vspace{.1cm}
\noindent{\bf (b)} if $\hat{q}$ minimizes~$\widetilde{E}[q]$
then $(\vasth,\hat{q})$
minimizes 
$E[v,q]$, and $E[\vasth,\hat{q}\,]=\widetilde{E}[\hat{q}\,]$.
\begin{proposition}
\label{th:Riem_E}
For an arbitrary landmark trajectory~$q\in \mathcal{Q}$
the energy~$\widetilde{E}[q]$ is given by:
\begin{equation}
\label{energy2}
\widetilde{E}[q] = \int_0^1 \dot{q}(t)^T g\big(q(t)\big) \, \dot{q}(t)\, dt
=
\int_0^1 \sum_{a,b=1}^N
\sum_{i=1}^D \dot{q}^{ai}(t)\,\dot{q}^{bi}(t)
\Big(\kermat\big(q\big(t))+\frac{\mathbb I_N}{\lambda}\Big)^{-1}_{ab}
\,dt
\end{equation}
\end{proposition}
In the above equation~$\dot{q}(t)$ is intended as an $ND$-dimensional column vector
obtained by stacking the column vectors $(\dot{q}^{1i}(t),\ldots,
\dot{q}^{Ni}(t))^T$, $i=1,\ldots, D$ 
(again, the superscript~$^T$ indicates the transpose of a vector).
\begin{proof}
Following definition \eqref{Etil}, formulae~\eqref{Expr_Mom}
for the momenta are inserted into the modified
expression~\eqref{Evq_aux} for energy~$E[v,q]$. 
Simple matrix manipulations finally yield the right-hand side
of~\eqref{energy2}.
\end{proof}
\begin{remarks}
Expression~\eqref{energy2} has exactly the form of the energy of a path~$q$
with respect to Riemannian metric tensor~\eqref{def:met_t}. Whence given two
landmark configurations~$I$ and~$I'$
in~$\mathcal{L}^N(\mathbb{R}^D)$ we have that if~$\hat{q}$ minimizes~\eqref{energy2}
among all paths in~$q\in\mathcal{Q}$ such that~$q(0)=I$
and~$q(1)=I'$ then~$(\widetilde{E}[\hat{q}])^{1/2}$ is the \em geodesic distance \em 
between~$I$ and~$I'$. By point~{\bf(b)} above we also have that $(\vasth,\hat{q})$
is a minimum of energy~$E[v,q]$, so~$d(I,I')$ defined in~\eqref{lan_dist}
coincides with~$(\widetilde{E}[\hat{q}])^{1/2}$ and \em is \em
the geodesic distance between~$I$ and~$I'$ with respect to the metric tensor~$g$.
\par
The Lagrangian function that 
corresponds to the energy~\eqref{energy2}
is:
\begin{equation}
\label{eq:lagr}
\mathcal{L}(q,\dot{q})=\frac{1}{2}\,\dot{q}^Tg(q)\dot{q}
=
\frac{1}{2}
\sum_{a,b=1}^N
\sum_{i=1}^D
\dot{q}^{ai}\dot{q}^{bi} \Big(\kermat(q)+\frac{\mathbb I_N}{\lambda}\Big)^{-1}_{ab}.
\end{equation}
In Hamiltonian mechanics~\cite[p.~60]{arnold:1} 
the ``momenta'' are defined as
$
p_{ai}
=
\partial\mathcal{L}/\partial \dot{q}^{ai}
$,
or, in vector notation, 
$p_{(i)}=\partial\mathcal{L}/\partial \dot{q}^{(i)}$
(for $i=1,\ldots,D
$).
Applying such definition to~\eqref{eq:lagr}
yields
precisely
equations~\eqref{Expr_Mom}
of Proposition~\ref{pr:vstar}. Whence the
use of the term momenta is justified.
\par
Note that for small values of the parameter~$\lambda$
the metric tensor~$g$, written in coordinates, gets close
(up to a multiplicative constant) to the~$ND\times ND$ 
identity matrix; in other words, for~$\lambda\rightarrow 0$,
$g$ converges to a Euclidean metric and the geodesic curves become
straight lines. On the other hand, for~$\lambda\rightarrow\infty$
(exact matching) the metric converges to 
$
[\mathrm{diag}\{\kermat(q),\ldots,\kermat(q)\}]^{-1}$ (block~$\kermat(q)$ is repeated~$D$ times).
In general, the block-diagonal form of the metric tensor~$g$ given by \eqref{def:met_t}
follows from the fact that the operator~$L$ in~\eqref{eq:SoNorm}
is  applied separately to each of the components of the velocity field;
however the dynamics of the~$D$ dimensions of~$q$ are \em not \em decoupled since
\em all \em $N\!D$ components of~$q$ appear in each diagonal block of~$g$. 

In the case of exact matching  
landmarks ``never collide'' (their trajectories are precisely defined by
diffeomorphisms of~$\mathbb{R}^D$): it takes
an infinite amount of energy to make any two landmarks coincide.
So  under the condition~$\lambda=\infty$  
the manifold of landmarks  can actually be taken as the set: 
\begin{equation}
\label{lambdainf}
\mathcal{L}^N(\mathbb{R}^D) = \Big\{(P^1,\ldots,P^N)\,\big| \, 
P^a\in \mathbb{R}^D, P^a\not =P^ b \mbox{ if }a\not=b \Big\}.
\end{equation}
\par
Figure~\ref{2ptgeod} shows the qualitative behavior of
geodesics in~$\mathcal{L}^2(\mathbb{R}^2)$, with $\lambda=\infty$. 
In the case illustrated
on the left-hand side both landmarks travel in the same direction
(from left to right, as indicated by the arrows):
the two arcs of the
geodesic ``attract'' each other, or in other words 
the two landmarks tend to 
``carpool'' by using a velocity field with the smallest
possible support so to minimize the~$L^2$ part
(i.e.~the first term)
of the Sobolev norm~\eqref{eq:SoNorm} of the velocity field.
 On the other hand when the two landmarks travel in \em opposite \em
directions (as illustrated on the right-hand side of Figure~\ref{2ptgeod}) 
they try to avoid each other so that the higher order terms
of the Sobolev norm are kept small; 
we shall return on the issue of obstacle avoidance at the end of this paper.
A typical
geodesic in $\mathcal{L}^4(\mathbb{R}^2)$
(again with~$\lambda=\infty$) is shown in Figure~\ref{4ptgeod}.
\end{remarks}
\begin{figure}[t]
\center{
\includegraphics[height=7.2cm]{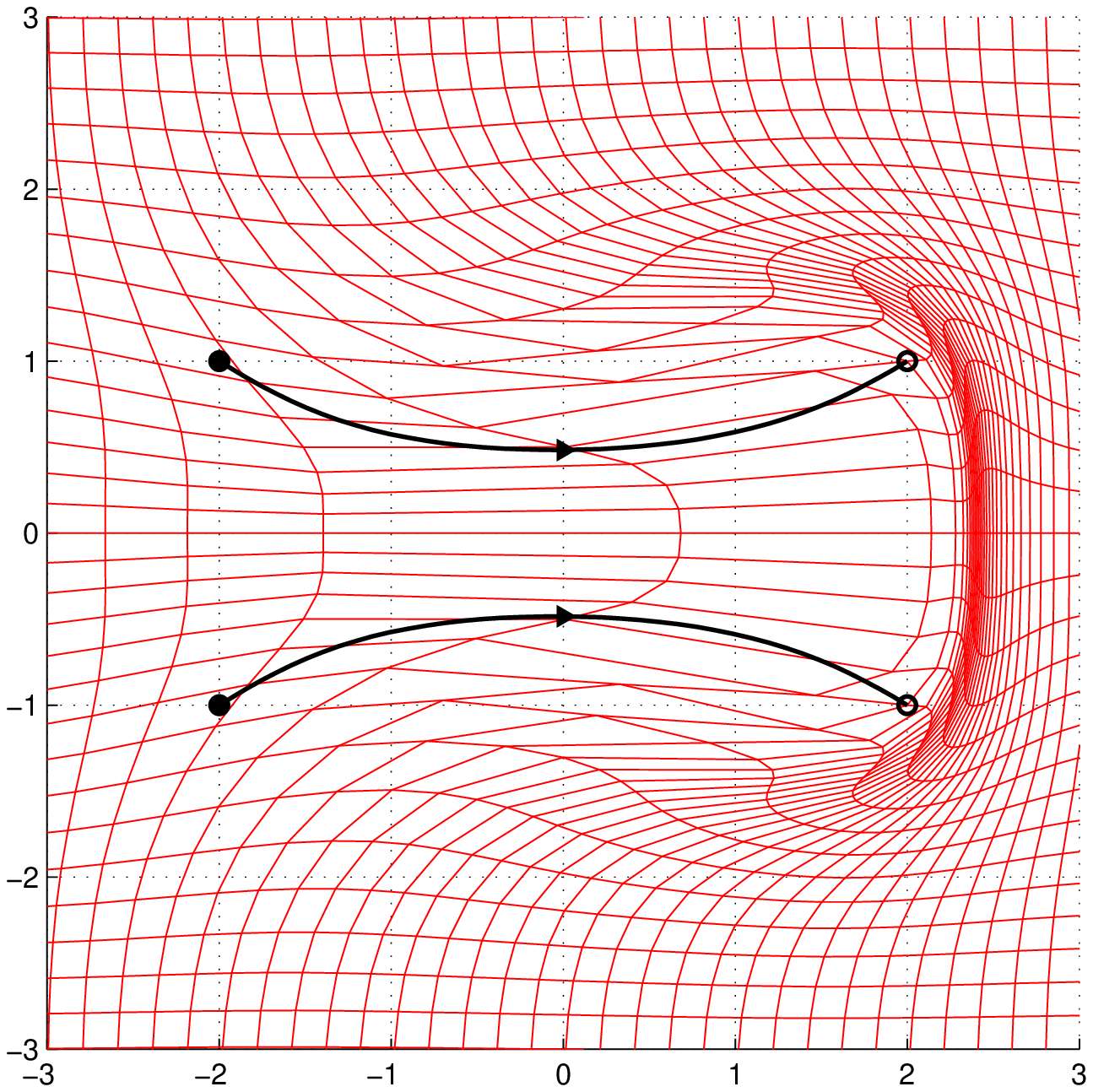}
\includegraphics[height=7.2cm]{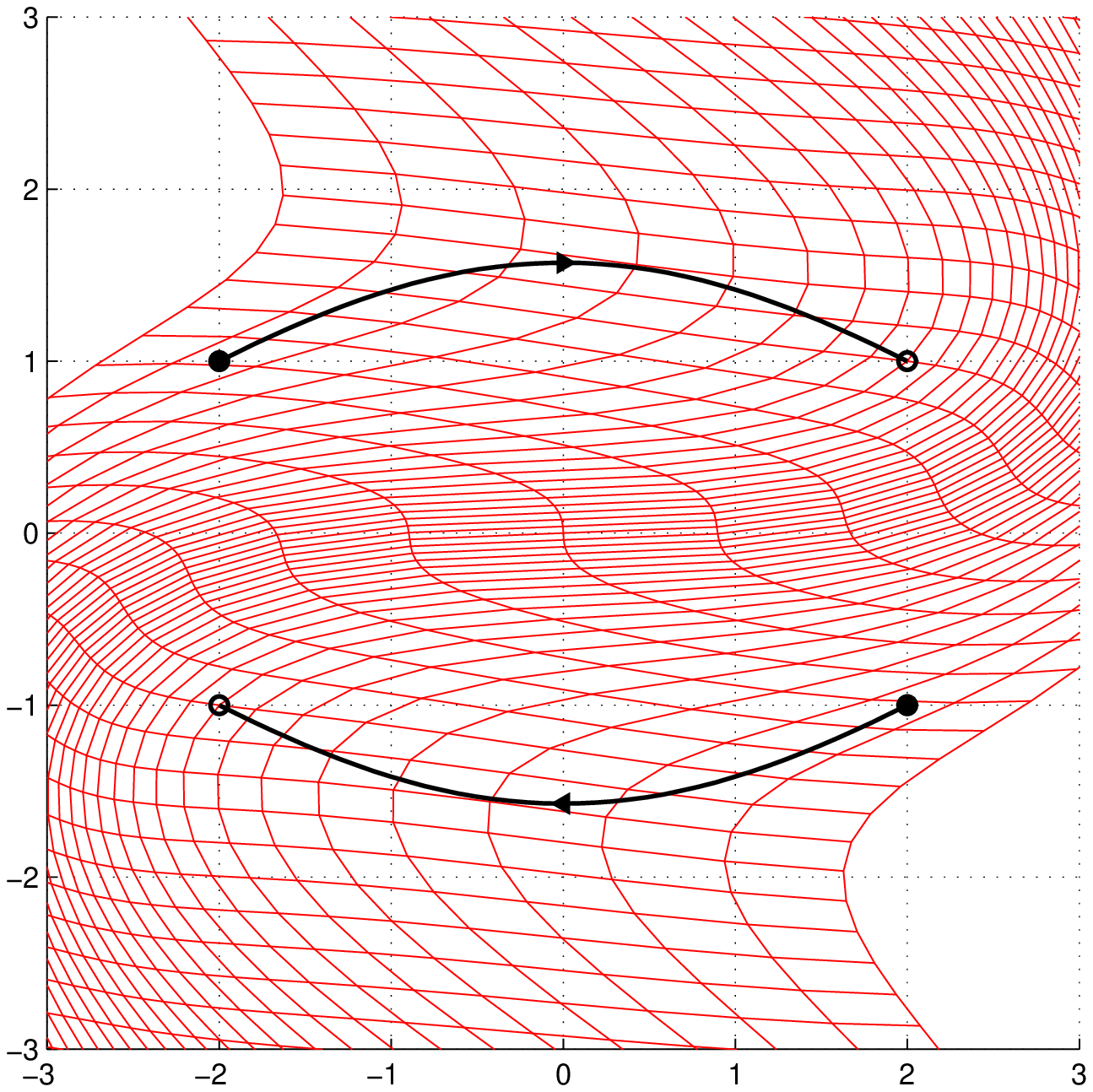}
}
\caption{
Two trajectories in~$\mathcal{L}^2(\mathbb{R}^2)$. 
Bullets~$(\bullet)$ and circles~$(\circ)$
are the initial and final sets of landmarks, respectively. 
The grids represents the two corresponding
diffeomorphisms~$\varphi_{01}^v$.}
\label{2ptgeod}
\end{figure} 
\begin{figure}[t]
\center{
\includegraphics[height=7.2cm]{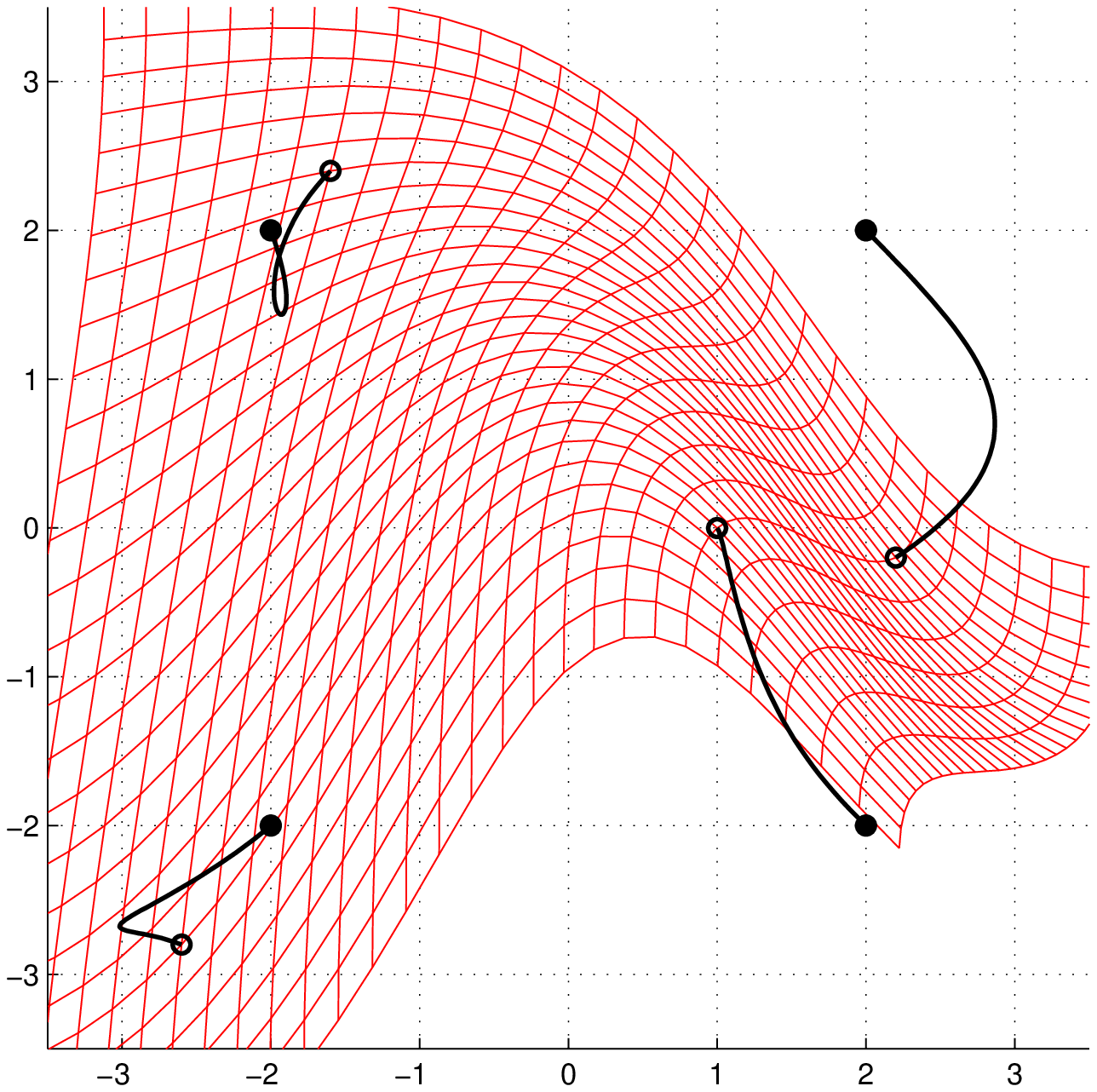}
}
\caption{A typical geodesic trajectory in $\mathcal{L}^4(\mathbb{R}^2)$. 
Bullets~$(\bullet)$ and circles~$(\circ)$
are the initial and final sets of landmarks, respectively. 
The grid represents the corresponding
diffeomorphism~$\varphi_{01}^v$.}
\label{4ptgeod}
\end{figure} 
\begin{conclusion}
We have shown that distance~$d(I,I')$, $I,I'\in \mathcal{L}^N(\mathbb{R}^D)$ defined in~\eqref{lan_dist}
is in fact the geodesic distance with respect to a
Riemannian metric. In coordinates, the corresponding 
Riemannian metric tensor is given by~\eqref{def:met_t}, which is such
that each element of its \em inverse \em (the cometric) 
depends on at most~$2D$ of the~$N\!D$ coordinates. Whence the first
and second partial derivatives of the cometric have a very sparse
structure. This gives us motivation for deriving a general formula
for computing sectional curvature in terms of the cometric and its derivatives \em
in lieu \em of the metric and its derivatives, which will be done in the next section.
\end{conclusion}

\section{Sectional Curvature in terms of the Cometric}
\label{sec:MaFo}
\subsection{Generalities and notation on sectional curvature}
\label{dual_R}
Let~$\mathcal{M}$ be an $n$-dimensional Riemannian manifold. If we consider a local chart~$(U,\varphi)$ on the manifold with coordinates $(x^1,\ldots,x^n)$, we have the induced 1-forms $dx^1,\dots,dx^n$ and coordinate vector fields 
$\{\partial_1:=\frac{\partial}{\partial x^1},\ldots,\partial_n=\frac{\partial}{\partial x^n}\}$.
The metric tensor
$g:T\mathcal{M}\times_{\mathcal M} T\mathcal{M}\rightarrow \mathbb{R}$
can be represented as 
$g|_U=g(\partial_i,\partial_j)\,dx^i\otimes dx^j =: g_{ij}\,dx^i\otimes dx^j$
(here, as in the rest of the current section, we are using Einstein's summation 
convention).
For each $p\in \mathcal{M}$ 
we get a positive definite matrix with elements~$g_{ij}(p)=g_p(\partial_i,\partial_j)$.
With an abuse of notation we will write
$g_{ij}(x)$ instead of~$(g_{ij}\circ\varphi^{-1})(x)$, $x\in\varphi(U)\subset \mathbb{R}^n$.
\begin{notation}
We shall denote the partial derivatives of
the elements of the metric tensor~$g$ as
$g_{ij,k}(x):=\,\frac{\partial}{\partial x^k}g_{ij}(x)= \partial_{k}g_{ij}$ 
and 
$g_{ij,k\ell}(x):=\frac{\partial^2}{\partial x^\ell\partial x^k}\,g_{ij}(x) = 
\partial_\ell\partial_k g_{ij}$.
Also, we will indicate the \em cometric \em as~$g^{-1}|_U = g^{ij}\pari\otimes\parj$
(so that~$g^{ij}g_{jk}=\delta^i_k$)
and their partial derivatives with
$\gul{ij}{k}\!\!(x):=\,\frac{\partial}{\partial x^k}g^{ij}(x)$ 
and 
$\gul{ij}{k\ell}\!\!(x):=\frac{\partial^2}{\partial x^\ell\partial x^k}\,g^{ij}(x)$.

For a tangent vectors $X=X^i\pari$ we consider the 1-form $X^\flat := X^i g_{ij}dx^j =: X_jdx^j$ (indices 
lowered), and for a 1-form  $\alpha = \alpha_idx^i$ we have the tangent vector $\alpha^\sharp := 
\alpha_i g^{ij}\parj$ (indices lifted).

Indicating with~$\mathcal{X}(\mathcal{M})$
the space of smooth vector fields on the manifold~$\mathcal{M}$,
let $\nabla:\mathcal{X}(\mathcal{M})
\times\mathcal{X}(\mathcal{M})
\rightarrow\mathcal{X}(\mathcal{M})$ 
be the Levi-Civita connection~\cite{jost,lee:2} of the Riemannian manifold. The Christoffel symbols
$\Gamma_{ij}^k$ are defined by
$\nabla_{\partial_i}\partial_j=\Gamma_{ij}^k\partial_k$, 
and it is well known that they have the form:
$ \Gamma_{ij}^k=\frac{1}{2}g^{k\ell}
(g_{i\ell,j}+g_{j\ell,i}-g_{ij,\ell})$.
The \em Riemannian curvature endomorphism \em is the map
$R:\mathcal{X}(\mathcal{M})\times
\mathcal{X}(\mathcal{M})\times
\mathcal{X}(\mathcal{M})
\rightarrow \mathcal{X}(\mathcal{M})$
given  by
$ R(X,Y)Z=\nabla_X\nabla_YZ-\nabla_Y\nabla_XZ-\nabla_{[X,Y]}Z.$
In local coordinates
$R(\partial_i,\partial_j)\partial_k = R_{ijk}^\ell \partial_\ell$,
and
$R_{ijkm}:= 
\langle R(\partial_i,\partial_j)\partial_k, \partial_m \rangle_g 
=g_{m\ell}R_{ijk}^\ell.$
The \em Riemannian curvature tensor \em
acts on vector fields as follows:
\begin{equation}
\label{Raction}
R(X,Y,Z,W) := \langle R(X,Y)Z,W \rangle_g
\end{equation}
and in coordinates it is written as
$
R=R_{ijkm} dx^i \otimes dx^j \otimes dx^k \otimes dx^m
$.
The Riemannian curvature tensor has a number of symmetries:
(i) $R_{ijk\ell}  = -R_{jik\ell}$;
(ii) $R_{ijk\ell}  = -R_{ij\ell k}$;
(iii)~$R_{ijk\ell}  = R_{k\ell ij}$; 
and~(iv)~$R_{ijk\ell}+R_{jki\ell}+R_{kij\ell}= 0$ 
(first Bianchi identity).
With such conventions, the \em sectional curvature \em associated to a pair of non-parallel tangent vectors~$X$ and~$Y$
is computed by:
\begin{equation}
\label{seccurv}
\mathcal{K}(X,Y) \;\,= \;\, \frac{R(X,Y,Y,X)}{\|X\|_g^2\|Y\|_g^2-\langle X,Y\rangle_g^2} \;\,=\;\, 
\frac{R_{ijkm}X^iY^jY^kX^m}{\|X\|_g^2\|Y\|_g^2-\langle X,Y\rangle_g^2}\,.
\end{equation}
\par
In order to express the numerator of sectional curvature~\eqref{seccurv} in terms of the elements of the cometric and its derivatives (i.e.~$g^{ij}$, $\gul{ij}{k}\!$, and 
$\gul{ij}{k\ell}\!\!$) 
we consider the covariant expression of the Riemannian
curvature tensor: 
\begin{equation}
\label{def_drct}
R^{ursv}:= R_{ijkm}\,g^{iu}g^{jr}g^{ks}g^{mv},
\end{equation}
which we call the {\it dual Riemannian curvature tensor\/}.
Similarly we consider the covariant or \em dual Christoffel symbols \em
\begin{equation}
\label{lemma_dcs}
\Gamma_{u}^{rs}:=g^{ir}g^{js}g_{ku}\Gamma^k_{ij}\,,
\end{equation}
which are symmetric in the indices~$r$ and~$s$.
\par
To achieve notational compactness we will use the following symbols:
\begin{equation}
\label{u_deriv_2}
g^{ij,k}:=\gul{ij}{\xi}g^{\xi k}
\;\;\mbox{ and }\;\;
g^{ij,k\ell}:= \gul{ij}{\xi\eta}g^{\xi k}g^{\eta \ell};
\end{equation}
Using that $g=Q^{-1}$ implies
$\partial_k g= -Q^{-1}\cdot\,\partial_k Q\cdot Q^{-1}$ one immediately sees that
$$
\Gamma_{u}^{rs} = 
	- \frac{1}{2}\, g_{u\varphi} \big( g^{s\varphi,r} + g^{r\varphi,s} - g^{rs,\varphi} \big).
$$
\end{notation}
\begin{proposition} 
\label{pr:1st_te}
The following expression holds for the Riemannian curvature tensor:
\begin{equation}
\label{TensExpr}
2R_{ijkm}  =  g_{ik,jm} + g_{jm,ik} - g_{jk,im} - g_{im,jk} 
	+ 2 \Gamma^\alpha_{ik} \Gamma^\beta_{jm} g_{\alpha\beta} 
	- 2 \Gamma^\alpha_{jk} \Gamma^\beta_{im} g_{\alpha\beta}.
\end{equation}
\end{proposition}
\noindent{For} a proof see~\cite[\S 24.9]{michor:dgbook}.

\subsection{Mario's formula}
\begin{proposition}
\label{prop_drct}
The following expression holds for 
the dual Riemannian curvature tensor:
\begin{align}
\nonumber
2R^{ursv} = &
-g^{us,rv} -g^{rv,us} +g^{rs,uv} +g^{uv,rs} 
+2\Gamma_\rho^{rv}\Gamma_\sigma^{us}g^{\rho\sigma}
-2\Gamma_\rho^{rs}\Gamma_\sigma^{uv}g^{\rho\sigma}
\\
\label{drct}
& 
+\;\, g^{r\lambda,u} g_{\lambda\mu}\, g^{\mu v,s} -
g^{r\lambda,u} g_{\lambda\mu}\, g^{\mu s,v} +
g^{u\lambda,r} g_{\lambda\mu}\, g^{\mu s,v} -
g^{u\lambda,r} g_{\lambda\mu}\, g^{\mu v,s}
\\
& 
+\;\, g^{r\lambda,s} g_{\lambda\mu}\, g^{\mu v,u} +
g^{u\lambda,v} g_{\lambda\mu}\, g^{\mu s,r} -
g^{r\lambda,v} g_{\lambda\mu}\, g^{\mu s,u} -
g^{u\lambda,s} g_{\lambda\mu}\, g^{\mu v,r}. 
\nonumber
\end{align}
\end{proposition}

\begin{proof}
We will manipulate~\eqref{TensExpr} and write it in the form 
$R_{ijkm} = g_{iu} g_{jr} g_{ks} g_{mv} R^{ursv}$
by factoring~$g_{iu} g_{jr} g_{ks} g_{mv} $ 
out of each term; what will be left will be precisely the expression for~$R^{ursv}$.
\par
The terms in~\eqref{TensExpr} involving Christoffel symbols are,
by (\ref{lemma_dcs}):
\begin{align}
\label{aux_A}
\Gamma^\alpha_{ik} \Gamma^\beta_{jm} g_{\alpha\beta} =
g_{iu}g_{ks}g^{\alpha\sigma} \, \Gamma^{us}_{\sigma} \,
g_{jr}g_{mv}g^{\rho\beta} \, \Gamma^{rv}_{\rho} \, g_{\alpha\beta}
&=
g_{iu}g_{jr} g_{ks}g_{mv} 
\big( \Gamma^{rv}_{\rho} \Gamma^{us}_{\sigma} g^{\rho\sigma} \big),
\\
\mbox{and similarly: }
\;\;\;\;
\;\;\;\;
\;\;\;\;
\;\;\;\;
\Gamma^\alpha_{jk} \Gamma^\beta_{im} g_{\alpha\beta}
&=
g_{iu}g_{jr} g_{ks}g_{mv} 
\big(\Gamma^{rs}_{\rho} \Gamma^{uv}_{\sigma} g^{\rho\sigma} \big).
\end{align}
As we noted before, if~$g=\invg^{-1}$
then~$\partial_j g= -Q^{-1}\cdot\,\partial_j Q\cdot Q^{-1}$
and similarly it is the case that
$ \partial_m \partial_j g =
\invg^{-1} \cdot \big( \partial_m \invg \cdot \invg^{-1} \cdot \partial_j \invg +
\partial_j \invg \cdot \invg^{-1} \cdot \partial_m \invg -\partial_m\partial_j \invg \big)
\cdot \invg^{-1}, $
i.e., in index notation,
\begin{align}
\nonumber
g_{ik,jm} & = g_{iu} \big( \gul{u\lambda}{m} g_{\lambda\mu} \, \gul{\mu s}{j} + \gul{u\lambda}{j} g_{\lambda\mu}
\, \gul{\mu s}{m} -
\gul{us}{jm}\!\! \big)
g_{sk}\,
\\
\nonumber
&= g_{iu} g_{ks} \delta^\xi_j \delta^\eta_m \big( \gul{u\lambda}{\eta} g_{\lambda\mu} \,
\gul{\mu s}{\xi} + \gul{u\lambda}{\xi} g_{\lambda\mu} \, \gul{\mu s}{\eta} - \gul{us}{\xi\eta}\!\! \big)
\\
& \nonumber =
g_{iu} g_{ks} g_{jr} g_{mv} \big[ g^{r\xi} g^{v\eta} \big( \gul{u\lambda}{\eta} g_{\lambda\mu} \,
\gul{\mu s}{\xi} + \gul{u\lambda}{\xi} g_{\lambda\mu} \, \gul{\mu s}{\eta} - \gul{us}{\xi\eta}\!\! \big)
\big]
\\ 
&= g_{iu} g_{jr} g_{ks} g_{mv} \big(
g^{u\lambda,v} g_{\lambda\mu} \, g^{\mu s,r} + g^{u\lambda,r} g_{\lambda\mu} \, g^{\mu s,v} - g^{us,rv} \big),
\end{align}
where we have used definitions~\eqref{u_deriv_2}. Similarly, we can achieve the factorizations:
\begin{align}
g_{jm,ik} & =
g_{iu} g_{jr} g_{ks} g_{mv} \, \big(
g^{r\lambda,u} g_{\lambda\mu}\, g^{\mu v,s} +
g^{r\lambda,s} g_{\lambda\mu}\, g^{\mu v,u} -
g^{rv,us} \big),
\\
-g_{jk,im} & =
g_{iu} g_{jr} g_{ks} g_{mv} \, \big(
- g^{r\lambda,u} g_{\lambda\mu}\, g^{\mu s,v} -
g^{r\lambda,v} g_{\lambda\mu}\, g^{\mu s,u} +
g^{rs,uv} \big),
\\
-g_{im,jk} & =
g_{iu} g_{jr} g_{ks} g_{mv} \, \big(
- g^{u\lambda,r} g_{\lambda\mu}\, g^{\mu v,s} -
g^{u\lambda,s} g_{\lambda\mu}\, g^{\mu v,r} +
g^{uv,rs} \big).
\label{aux_F}
\end{align}
Inserting~\eqref{aux_A}$\div$\eqref{aux_F} into~\eqref{TensExpr} we can write
$ R_{ijkm} = g_{iu} g_{jr} g_{ks} g_{mv} R^{ursv} $, with~$R^{ursv}$ given  by~\eqref{drct}.
\end{proof}
\begin{proposition}
\label{pr:RDTa}
The dual Riemannian curvature tensor
may also be written
as follows:
\begin{subequations}
\begin{align}
2R^{ursv}  = & 
-g^{us,rv} -g^{rv,us} +g^{rs,uv} +g^{uv,rs} 
\tag{T$_1$} \label{T1a}
\\ & 
- \frac{1}{2} \big\{ \gul{rs}{\rho} g^{\rho\sigma} \gul{uv}{\sigma}\!\!
- \gul{rs}{\rho}\! \big( g^{\rho u,v} + g^{\rho v,u} \big)
- \gul{uv}{\sigma}\! \big( g^{\sigma r,s} + g^{\sigma s,r} \big) \big\}
\tag{T$_2$} \label{T2a}
\\
&  +\frac{1}{2} \big\{ \gul{rv}{\rho} g^{\rho\sigma} \gul{us}{\sigma}\!\!
- \gul{rv}{\rho}\! \big( g^{\rho u,s} + g^{\rho s,u} \big)
- \gul{us}{\sigma}\! \big( g^{\sigma r,v} + g^{\sigma v,r} \big) \big\}
\tag{T$_3$} \label{T3a}
\\ & 
-\frac{1}{2} \big( g^{\lambda r,s} - g^{\lambda s,r} \big)
g_{\lambda\mu} \big( g^{\mu u,v} - g^{\mu v,u} \big)
\tag{T$_4$} \label{T4a}
\\ & 
+\frac{1}{2} \big( g^{\lambda r,v} - g^{\lambda v,r} \big)
g_{\lambda\mu} \big( g^{\mu u,s} - g^{\mu s,u} \big)
\tag{T$_5$} \label{T5a}
\\ & 
+ \big( g^{\lambda r,u} - g^{\lambda u,r} \big)
g_{\lambda\mu} \big( g^{\mu v,s} - g^{\mu s,v} \big).
\tag{T$_6$} \label{T6a}
\end{align}
\end{subequations}
\end{proposition}

\begin{proof}
We will expand and recombine 
the terms in expression~\eqref{drct}.
The terms involving second derivatives need no 
manipulation and correspond to term~$\mathrm{T}_1$. 
The terms in the second line of~\eqref{drct}
can be written as:
\begin{align*}
g^{r\lambda,u} g_{\lambda\mu}\, g^{\mu v,s}\!
- g^{r\lambda,u} g_{\lambda\mu}\,
&
g^{\mu s,v}\! + g^{u\lambda,r} g_{\lambda\mu}\, g^{\mu s,v}\!
- g^{u\lambda,r} g_{\lambda\mu}\, g^{\mu v,s}\!
= ( g^{\lambda r,u}\! - g^{\lambda u,r} ) g_{\lambda\mu} ( g^{\mu v,s}\! - g^{\mu s,v} )
\end{align*}
which is precisely~$\mathrm{T}_6$. It is also the case that:
\begin{align*}
2\, \Gamma&^{rv}_\rho \, \Gamma^{us}_\sigma \, g^{\rho\sigma}
- g^{r\lambda,v} g_{\lambda\mu}\, g^{\mu s,u}
- g^{u\lambda,s} g_{\lambda\mu}\, g^{\mu v,r}
\\ =
& \; {\textstyle\frac{1}{2}}
\big[ ( g^{\lambda r,v} + g^{\lambda v,r} ) - g^{rv,\lambda} \big]
g_{\lambda\rho}\, g^{\rho\sigma} g_{\sigma\mu}
\big[ ( g^{\mu u,s} + g^{\mu s,u} ) - g^{us,\mu} \big]
- g^{r\lambda,v} g_{\lambda\mu}\, g^{\mu s,u}
- g^{u\lambda,s} g_{\lambda\mu}\, g^{\mu v,r}
\\ =
& \; {\textstyle\frac{1}{2}}
\big\{ \gul{rv}{\rho} \! g^{\rho\sigma} \gul{us}{\sigma}
\! - \gul{rv}{\rho} \! ( g^{\rho u,s} + g^{\rho s,u} )
- \gul{us}{\sigma} \! ( g^{\sigma r,v} + g^{\sigma v,r} ) \big\}
\\ &
+ {\textstyle\frac{1}{2}}
( g^{\lambda r,v} + g^{\lambda v,r} ) g_{\lambda\mu}
( g^{\mu u,s} + g^{\mu s,u} )
- g^{r\lambda,v} g_{\lambda\mu}\, g^{\mu s,u}
- g^{u\lambda,s} g_{\lambda\mu}\, g^{\mu v,r}
\\
= & \; \mathrm{T}_3
+ {\textstyle\frac{1}{2}} ( g^{\lambda r,v} - g^{\lambda v,r} )
g_{\lambda\mu} ( g^{\mu u,s} - g^{\mu s,u} )
= \mathrm{T}_3 +\mathrm{T}_5.
\end{align*}
Similarly one can prove that:
$-2 \Gamma^{rs}_\rho \Gamma^{uv}_\sigma g^{\rho\sigma}
+ g^{r\lambda,s} g_{\lambda\mu}\, g^{\mu v,u}
+ g^{u\lambda,v} g_{\lambda\mu}\, g^{\mu s,r}
=\mathrm{T}_2+\mathrm{T}_4
$.
\end{proof}
\par
For any point $p\in \mathcal{M}$ and an arbitrary 
pair of  tangent vectors~$X=X^i\pari$,
$Y=Y^i\pari$ in~$T_p\mathcal{M}$ we consider the covectors
$X^\flat=X_idx^i$ and $Y^\flat=Y_idx^i$ in~$T^\ast_p\mathcal{M}$,
with $X_i=g_{ij}X^j$ and  $Y_i=g_{ij}Y^j$.
The numerator of sectional curvature~\eqref{seccurv}
may be rewritten as
$ R_{ijkm} X^iY^jY^kX^m  = R^{ursv} X_u Y_r Y_s X_v.$

\begin{thm_mf}
\label{th:MaFo}	
For an arbitrary pair of vectors~$X=X^i\partial_i$ and
$Y=Y^i\partial_i$ in $T_p\mathcal{M}$
the numerator of sectional curvature~\eqref{seccurv}
at point~$p\in\mathcal{M}$ may be written as:
\begin{equation*}
\boxed{
\begin{aligned}
&g\big(R(X,Y)Y,X\big) =
R^{ursv} X_u Y_r Y_s X_v =
\\&\quad
=  \big ( X_u Y_r - Y_u X_r \big)
\Big( \tfrac12g^{su,rv} +\tfrac12\gul{us}{\rho}\! g^{\rho r,v} -\tfrac{1}{8}\, \gul{us}{\sigma}\! g^{rv,\sigma}
-\tfrac{3}{4}\, g^{ \lambda u,r} g_{\lambda\mu}\, g^{\mu s,v} \Big)
\big(X_s Y_v - Y_s X_v \big).
\end{aligned}}
\end{equation*}
Moreover, if we extend $X^\flat$ and $Y^\flat$ locally on $\mathcal M$ to {\em constant} 1-forms in terms of local coordinates (i.e.~make its coefficients $X_u, Y_r$ constant functions), then the formula becomes: 
\begin{equation*}
\boxed{
\begin{aligned}
&g\big(R(X,Y)Y,X\big) = \\ & \quad
=\left\{\tfrac12XX(\|Y^\flat\|^2)
+\tfrac12YY(\|X^\flat\|^2)
-\tfrac12(XY+YX)g^{-1}(X^\flat,Y^\flat)\right\}
\\&\qquad\quad
+\left\{\tfrac14\|d(g^{-1}(X^\flat,Y^\flat))\|^2
-\tfrac14g^{-1}\big(d(\|X^\flat\|^2),d(\|Y^\flat\|^2)\big)\right\}
-\tfrac34g\big([X,Y],[X,Y]\big),
\end{aligned}}
\end{equation*}
where the term in the first set of braces equals the sum of the first two terms in the coordinate form, the term in the second set of braces equals the third term in the coordinate form and finally the last terms are equal.
In the above formula, $\|X^\flat\|^2=X_sX_ug^{su}$ and $\|Y^\flat\|^2=Y_rY_vg^{rv}$.
\end{thm_mf}
\begin{proof}
We will write the six terms  provided by Proposition~\ref{pr:RDTa} as~$\mathrm{T}_i^{ursv}$,
$i=1,\ldots,6$. 
We have: 
\begin{align*}
\mathrm{T}_1^{ursv} X_uY_rY_sX_v &=
-g^{us,rv} X_uY_rY_sX_v -g^{rv,us} X_uY_rY_sX_v +g^{rs,uv} X_uY_rY_sX_v +g^{uv,rs}  X_uY_rY_sX_v
\\ &
= g^{us,rv} (- X_uY_rY_sX_v - X_rY_uY_vX_s + X_rY_uY_sX_v + X_uY_rY_vX_s )
\\ &=
g^{us,rv} (X_uY_r-Y_uX_r) (X_sY_v-Y_sX_v),
\end{align*}
where the second step follows from relabeling the indices.
As far as~$\mathrm{T}_2$ and~$\mathrm{T}_3$
are concerned,
\begin{align*}
(\mathrm{T}_2^{ursv} +&\,\mathrm{T}_3^{ursv}) \, X_uY_rY_sX_v =
-\textstyle{\frac{1}{2}} \big\{
Y_rY_s \gul{rs}{\rho} g^{\rho\sigma} \gul{uv}{\sigma}\!\! X_uX_v -
Y_rX_v \gul{rv}{\rho} g^{\rho\sigma} \gul{us}{\sigma}\!\! X_uY_s \big\}
\\ & \qquad\qquad\qquad\qquad
\,\,\,
+\! {\textstyle\frac{1}{2}} \big \{
Y_rY_s\, \gul{rs}{\rho}\!\! \big( g^{\rho u,v} + g^{\rho v,u} \big) X_uX_v
+ X_uX_v\, \gul{uv}{\rho}\!\! \big( g^{\rho r,s} + g^{\rho s,r} \big) Y_rY_s
\\ &
\qquad\qquad\qquad\qquad
\,\,\,
- Y_rX_v\, \gul{rv}{\rho}\!\! \big( g^{\rho u,s} + g^{\rho s,u} \big) X_uY_s
- X_uY_s\, \gul{us}{\rho}\!\! \big( g^{\rho r,v} + g^{\rho v,r} \big) Y_rX_v \big\}
\\
= & -{\textstyle\frac{1}{4}}\big\{
2Y_rY_s \gul{rs}{\rho} g^{\rho\sigma} \gul{uv}{\sigma}\!\! X_uX_v
- 2Y_rX_v \gul{rv}{\rho} g^{\rho\sigma} \gul{us}{\sigma}\!\! X_uY_s
\big\}
\\ &
+{\textstyle\frac{1}{2}}\big\{
2 Y_rY_s\, \gul{rs}{\rho}\! g^{\rho u,v} X_uX_v
+ 2 X_uX_v\, \gul{uv}{\sigma}\! g^{\sigma r,s} Y_rY_s\,
- 2 X_uY_s\, \gul{us}{\rho}\!\! \big( g^{\rho r,v} + g^{\rho v,r} \big) Y_rX_v
\big\}
\\
\stackrel{(\ast)}{=} &
-{\textstyle\frac{1}{4}}\, \gul{rv}{\rho} g^{\rho\sigma} \gul{us}{\sigma}\!\! \big\{
Y_uY_sX_rX_v + Y_rY_vX_uX_s - Y_rX_vX_uY_s - Y_vX_rX_sY_u \big\}
\\
& + \gul{us}{\rho}\! g^{\rho r,v} \big\{
Y_uY_sX_rX_v + X_uX_sY_rY_v - X_uY_sY_rX_v - X_sY_uY_vX_r \big\}
\\
= & \big( -{\textstyle\frac{1}{4}}\, \gul{us}{\sigma}\! g^{rv,\sigma} + \gul{us}{\rho}\! g^{\rho r,v} \big)
(X_uY_r-Y_uX_r) (X_sY_v-Y_sX_v),
\end{align*}
where, once again, step~$(\ast)$ follows from relabeling the indices.
Also, one can easily see that
$ \mathrm{T}_4^{ursv} X_uY_rY_sY_v =
-\frac{1}{2} Y_rY_s
( g^{\lambda r,s} - g^{\lambda s,r} ) g_{\lambda\mu}
( g^{\mu u,v} - g^{\mu v,u} ) X_uX_v =0.$
Finally,
\begin{align*}
(\mathrm{T}_5&^{ursv}+ \mathrm{T}_6^{ursv}) X_uY_rY_sY_v
\\
=& \, {\textstyle \frac{1}{2}}\,
Y_rX_v ( g^{\lambda r,v} - g^{\lambda v,r} ) g_{\lambda\mu} ( g^{\mu u,s} - g^{\mu s,u} ) X_u Y_s
+ Y_rX_u ( g^{\lambda r,u} - g^{\lambda u,r} ) g_{\lambda\mu} ( g^{\mu v,s} - g^{\mu s,v} ) X_vY_s
\\
=& \, {\textstyle\frac{3}{2}} \, 
Y_rX_u ( g^{\lambda r,u} - g^{\lambda u,r} ) g_{\lambda\mu} ( g^{\mu v,s} - g^{\mu s,v} ) X_vY_s
\\
=& \, {\textstyle\frac{3}{2}}\,
Y_rX_u \big\{ g^{\lambda r,u} g_{\lambda\mu} g^{\mu v,s} - g^{\lambda r,u} g_{\lambda\mu} g^{\mu s,v}
- g^{\lambda u,r} g_{\lambda\mu} g^{\mu v,s}
+ g^{\lambda u,r} g_{\lambda\mu} g^{\mu s,v} \big\} X_vY_s
\\
=& -{\textstyle\frac{3}{2}}\,
g^{\lambda u,r} g_{\lambda\mu} g^{\mu s,v} \big\{
-Y_uX_rX_sY_v +Y_uX_rX_vY_s +Y_rX_uX_sY_v -Y_rX_uX_vY_s \big\}
\\
=& -{\textstyle\frac{3}{2}}\,
g^{\lambda u,r} g_{\lambda\mu} g^{\mu s,v} (X_uY_r-Y_uX_r) (X_sY_v-Y_sX_v).
\end{align*}
Divide by 2 to get the coordinate formula. The non-local version of the formula follows easily by bringing the $X$ and $Y$'s into the formula. Thus (indicating~$\partial_i$ with the subscript $_{,i}$):
\begin{align*}
Y_u X_r 
&
(g^{su,rv}
+
\gul{su}{\rho}\!
g^{\rho r, v})Y_s X_v 
=
X_rX_v
\big(Y_sY_u\,\, \gul{su}{\rho\sigma}g^{\rho\xi}g^{\sigma\eta}
+
Y_sY_u\,\,\gul{su}{\rho}\gul{\rho r}{\sigma}g^{\sigma v}\big)
\\
&= X_r X_v \big( (\|Y^\flat\|^2)_{,\rho \sigma}\, g^{\rho r} g^{\sigma v} + (\|Y^\flat\|^2)_{,\rho}\, 
\gul{\rho r}{\sigma}
g^{\sigma v} \big)
\qquad
\qquad
\quad
\mbox{(because } Y_s,Y_u\mbox{ are constants)}
\\
&
= X_v g^{\sigma v} \big( X_r g^{\rho r} (\|Y^\flat\|^2)_{,\rho}\big)_{,\sigma} 
= X^\sigma \big(X^\rho (\|Y^\flat\|^2)_{,\rho} \big)_{,\sigma} = 
XX\big(\|Y^\flat\|^2\big).
\end{align*}
A typical term from the third part of Mario's formula is rewritten like this:
\begin{align*}
Y_u X_r\, \gul{us}{\sigma} \!g^{rv,\sigma} Y_s X_v &=
X_r X_v \big(\|Y^\flat\|^2\big)_{,\sigma} \,g^{rv,\sigma} 
= \big(\|Y^\flat\|^2\big)_{,\sigma} \big(\|X^\flat\|^2\big)_{,\rho} \,g^{\rho \sigma}
= g^{-1}\big( d(\|Y^\flat\|^2), d(\|X^\flat\|^2)\big);
\end{align*}
the other terms are similar.
Finally, it is the case that:
\begin{align*}
(X_uY_r-Y_uX_r)g^{\lambda u, r}\partial_\lambda
&=
(X_uY_r-Y_uX_r)\gul{\lambda u}{\eta}g^{\eta r}\partial_\lambda
=
(X_uY^\eta-Y_uX^\eta)\gul{\lambda u}{\eta}\partial_\lambda
\\
&=
\big((X_u g^{\lambda u})_{,\eta}\,Y^\eta
-(Y_u g^{\lambda u})_{,\eta}\,X^\eta\big)\partial_\lambda
=
\big(X^\lambda_{\;\;,\eta}\,Y^\eta
-Y^\lambda_{\;\;,\eta}\,X^\eta)\partial_\lambda
=-[X,Y],
\end{align*}
and the proof is easily completed.
\end{proof}

\begin{remark}
It is convenient to split Mario's  formula in four terms:
\begin{align}
\label{split_1}
R_1:= & \,\textstyle\frac{1}{2} ( X_u Y_r - Y_u X_r \big) g^{su,rv} \big(X_s Y_v - Y_s X_v \big),
\\
\label{split_2}
R_2:= & \,\textstyle\frac{1}{2} ( X_u Y_r - Y_u X_r \big) \gul{us}{\rho} g^{\rho r,v} \big(X_s Y_v - Y_s X_v \big),
\\
\textstyle
\label{split_3}
R_3:= & \,\textstyle\frac{1}{2} ( X_u Y_r - Y_u X_r \big) \big( - \tfrac{1}{4}\, \gul{us}{\sigma}\! g^{rv,\sigma} \big)
\big(X_s Y_v - Y_s X_v \big),
\\
\label{split_4}
R_4:= & \,\textstyle\frac{1}{2}
( X_u Y_r - Y_u X_r \big) \big( -\tfrac{3}{2}\, g^{ \lambda u,r} g_{\lambda\mu}\, g^{\mu s,v} \big) 
\big(X_s Y_v - Y_s X_v \big);
\end{align}
all the terms with the exception of~$R_4$
(where $g$ appears, but \em not \em its derivatives) depend only on elements of the cometric and their derivatives.
\end{remark}
\begin{remark} 
The  \em denominator \em of sectional curvature~\eqref{seccurv} 
can also be expressed in terms of the cometric: 
\begin{equation}
\label{denom}
\|X\|_g^2\|Y\|_g^2-\langle X,Y \rangle_g^2= X_uX_sY_rY_v(g^{us}g^{rv}-g^{uv}g^{sr}).
\end{equation}
\end{remark}

\section{Curvature of the Manifolds of Landmarks}
\label{secCurvLM}

In this section we will apply Mario's formula to the computation of sectional curvature for the Riemannian manifold of landmarks, introduced in section~\ref{Mot_Land}. We first introduce the Hamiltonian 
formalism, since it will allow us to write the geodesic equations in a simple form and to introduce
geometric quantities that will eventually appear in the formula for sectional curvature.

\subsection{Hamiltonian formalism}
\label{ham_form}
On the $ND$-dimensional manifold $\mathcal L=\mathcal L^N(\mathbb R^D)$ of landmarks we consider 
the Riemannian metric~$g$ given, in coordinates, by the matrix~\eqref{def:met_t}; it is in block-diagonal form
and we write its generic element as~$g_{(ai)(bj)}$, with
$a,b=1,\ldots,N$ (landmark labels) and $i,j=1,\ldots,D$ (coordinate labels, respectively of landmarks~$a$ and~$b$). More precisely: the matrix $g(q)$ is made of~$D$ square \mbox{($N\times N$)} blocks; indices~$i,j=1,\ldots,D$ indicate the block, whereas indices~$a,b=1,\ldots,N$ locate the element within the~$(i,j)$-block. Therefore if we indicate with~$h_{ab}(q)$ 
the generic element of the ~$N\times N$ matrix $\big(\kermat(q)+\frac{\mathbb I_N}{\lambda}\big)^{-1}$ we have that 
$$
g_{(ai)(bj)}=h_{ab}(q)\,\delta_{ij}, \qquad a,b=1,\ldots,N, \quad i,j=1,\ldots,D,
$$
where~$\delta_{ij}$ is Kronecker's delta. Similarly, if we indicate as~$g^{(ai)(bj)}$ the elements of the cometric tensor~$g(q)^{-1}$, they are given by
$ g^{(ai)(bj)}(q)=h^{ab}(q)\,\delta^{ij}$,
where $h^{ab}(q)=K(q^a-q^b)+\frac{\delta^{ab}}{\lambda}$. In analogy with the notation introduced in section~\ref{sec:MaFo}
we also denote the partial derivatives by
$\gul{(ai)(bj)}{(ck)}=\frac{\partial}{\partial q^{ck}}g^{(ai)(bj)}$ and $\gul{(ai)(bj)}{(ck)(d\ell)}=
\frac{\partial^2}{\partial q^{ck}\partial q^{d\ell}}g^{(ai)(bj)}$;
they will be computed later.

For simplicity \underline{from now on we shall assume
that~$\lambda=\infty$}, i.e.~that we are dealing with \em exact matching \em of landmarks
so that~$\mathcal{L}^N(\mathbb{R}^D)$ has the form~\eqref{lambdainf}.
The element of the cometric becomes $ g^{(ai)(bj)}(q)=K(q^a-q^b)\,\delta^{ij}$ and 
the Hamiltonian~\cite[p.~50]{jost} for the system can be written as:
$$
\mathcal{H}(p,q) =\frac{1}{2}\,p^Tg(q)^{-1}p
=\frac{1}{2}\sum_{a,b=1}^N\sum_{i,j=1}^D g^{(ai)(bj)}(q)\,p_{ai}p_{bj}
= \frac{1}{2} \sum_{a,b=1}^N\sum_{i,j=1}^D K(q^a-q^b) \,\delta^{ij}\,p_{ai}p_{bj},
$$
that is $
\displaystyle
\mathcal{H}(p,q) =\frac{1}{2} \sum_{a,b=1}^N K(q^a-q^b)\big\langle p_a,p_b\big\rangle_{\mathbb{R}^D}.$

\begin{proposition}
\label{prop_ham}
Hamilton's equations for the Riemannian manifold of landmarks are: 
\begin{equation}
\label{ham_eq}
\begin{aligned}
\dot{q}^a & = \sum_{b=1}^NK(q^a-q^b)\,p_b \\
\dot{p}_a &= - \sum_{b=1}^N\nabla K(q^a-q^b)\,\big\langle p_a,p_b \big\rangle_{\mathbb{R}^D}
\end{aligned}
\qquad a=1,\ldots,N.
\end{equation}
\end{proposition}
\begin{proof} Equation~\eqref{Expr_Mom} can be written as 
$\dot{q}^{ai}=\sum_{b=1}^NK(q^a-q^b)\,p_{bi}$,
for $a=1,\ldots,N$, $i=1,\ldots,D$; alternatively,
computing $\dot{q}^{ai}=\frac{\partial \mathcal{H}}{\partial p_{ai}}$
yields the same result. Also:
\begin{align}
\textstyle\frac{\partial}{\partial q^{ai}}
K(q^{b1}-q^{c1},\ldots,q^{bD}-q^{cD}) &=
\textstyle\sum_{\ell=1}^D \frac{\partial K}{\partial x^{\ell}} (q^{b}-q^{c}) \frac{\partial}{\partial q^{ai}} (q^{b\ell}-q^{c\ell})
\nonumber \\
\label{der_K}
&=
\textstyle\sum_{\ell=1}^D \frac{\partial K}{\partial x^{\ell}} (q^{b}-q^{c})\, (\delta^b_a-\delta^c_a)\delta^\ell_i =
\textstyle\frac{\partial K}{\partial x^i} (q^{b}-q^{c})\, (\delta^b_a-\delta^c_a)
\end{align}
so that
\begin{align*}
\dot{p}_{ai} & = -\textstyle\frac{\partial \mathcal{H}}{\partial q^{ai}}(p,q)=
-\frac{1}{2}\sum_{c=1}^N \frac{\partial K}{\partial x^i} (q^{a}-q^{c})\, \,\langle p_a,p_c\rangle_{\mathbb{R}^D} 
+\frac{1}{2}\sum_{b=1}^N \frac{\partial K}{\partial x^i} (q^{b}-q^{a})\, \,\langle p_b,p_a\rangle_{\mathbb{R}^D} \\
& 
\stackrel{(\ast)}{=} -\textstyle\sum_{b=1}^N \frac{\partial K}{\partial x^i} (q^{a}-q^{b})\, \,\langle p_a,p_b\rangle_{\mathbb{R}^D};
\end{align*}
in $(\ast)$ we used the skew-symmetry of~$\nabla K(q^a-q^b)$ in indices~$a$ and $b$, which
follows from~(K2).
\end{proof}
\begin{corollary} 
\label{cor_momenta}
If 
$p_a(t_0)=0$ for some landmark~$a=1,\ldots, N$ and time $t_0\in \mathbb{R}$,
then $p_a(t)\equiv 0$.
\end{corollary}
\subsection{Notation}
\em From now on we shall also assume that\em~(K3) \em holds\em\/, i.e.~that the kernel $K$
is twice continuously differentiable;
for the time being we will not assume rotational invariance. We define:
\begin{equation}
\label{K_ders}
\begin{aligned}
& & K^{ab}&:= K(q^a-q^b) \in \mathbb R,  \\ 
\partial_iK(x) &:=
\frac{\partial K}{\partial x^i}(x), & \partial_iK^{ab} &:=
\partial_iK(q^a-q^b)\in \mathbb R, \\ 
\nabla K &:= (\partial_1 K, \cdots, \partial_D K)^T, & \nabla K^{ab} &:= \nabla  K(q^a-q^b)\in \mathbb R^D, \\
\partial^2_{ij} K(x) &:= \frac{\partial^2 K}{\partial x_i \partial x_j} (x), &
\partial^2_{ij} K^{ab} &:= \partial^2_{ij} K(q^a-q^b) \in \mathbb R, \\
D^2 \! K &:= \text{Hessian} (K),
&D^2 \! K^{ab} &:=  D^2 \! K (q^a-q^b) \in \mathbb R^{D\times D}.
\end{aligned} \end{equation}
Note that~$\nabla K^{ab}=-\nabla K^{ba}, \nabla K^{aa}=0$ and $D^2K^{ab}=D^2K^{ba}$, for all~$a,b=1,\ldots,N$, by~(K2).
\par
For a fixed set of landmark points $q$ in
$\mathcal{L}=\mathcal{L}^N(\mathbb{R}^D)$
consider any pair of cotangent vectors 
$\alpha,\beta\in T_q^\ast\mathcal{L}$: we shall write
$\alpha=(\alpha_1,\ldots,\alpha_N)$ and $\beta=(\beta_1,\ldots,\beta_N)$, 
where each component is $D$-dimensional. We define the vector field $\alpha^{\text{hor}}:\mathbb{R}^D\rightarrow \mathbb{R}^D$ and its values at the landmark \mbox{points by:}
\begin{align*}
\alpha^{\text{hor}}(x)  :=&\,\, \sum_{b=1}^N K(x-q^b)\alpha_b
, \quad x\in\mathbb{R}^D, \\
(\alpha^\sharp)^a  
:=&\,\, \alpha^{\text{hor}}(q^a) = 
\sum_{b=1}^N K^{ab}\alpha_b,
\end{align*}
which are, by virtue of formula~\eqref{vstar1}, the velocity field $\alpha^{\text{hor}}$ on 
$\mathbb R^D$ induced by the landmark 
momentum~$\alpha=(\alpha_1,\ldots,\alpha_N)$ and the corresponding landmark velocity 
$\alpha^\sharp\in T_q\mathcal L$ (which obviously 
coincides with the first of~Hamilton's equations~\eqref{ham_eq}). Note that  
$\alpha^\sharp=(\alpha_1^\sharp,\ldots,\alpha_N^\sharp)$
is the tangent vector in~$T_q\mathcal{L}$ with metrically lifted indices.
Note that~$\alpha^\mathrm{hor}$ is
the \em horizontal lift\em~\cite[p.~148]{gallot:04} of the tangent vector~$\alpha^\sharp$
on the admissible Hilbert space~$V$: simply put, of all vector 
fields~$v:\mathbb{R}^D\rightarrow \mathbb{R}^D$
in~$V$ such that $v(q^a)=(\alpha^\sharp)^a$,
$a=1,\ldots, N$,
$\alpha^\mathrm{hor}$ is the one of minimum norm.
\par
The curvature of the Riemannian manifold of landmarks will be expressed in terms of three auxiliary quantities which we now introduce. We will call these {\it force}, {\it discrete strain} and {\it landmark derivative}. We start with the force.
For a fixed covector
$\alpha=(\alpha_1,\ldots,\alpha_N)\in T^\ast_q\mathcal{L}$,
having the dual vector extended to a vector field $\alpha^{\text{hor}}$ on all of $\mathbb R^D$ allows us to take its derivatives at the landmark points, a $D\times D$ matrix-valued function on $\mathbb R^D$:
\begin{align*}
(D\alpha^{\text{hor}})_{i}^j(x) &:= \partial_i(\alpha^{\text{hor}})^j(x)
=\sum_{b=1}^N\alpha_{bj} \partial_i K(x-q^b), \\
(D\alpha^{\mathrm{hor}})_{i}^j(q^a) &= \sum_{b=1}^N \partial_i K^{ab}\alpha_{bj}.
\end{align*} 
\par
For a trajectory~$(q(t),p(t))$ of the cotangent flow one has that
$(p_1(t),\ldots,p_N(t))\in T_{q(t)}^\ast\mathcal{L}$
for all~$t$ where the trajectory is defined, so the above notation can be used to rewrite Hamilton's equations in a more compact form. In  particular, the following result holds.
\begin{proposition}\label{prop_ham_eq1}
The second of Hamilton's equations~\eqref{ham_eq} can be written as
\begin{equation}
\label{ham_eq2}
\dot{p}_a=- Dp^{\mathrm{hor}}(q^a)\cdot p_a \qquad a=1,\ldots, N.
\end{equation}
\end{proposition}
\begin{proof}
$\dot{p}_{ai} = -\sum_{b=1}^N \partial_i K^{ab} \langle p_b,p_a \rangle_{\mathbb{R}^D}
=- \sum_{j=1}^D\big(\sum_{b=1}^N \partial_i K^{ab}\, p_{bj}\big) p_{aj}
=
- \sum_{j=1}^D
(Dp^{\mathrm{hor}})_i^j(q^a)\,p_{aj}
=
-\big(Dp^{\text{hor}}(q^a)\cdot p_a\big)_i
\,$, for any~$a=1,\ldots,N$ and $i=1,\ldots,D$. 
\end{proof}
\noindent{For} a fixed cotangent vector~$\alpha\in T_q^\ast \mathcal{L}$, this motivates defining the negative right-hand side of~\eqref{ham_eq2} to be \em force\em\/:
\begin{align*}
F_a(\alpha,\alpha)&:=D\alpha^{\text{hor}}(q^a)\cdot\alpha_a,\qquad a=1,\ldots, N.
\end{align*}
The full bilinear, symmetrized 
force
may be thought of as a map 
$F: T_q^\ast\mathcal{L}\times T_q^\ast\mathcal{L} \rightarrow T_q^\ast\mathcal{L}$. We call the 
covectors given by this the \em mixed force, \em with the definition:    
\begin{align} 
\notag F_a(\alpha,\beta) :=&\;
\tfrac12 \big( D\alpha^{\text{hor}}(q^a)\cdot\beta_a + D\beta^{\text{hor}}(q^a)\cdot\alpha_a \big), \\
F_{ai}(\alpha,\beta) :=& \;
\tfrac12 \sum_{j=1}^D\sum_{b=1}^N \partial_i K^{ab}\, \big(\alpha_{bj} \beta_{aj} + \beta_{bj} \alpha_{aj}\big) 
\label{Fdef} =
\tfrac12 \sum_{b=1}^N \partial_i K^{ab} \big( 
\langle\alpha_a,\beta_b\rangle_{\mathbb{R}^D} + \langle\beta_a,\alpha_b\rangle_{\mathbb{R}^D} \big),
\end{align}
for $a=1,\ldots,N$ and $i=1,\ldots,D$. 
(The angle brackets are inner products in $\mathbb R^D$.) Note that the ``complete'' cotangent vectors $\alpha=(\alpha_1,\ldots,\alpha_N)$ and $\beta=(\beta_1,\ldots,\beta_N)$ (not only their $a$-components) are needed to compute each component $F_a(\alpha,\beta)$ of the mixed force.
The mixed force has 
simple interpretation. If we extend $\alpha$ and $\beta$ to constant 1-forms on $\mathcal L$, then
the differential of the map 
$q\mapsto g^{-1}_q(\alpha,\beta) = 
\sum_{a,b}K(q^a-q^b)\langle \alpha_a,\beta_b \rangle_{\mathbb{R}^D}$ is given by:
\begin{align}
\label{R3coordfree}
d\big(g^{-1}_q(\alpha,\beta)\big) &= 
\sum_{a,b=1}^N
\sum_{i=1}^D
\partial_i K(q^a-q^b)\,(dq^{ai}-dq^{bi})\,
\langle \alpha_a,\beta_b \rangle_{\mathbb{R}^D} \nonumber
\\&
=
\sum_{a,b=1}^N
\sum_{i=1}^D
\partial_i K(q^a-q^b)\,\big(\langle \alpha_a,\beta_b \rangle_{\mathbb{R}^D} 
+ \langle \beta_a,\alpha_b \rangle_{\mathbb{R}^D}\big)\,dq^{ai}
= 2F(\alpha,\beta).
\end{align}
\par
For a fixed~$\alpha\in T^\ast_q \mathcal{L}$ 
we define the {\it discrete vector strain}:
$$ S^{ab}(\alpha) := (\alpha^\sharp)^a - (\alpha^\sharp)^b, \qquad \text{or} \qquad S^{ab}(\alpha)^i 
:= 
\sum_{c=1}^N \sum_{j=1}^D (K^{ac}-K^{bc})\delta^{ij}\alpha_{cj}
=\sum_{c=1}^N (K^{ac}-K^{bc})\alpha_{ci}$$ 
for all $a,b = 1, \ldots, N$ 
(we call it like that because it measures the infinitesimal
change of relative position of the landmarks~$a$ and~$b$ 
induced by the cotangent vector~$\alpha$). 
These are vectors and are skew-symmetric in the points $a,b$: $S^{ab}(\alpha)=-S^{ba}(\alpha)$, $S^{aa}(\alpha)=0$. The scalar quantities:
$$ C^{ab}(\alpha) := \big\langle (\alpha^\sharp)^a - (\alpha^\sharp)^b,\nabla K^{ab} \big\rangle_{\mathbb{R}^D} = 
\sum_{c=1}^N\sum_{i=1}^D
(K^{ac}-K^{bc})\, \partial_iK^{ab}\ \alpha_{ci}$$
we define to be the \em scalar compressions \em felt by kernel $K$; they are symmetric (since both factors in the inner product are skew-symmetric), i.e.~$C^{ab}(\alpha)=C^{ba}(\alpha)$, with the property $C^{aa}(\alpha)=0$. We call these compressions because if $K$ is a monotone decreasing function of the distance from the origin (the most common case), then $\nabla K^{ab}$ points from $q^a$ to $q^b$.
\par
Finally, if $v$ and $w$ are any two vector fields on the manifold of 
landmarks, we may write their Lie derivative as the difference of covariant derivatives:
$$ [v,w]_\mathcal{L} = \nabla_v^{\mathcal L, \text{flat}}(w) - \nabla_w^{\mathcal L, \text{flat}}(v)$$
where the flat connection on $\mathcal L$ is just the one induced by its embedding in $\mathbb R^{ND}$. In other words, 
$\nabla_v^{\mathcal L, \text{flat}}(w)$ is the usual derivative of $w$ in the direction $v$ if we use the coordinates $q^{ai}$ on landmark space:
that is, $\nabla_v^{\mathcal L, \text{flat}}(w):=\sum_{ai}
v(w^{ai})\partial_{ai}=
\sum_{ai}\sum_{bj}
v^{bj}(\partial_{bj}w^{ai})\partial_{ai}$. If $\alpha$, $\beta$ are constant 1-forms everywhere on 
$\mathcal L^N$ we can 
take $v=\alpha^\sharp$ and $w=\beta^\sharp$, now as vector fields on $\mathcal L$, and then we find:
\begin{align*}
\nabla_{\alpha^\sharp}^{\mathcal L, \mathrm{flat}}(\beta^\sharp)
&=
\sum_{a,i}\sum_{b,j}
(\alpha^\sharp)^{bj}
\frac{\partial}{\partial q^{bj}}(\beta^\sharp)^{ai}\,\partial_{ai}
=
\sum_{a,i}\sum_{b,j}
(\alpha^\sharp)^{bj}
\Big(\frac{\partial}{\partial q^{bj}}\sum_cK(q^a-q^c)\beta_{ci}\Big)
\partial_{ai}
\\
&=
\sum_{a,i}\sum_{b,c,j}
(\alpha^\sharp)^{bj}\,
\partial_jK^{ac}\,(\delta^a_b-\delta^c_b)\,\beta_{ci}\,\partial_{ai}
=
\sum_{a,i}\sum_{b,j}
\big(
(\alpha^\sharp)^{aj}-
(\alpha^\sharp)^{bj}
\big)\partial_j
K^{ab}
\,\beta_{bi}\,\partial_{ai}
\\
&=
\sum_{a,i}
\sum_b
\big\langle
(\alpha^\sharp)^a
-
(\alpha^\sharp)^b,
\nabla K^{ab}
\big\rangle_{\mathbb{R}^D}
\beta_{bi}\,\partial_{ai}
=
\sum_{a,i}
\Big(
\sum_b
C^{ab}(\alpha)\beta_{bi}
\Big)
\partial_{ai}.
\end{align*}
This is a vector in $T_q\mathcal L$ which we define to be the {\it landmark derivative} of $\beta^\sharp$ with respect to $\alpha^\sharp$.
The coefficients with respect to~$\partial_{a1},\ldots,\partial_{aD}$
(for fixed~$a$) are the elements of the following vector:
%
\begin{align} \label{landmarkderdef} 
D^a(\alpha,\beta) &:= \sum _{b=1}^N C^{ab}(\alpha)\beta_b 
= \sum_{b,c=1}^N (K^{ac}-K^{bc})\langle \alpha_c, \nabla K^{ab} \rangle_{\mathbb{R}^D} \beta_b\,,
\qquad a=1,\ldots,N.
\end{align}
We have that~$D(\alpha,\beta)=(D^a(\alpha,\beta))_{a=1}^N$
is the~$ND$-dimensional vector of the coefficients of~$\nabla_{\alpha^\sharp}^{\mathcal L, \mathrm{flat}}(\beta^\sharp)$ with respect to
the basis~$\{\partial_{ai}\}$ of $T_q\mathcal{L}$.
In particular, the coefficients of the Lie bracket of $\alpha^\sharp$ and $\beta^\sharp$ as vector fields on $\mathcal L$ are given by $D(\alpha,\beta) - D(\beta,\alpha)$.

\subsection{General formula for the sectional curvature of~$\mathcal{L}^N(\mathbb{R}^D)$}
We can write sectional curvature of~$\mathcal{L}^N(\mathbb{R}^D)$ 
in the following way, where we have split it in the terms introduced 
by~\eqref{split_1}--\eqref{split_4}. 
\par
{\bf Notation:} from now on~$\langle\;\;,\;\:\rangle$ will indicate the  dot product in~$\mathbb{R}^D$,
while~$\langle\;\;,\;\:\rangle_{T\mathcal{L}}$
and~$\langle\;\;,\;\:\rangle_{T^\ast\mathcal{L}}$
will be the inner products in the tangent and cotangent 
bundles of~$\mathcal{L}=\mathcal{L}^N(\mathbb{R}^D)$,
respectively.
\begin{theorem}
\label{ThCurvLM}
The numerator of sectional curvature 
of~$\mathcal{L}^N(\mathbb{R}^D)$, for an arbitrary pair of cotangent vectors~$\alpha$ and
$\beta$, is given
by~$R(\alpha^\sharp,\beta^\sharp,\beta^\sharp,\alpha^\sharp)=\sum_{i=1}^4R_i$, with:
\begin{align}
\label{gen_R1}
R_1 & = \tfrac{1}{2} \sum_{a\not=b} \big( \alpha_a \otimes S^{ab}(\beta) 
- \beta_a \otimes S^{ab}(\alpha) \big) ^T\big( \mathbb{I}_D \otimes D^2K^{ab} \big) 
\big( \alpha_b \otimes S^{ab}(\beta) - \beta_b \otimes S^{ab}(\alpha) \big),
\\
\label{gen_R2}
R_2 & = 
\sum_a \Big(
\big \langle D^a(\alpha,\alpha), F_a(\beta,\beta) \big \rangle
+ \big \langle D^a(\beta,\beta), F_a(\alpha,\alpha) \big \rangle
- \big \langle D^a(\alpha,\beta) + D^a(\beta,\alpha), F_a(\alpha,\beta) \big \rangle \Big),
\\
\nonumber
R_3 &= \big\|F(\alpha,\beta)\big\|^2_{T^\ast\mathcal L} - \big\langle F(\alpha,\alpha), F(\beta,\beta) \big\rangle_{T^\ast\mathcal{L}} \\
\label{gen_R3}
& = \sum_{ac} K^{ac} \Big(
\big \langle F_a(\alpha,\beta) , F_c(\alpha,\beta) \rangle
- \big \langle F_a(\alpha,\alpha), F_c(\beta,\beta) \big \rangle
\Big),
\\
\label{gen_R4}
R_4 & = -\tfrac{3}{4}\big\|[\alpha^\sharp,\beta^\sharp]_{\mathcal L}\big\|^2_{T\mathcal L} = -\tfrac{3}{4}\big\|D(\alpha,\beta)-D(\beta,\alpha)\big\|^2_{\mathbf{ K}^{-1}}.
\end{align}
\end{theorem}
In the formula  we have used  the
 definition: 
$
(v_1\otimes v_2)^T(M_1\otimes M_2) (w_1\otimes w_2) := (v^T_1M_1w_1) (v^T_2M_2w_2)
$ for the first term $R_1$,
while  we have used the norm for~$D\times N$ matrices
$ \|J\|_A^2 := \sum_{i=1}^D \sum_{a,b=1}^N J_{ia}J_{ib}A_{ab}$ for the fourth term $R_4$.
\par
The theorem is proven by applying Mario's formula to the cometric of the manifolds of landmarks. 
One needs to compute the elements of the cometric and its derivatives in terms of the kernel and its derivatives~\eqref{K_ders}. In agreement with notation~\eqref{u_deriv_2} we will define (note that we will keep using Einstein's  summation convention wherever possible):
$$
g^{(ai)(bj),(d\ell)}   := \gul{(ai)(bj)}{(ck)} g^{(ck)(d\ell)}
\qquad
\mathrm{and}
\qquad
g^{(ai)(bj),(ck)(d\ell)}  :=  \gul{(ai)(bj)}{(\mu\rho)(\xi\sigma)} g^{(\mu\rho)(ck)}g^{(\xi\sigma)(d\ell)}.
$$
\begin{lemma}
\label{le:MetDer}
It is the case that
\begin{align}
\label{diff_1}
\gul{(ai)(bj)}{(ck)} & =  \partial_k K^{ab}\;(\delta^a_c-\delta^b_c)\,\delta^{ij},
\\
\label{diff_1b}
\gul{(ai)(bj)}{(ck)(d\ell)}  & =   \partial^2_{k \ell} K^{ab}\;(\delta^a_c-\delta^b_c)
\, (\delta^a_d-\delta^b_d)\,\delta^{ij},
\\
\label{diff_2}
g^{(ai)(bj),(d\ell)} & = \partial_\ell K^{ab}\;(K^{ad}-K^{bd})\,\delta^{ij}, 
\\
\label{diff_2b}
g^{(ai)(bj),(ck)(d\ell)} & =  \partial^2_{k\ell} K^{ab} \; (K^{ac}-K^{bc})\, (K^{ad}-K^{bd})\, \delta^{ij}. 
\end{align}
\end{lemma}
\begin{proof}
Since $g^{(ai)(bj)}=K^{ab}\delta^{ij}$
and also
$
\frac{\partial}{\partial q^{ck}} K(q^{a}-q^b)= \partial_kK^{ab}(\delta^a_c-\delta^b_c) $
by~\eqref{der_K},
equation~\eqref{diff_1} follows immediately. 
Similarly to~\eqref{der_K} one can prove that
$\frac{\partial}{\partial q^{d\ell}}\partial_kK(q^a-q^b) =\partial^2_{\ell k}K^{ab}\,(\delta_d^a-\delta_d^b)$,
whence:
$ \gul{(ai)(bj)}{(ck)(d\ell)}\! = \frac{\partial}{\partial q^{d\ell}} \, \gul{(ai)(bj)}{(ck)}\! =
\partial_{\ell k}^2K^{ab} \, (\delta^a_d-\delta^b_d) \, (\delta^a_c-\delta^b_c)\,\delta^{ij},
$
so~\eqref{diff_1b} holds too. Now, by expression~\eqref{diff_1}:
$$
g^{(ai)(bj),(d\ell)}  =  \gul{(ai)(bj)}{(ck)} g^{(ck)(d\ell)} =
\textstyle\sum_{ck} \partial_k K^{ab} \, (\delta^a_c-\delta^b_c)\,\delta^{ij} \, K^{cd} \,\delta^{k\ell}
=
\partial_\ell K^{ab}\;(K^{ad}-K^{bd})\,\delta^{ij}. 
$$
which is~\eqref{diff_2}. We can use~\eqref{diff_1b} to compute
$g^{(ai)(bj),(ck)(d\ell)} 
=  \gul{(ai)(bj)}{(\mu\rho)(\xi\sigma)} g^{(\mu\rho)(ck)}g^{(\xi\sigma)(d\ell)} $:
\begin{align*}
g^{(ai)(bj),(ck)(d\ell)}  & = \textstyle\sum_{\mu\rho\xi\sigma} \partial^2_{\rho\sigma} K^{ab} \,
(\delta^a_\mu-\delta^b_\mu) \, (\delta^a_\xi-\delta^b_\xi) \, \delta^{ij} \,
K^{\mu c} \, \delta^{\rho k} \, K^{\xi d} \, \delta^{\sigma \ell}
\\ & 
=  \partial^2_{k\ell} K^{ab} \; (K^{ac}-K^{bc})\, (K^{ad}-K^{bd})\, \delta^{ij}, 
\end{align*}
which completes the proof.
\end{proof} 
\begin{proof}[Proof of Theorem~\ref{ThCurvLM}]
We will compute terms~$R_1,\ldots,R_4$
introduced by formulae~\eqref{split_1}--\eqref{split_4}.
For simplicity, sometimes we will write
$D\alpha^\mathrm{hor}_a$
instead of~$D\alpha^{\mathrm{hor}}(q^a)$.
\\
\noindent{$\bullet$ \it Computation of~$R_1$.}
We have 
$R_1 = \frac{1}{2} ( \alpha_{au} \beta_{cr} - \beta_{au} \alpha_{cr} )\, g^{(au)(bs),(cr)(dv)}
(\alpha_{bs} \beta_{dv} - \beta_{bs} \alpha_{dv} )$.
Inserting expression~\eqref{diff_2b}
into such formula yields:
$$
2R_1 = \textstyle\sum_{\mathrm{all\; indices}} ( \alpha_{au} \beta_{cr}
- \beta_{au} \alpha_{cr} \big) \,
\partial^2_{rv} K^{ab} \, (K^{ac}-K^{bc})\, (K^{ad}-K^{bd}) \, \delta^{us} \, \big(\alpha_{bs} \beta_{dv}
- \beta_{bs} \alpha_{dv} \big).
$$
Performing the above multiplications gives rise to four terms, 
which we will now compute one by one. First of all we have:
\begin{align*}
2R_{1,1} & := \textstyle\sum_{\mathrm{all\;indices}} \alpha_{au}\beta_{cr}\alpha_{bs}\beta_{dv}
\, \partial^2_{rv} K^{ab} \, (K^{ac}-K^{bc})\, (K^{ad}-K^{bd}) \, \delta^{us}
\\
& = \textstyle\sum_{abrv} \big[ {\textstyle \sum_{us}} \alpha_{au}\delta^{us}\alpha_{bs} \big]
\big[ {\textstyle \sum_{c}} (K^{ac}-K^{bc}) \beta_{cr} \big]
\,\partial^2_{rv} K^{ab}\,
\big[ {\textstyle \sum_{d}} (K^{ad}-K^{bd}) \beta_{dv} \big]
\\
& = \textstyle\sum_{ab} \alpha_a^T \alpha_b \sum_{rv} 
S^{ab}(\beta)^r
\,
\partial^2_{rv} K^{ab}
\,
S^{ab}(\beta)^v
= \sum_{ab} \alpha_a^T \alpha_b \, \big( S^{ab}(\beta) \big)^T D^2 K^{ab} \,
S^{ab}(\beta);
\\
&= \textstyle\sum_{ab} \big(\alpha_a \otimes S^{ab}(\beta)\big)^T \big(\mathbb{I}_D\otimes D^2 K^{ab} \big)
\big(\alpha_b \otimes S^{ab}(\beta)\big),
\end{align*}
where, once again, the superscript $^T$ indicates the transpose of a vector; similarly,
\begin{align*}
\qquad
2R_{1,2} & :=
-\textstyle\sum_{\mathrm{all}}
\alpha_{au}\beta_{cr}\beta_{bs}\alpha_{dv}\, \partial^2_{rv} K^{ab}\, (K^{ac}-K^{bc})\, (K^{ad}-K^{bd})\, \delta^{us} \\
& =
-\textstyle\sum_{ab} \big(\alpha_a \otimes S^{ab}(\beta)\big)^T \big(\mathbb{I}_D\otimes D^2 K^{ab} \big)
\big(\beta_b \otimes S^{ab}(\alpha)\big),\\
2R_{1,3} & := -\textstyle\sum_{\mathrm{all}}
\beta_{au}\alpha_{cr}\alpha_{bs}\beta_{dv}\, \partial^2_{rv} K^{ab}\, (K^{ac}-K^{bc})\,
(K^{ad}-K^{bd})\, \delta^{us}\\ 
&= -\textstyle\sum_{ab} \big(\beta_a \otimes S^{ab}(\alpha)\big)^T \big(\mathbb{I}_D\otimes D^2 K^{ab} \big)
\big(\alpha_b \otimes S^{ab}(\beta)\big),\\
2R_{1,4} & := \textstyle\sum_{\mathrm{all}}
\beta_{au}\alpha_{cr}\beta_{bs}\alpha_{dv}\, \partial^2_{rv} K^{ab}\, (K^{ac}-K^{bc})\, (K^{ad}-K^{bd})\, \delta^{us}
\\
& = \textstyle\sum_{ab} \big(\beta_a \otimes S^{ab}(\alpha)\big)^T \big(\mathbb{I}_D\otimes D^2 K^{ab} \big)
\big(\beta_b \otimes S^{ab}(\alpha)\big).
\end{align*}
%
Now we can take the summation~$R_1=\sum_{i=1}^4 R_{1,i}$,
which yields precisely expression~\eqref{gen_R1}.
\\
\noindent{$\bullet$ \it Computation of~$R_2$.}
We may combine equations~\eqref{diff_1} and~\eqref{diff_2} 
from Lemma~\ref{le:MetDer} to get:
\begin{align}
\nonumber
\gul{(au)(bs)}{(\lambda\rho)} g^{(\lambda\rho)(cr),(dv)} & =
\textstyle\sum_{\lambda\rho} \partial_\rho K^{ab} \, (\delta^a_\lambda-\delta^b_\lambda)\,\delta^{us} \,
\partial_v K^{\lambda c}\;(K^{\lambda d}-K^{cd})\,\delta^{\rho r} 
\\ &
= \partial_r K^{ab} \big[ \partial_v K^{ac} (K^{ad}-K^{cd}) - \partial_v K^{bc} (K^{bd}-K^{cd}) \big]\delta^{us}.
\label{diff_3}
\end{align}
Inserting~\eqref{diff_3} into
$ 2R_2 =
( \alpha_{au} \beta_{cr} - \beta_{au} \alpha_{cr} )\,
\gul{(au)(bs)}{(\lambda\rho)} g^{(\lambda\rho)(cr),(dv)} (\alpha_{bs} \beta_{dv} - \beta_{bs} \alpha_{dv} )$ yields:
\begin{align*}
2R_2 = \textstyle\sum_{\mathrm{all\; indices}} \big\{ &
\quad
\alpha_{au}\beta_{cr}\alpha_{bs}\beta_{dv} \; \partial_r K^{ab} 
\big[ \partial_v K^{ac} (K^{ad}-K^{cd}) - \partial_v K^{bc} (K^{bd}-K^{cd}) \big]\delta^{us}
\\
&
- \alpha_{au}\beta_{cr}\beta_{bs}\alpha_{dv} \; \partial_r K^{ab} \big[ \partial_v K^{ac} (K^{ad}-K^{cd}) -
\partial_v K^{bc} (K^{bd}-K^{cd}) \big]\delta^{us}
\\
&
- \beta_{au}\alpha_{cr}\alpha_{bs}\beta_{dv} \; \partial_r K^{ab} \big[ \partial_v K^{ac} (K^{ad}-K^{cd}) -
\partial_v K^{bc} (K^{bd}-K^{cd}) \big]\delta^{us}
\\ &
+ \beta_{au}\alpha_{cr}\beta_{bs}\alpha_{dv} \; \partial_r K^{ab} \big[ \partial_v K^{ac} (K^{ad}-K^{cd}) -
\partial_v K^{bc} (K^{bd}-K^{cd}) \big]\delta^{us} \big\},
\end{align*}
which immediately implies:
\begin{align*}
\hspace*{-2cm}
R_2& =
\\
\tag{$=: R_{2,1}$}
&\textstyle{\frac{1}{2}\sum_{abcd}}
\langle\alpha_a,\alpha_b\rangle \langle\beta_c,\nabla K^{ab} \rangle
\big[\langle\beta_d,\nabla K^{ac} \rangle(K^{ad}\!-\!K^{cd})
\!-\!\langle\beta_d,\nabla K^{bc} \rangle(K^{bd}\!-\!K^{cd})\big]
\\
\tag{$=: R_{2,2}$}
-&\textstyle{\frac{1}{2}\sum_{abcd}}
\langle\alpha_a,\beta_b\rangle \langle\beta_c,\nabla K^{ab} \rangle
\big[\langle\alpha_d,\nabla K^{ac} \rangle(K^{ad}\!-\!K^{cd})
\!-\!\langle\alpha_d,\nabla K^{bc} \rangle(K^{bd}\!-\!K^{cd})\big]
\\
\tag{$=: R_{2,3}$}
-&\textstyle{\frac{1}{2}\sum_{abcd}}
\langle\beta_a,\alpha_b\rangle \langle\alpha_c,\nabla K^{ab} \rangle
\big[\langle\beta_d,\nabla K^{ac} \rangle(K^{ad}\!-\!K^{cd})
\!-\langle\beta_d,\nabla K^{bc} \rangle(K^{bd}\!-\!K^{cd})\big]
\\
\tag{$=: R_{2,4}$}
+&\textstyle{\frac{1}{2}\sum_{abcd}}
\langle\beta_a,\beta_b\rangle \langle\alpha_c,\nabla K^{ab} \rangle
\big[\langle\alpha_d,\nabla K^{ac} \rangle(K^{ad}\!-\!K^{cd})
\!-\langle\alpha_d,\nabla K^{bc} \rangle(K^{bd}\!-\!K^{cd})\big]
\end{align*}
We will now manipulate terms
$R_{2,1}$,\ldots,$R_{2,4}$ one by one.
Since~$\nabla K^{ab}=-\nabla K^{ba}$, by relabeling the indices we have
\begin{align*}
R_{2,1} & =
\textstyle\sum_{abcd}
\langle\alpha_a,\alpha_b\rangle \langle\beta_c,\nabla K^{ab} \rangle
\langle\beta_d,\nabla K^{ac} \rangle(K^{ad}-K^{cd})
\\
&= \textstyle\sum_{abc}
\langle\alpha_a,\alpha_b\rangle \langle\beta_c,\nabla K^{ab} \rangle
\big\langle{\textstyle\sum_d}K^{ad}\beta_d-{\textstyle\sum_d}K^{cd}\beta_d,\nabla K^{ac} \big\rangle
\\
&= \textstyle\sum_{abc}
\langle\alpha_a,\alpha_b\rangle \langle\beta_c,\nabla K^{ab} \rangle
\langle S^{ac}(\beta),\nabla K^{ac} \rangle
= \textstyle\sum_{abc}
\langle\alpha_a,\alpha_b\rangle \langle\beta_c,\nabla K^{ab} \rangle C^{ac}(\beta)
\\
&= \textstyle\sum_{ab}
\langle\alpha_a,\alpha_b\rangle \big\langle {\textstyle \sum_c} C^{ac}(\beta)\beta_c,\nabla K^{ab} \big\rangle
=
\textstyle\sum_{ab}
\langle\alpha_a,\alpha_b\rangle \langle D^a(\beta,\beta),\nabla K^{ab} \rangle
\\ 
&= \textstyle \sum_{ab}
D^a(\beta,\beta)^T \nabla K^{ab} \alpha_b^T \alpha_a = 
\sum_{a} D^a(\beta,\beta)^T
D\alpha_a^{\text{hor}} \cdot \alpha_a = \sum_{a} \langle D^a(\beta,\beta) ,
F_a(\alpha,\alpha) \rangle.
\end{align*}
Similarly, $R_{2,4}=\sum_{a}
\langle D^a(\alpha,\alpha) , F_a(\beta,\beta) \rangle$. It is also the case that
\begin{align*}
R_{2,2} =&
\!-\! \textstyle{\frac{1}{2}\sum_{abc}}
\langle\alpha_a,\beta_b\rangle \langle\beta_c,\nabla K^{ab} \rangle
\big[ \big\langle{\textstyle\sum_d}(K^{ad}\!-\!K^{cd})\alpha_d,\nabla K^{ac}  \big\rangle
\!-\! \big\langle{\textstyle\sum_d}(K^{bd}\!-\!K^{cd})\alpha_d,\nabla K^{bc}  \big\rangle \big]
\\
= & \!-\! \textstyle{\frac{1}{2}\sum_{abc}} \langle\alpha_a,\beta_b\rangle
\langle\beta_c,\nabla K^{ab} \rangle
\big[ \langle S^{ac}(\alpha),\nabla K^{ac} \rangle - \langle S^{bc}(\alpha),\nabla K^{bc} \rangle \big]
\\
= & \!-\! \textstyle{\frac{1}{2}\sum_{abc}}
\langle\alpha_a,\beta_b\rangle \langle\beta_c,\nabla K^{ab} \rangle
\big[ C^{ac}(\alpha) - C^{bc}(\alpha) \big];
\end{align*}
relabeling the indices (and using the fact 
that~$\nabla K^{ab}=-\nabla K^{ba}$) yields:
\begin{align*}
R_{2,2} = & - \textstyle{\frac{1}{2}\sum_{abc}}
\big[ \langle\alpha_a,\beta_b\rangle + \langle\alpha_b,\beta_a\rangle \big]
\langle\beta_c,\nabla K^{ab} \rangle C^{ac}(\alpha)
\\
= 
& 
- \textstyle{\frac{1}{2}\sum_{ab}}
\big[ \langle\alpha_a,\beta_b\rangle + \langle\alpha_b,\beta_a\rangle \big]
\big\langle {\textstyle\sum_c} C^{ac}(\alpha) \beta_c,\nabla K^{ab} \big\rangle
\\
= & - \textstyle{\frac{1}{2}\sum_{ab}}
\big[ \langle\alpha_a,\beta_b\rangle + \langle\alpha_b,\beta_a\rangle \big]
\langle D^a(\alpha,\beta),\nabla K^{ab} \rangle
\\
= 
& 
- \textstyle{\frac{1}{2}\sum_{ab}}
D^a(\alpha,\beta)^T
\big[ \nabla K^{ab} \beta_b^T \alpha_a + \nabla K^{ab} \alpha_b^T \beta_a \big]
\\
= & - \textstyle{\frac{1}{2}\sum_{a}}
D^a(\alpha,\beta)^T \big[ D\beta_a^{\text{hor}}\cdot \alpha_a + D\alpha_a^{\text{hor}}\cdot \beta_a \big]
=
- \textstyle{\sum_{a}}
\langle D^a(\alpha,\beta) , F_a(\alpha,\beta) \rangle.
\end{align*}
Similarly,
$R_{2,3}=- \textstyle{\sum_{a}} \langle D^a(\beta,\alpha) , F_a(\beta,\alpha) \rangle $.
By the symmetry of~$F_a(\cdot,\cdot)$,
\begin{equation*}
R_{2,2} + R_{2,3} = - \textstyle{\sum_{a}} \langle D^a(\alpha,\beta) +D^a(\beta,\alpha)
, F_a(\alpha,\beta) \rangle
\end{equation*}
Adding the above sum to the expressions for~$R_{2,1}$
and~$R_{2,4}$ finally yields~\eqref{gen_R2}. 
\\
\noindent{$\bullet$ \it Computation of~$R_3$.}
We have
$$ R_3 = - \frac{1}{8}( \alpha_{au} \beta_{cr} - \beta_{au} \alpha_{cr} \big) \gul{(au)(bs)}{(\eta\sigma)}
g^{(cr)(dv),(\eta\sigma)} \big(\alpha_{bs} \beta_{dv} - \beta_{bs} \alpha_{dv} \big). $$
But by Lemma~\ref{le:MetDer},
\begin{align*}
\gul{(au)(bs)}{(\eta\sigma)} g^{(cr)(dv),(\eta\sigma)} & = \textstyle{ \sum_{\eta=1}^N \sum_{\sigma=1}^D}
\partial_\sigma K^{ab}\, (\delta^a_\eta-\delta^b_\eta) \delta^{us} \, \partial_\sigma K^{cd}\,
(K^{c\eta}-K^{d\eta}) \delta^{rv}
\\
& =  \langle \nabla K^{ab}, \nabla K^{cd} \rangle
\delta^{us} \delta^{rv} (K^{ac}-K^{ad}-K^{bc}+K^{bd}),
\end{align*}
whence:
\begin{align}
\nonumber
-8&R_3 =\textstyle\sum_{\mathrm{all}} \big\{\langle \nabla K^{ab}, \nabla K^{cd} \rangle
(K^{ac}-K^{ad}-K^{bc}+K^{bd})
\\
\nonumber
&\cdot\big(  \alpha_{au}\beta_{cr}\alpha_{bs}\beta_{dv}\delta^{us}\delta^{rv}
\!-\!\alpha_{au}\beta_{cr}\beta_{bs}\alpha_{dv}\delta^{us}\delta^{rv}
\!-\!\beta_{au}\alpha_{cr}\alpha_{bs}\beta_{dv}\delta^{us}\delta^{rv}
\!+\!\beta_{au}\alpha_{cr}\beta_{bs}\alpha_{dv}\delta^{us}\delta^{rv}
\big)\big\}
\\
=&\textstyle\sum_{abcd}
\big\{\big( \langle\alpha_a,\alpha_b\rangle \langle\beta_c,\beta_d\rangle
-\langle\alpha_a,\beta_b\rangle \langle\beta_c,\alpha_d\rangle
-\langle\beta_a,\alpha_b\rangle \langle\alpha_c,\beta_d\rangle
+\langle\beta_a,\beta_b\rangle \langle\alpha_c,\alpha_d\rangle
\big)
\nonumber
\\ \nonumber
& \cdot\langle \nabla K^{ab},\nabla K^{cd} \rangle
(K^{ac}-K^{ad}-K^{bc}+K^{bd})\big\}
\end{align}
Relabeling the indices in the above expression 
yields:
\begin{align*}
-&8R_3= \textstyle\sum_{abcd}
\big[
8\langle\alpha_a,\alpha_b\rangle \langle\beta_c,\beta_d\rangle
-2\langle\alpha_a,\beta_b\rangle \langle\beta_c,\alpha_d\rangle
-2\langle\beta_a,\alpha_b\rangle \langle\alpha_c,\beta_d\rangle
\\
&
\qquad\qquad\qquad\;-2
\langle\alpha_a,\beta_b\rangle \langle\beta_d,\alpha_c\rangle
-2
\langle\beta_a,\alpha_b\rangle \langle\alpha_d,\beta_c\rangle
\big]
\langle \nabla K^{ab}, \nabla K^{cd} \rangle K^{ac}
\\
=& \textstyle\sum_{abcd}
K^{ac}
\big[
8\alpha_a^T\alpha_b(\nabla K^{ab})^T\nabla K^{cd} \beta_d^T\beta_c
-2\alpha_a^T \beta_b (\nabla K^{ab})^T\nabla K^{cd} \alpha_d^T \beta_c
-2\beta_a^T \alpha_b (\nabla K^{ab})^T\nabla K^{cd}\beta_d^T\alpha_c
\\
&
\qquad\qquad\quad 
-2\alpha_a^T \beta_b (\nabla K^{ab})^T\nabla K^{cd} \beta_d^T
\alpha_c
-2 \beta_a^T \alpha_b (\nabla K^{ab})^T\nabla K^{cd}\alpha_d^T\beta_c
\big]
\\
= &\textstyle\sum_{ac}
K^{ac}
\big[
8\alpha_a^T (D\alpha^{\text{hor}}_a)^T (D\beta^{\text{hor}}_c)
\beta_c
-2 \alpha_a^T (D\beta^{\text{hor}}_a)^T (D\alpha^{\text{hor}}_c)
\beta_c
-2 \beta_a^T (D\alpha^{\text{hor}}_a)^T (D\beta^{\text{hor}}_c)
\alpha_c
\\
&
\qquad\qquad\quad
-2 \alpha_a^T (D\beta^{\text{hor}}_a)^T 
(D\beta^{\text{hor}}_c) \alpha_c
-2 \beta_a^T (D\alpha^{\text{hor}}_a)^T (
D\alpha^{\text{hor}}_c) \beta_c \big] \\
= &\textstyle\sum_{ac} K^{ac} \big[
8\langle D\alpha^{\text{hor}}_a \cdot \alpha_a, 
D\beta^{\text{hor}}_c \cdot \beta_c,\rangle
-2 \langle D\alpha^{\text{hor}}_a \cdot\beta_a
+ D\beta^{\text{hor}}_a \cdot \alpha_a,
D\alpha^{\text{hor}}_c \cdot \beta_c
+ D\beta^{\text{hor}}_c \cdot \alpha_c \rangle \big].\\
= &\textstyle\sum_{ac}
K^{ac}
\big[
8\langle F_a(\alpha,\alpha), F_c(\beta,\beta) \rangle
-8 \langle F_a(\alpha,\beta), F_c(\alpha,\beta) \rangle
\big],
\end{align*}
which is precisely~\eqref{gen_R3}. Alternatively, this can be derived from formula \eqref{R3coordfree}.
\\
\noindent{$\bullet$ \it Computation of~$R_4$.}
It is the case that:
\begin{equation*}
R_4 = -\frac{3}{4}
(\alpha_{au} \beta_{cr} - \beta_{au} \alpha_{cr} \big)
g^{ (\xi\lambda) (au),(cr)} g_{ (\xi\lambda)(\eta\mu)}\,
g^{ (\eta\mu)(bs),(dv)} \big(\alpha_{bs} \beta_{dv}
-\beta_{bs} \alpha_{dv} \big).
\end{equation*}
By Lemma~\ref{le:MetDer}:
\begin{align*}
\textstyle\sum&_{aucr}
(
\alpha_{au}
\beta_{cr}
-
\beta_{au}
\alpha_{cr}
\big)
g^{ (\xi\lambda) (au),(cr)}
=
\textstyle\sum_{aucr}
(
\alpha_{au}
\beta_{cr}
-
\beta_{au}
\alpha_{cr}
\big)
\partial_rK^{\xi a}
\,
(
K^{\xi c} 
- 
K^{a c}
)
\,
\delta^{\lambda u}
\\
& 
=
\textstyle\sum_{au}
\big\{
\alpha_{au}
\big[
{\textstyle \sum_r}
\partial_rK^{\xi a}
(
{\textstyle \sum_c}
K^{\xi c} 
\beta_{cr}
)
\big]
-
\alpha_{au}
\big[
{\textstyle \sum_r}
\,
\partial_rK^{\xi a}
(
{\textstyle \sum_c}
K^{a c} 
\beta_{cr}
)
\big]
\\
&
\qquad
\qquad
-
\beta_{au}
\big[
{\textstyle \sum_r}
\partial_rK^{\xi a}
(
{\textstyle \sum_c}
K^{\xi c} 
\alpha_{cr}
)
\big]
+
\beta_{au}
\big[
{\textstyle \sum_r}
\partial_rK^{\xi a}
(
{\textstyle \sum_c}
K^{a c} 
\alpha_{cr}
)
\big]
\big\}
\,
\delta^{\lambda u}
\\
&
=
\textstyle\sum_{au}
\big\{
\alpha_{au}
\big[
\langle
\nabla K^{\xi a}
,
\beta^{\text{hor}}_\xi
\rangle
-
\langle
\nabla K^{\xi a}
,
\beta^{\text{hor}}_a
\rangle
\big]
-
\beta_{au}
\big[
\langle
\nabla K^{\xi a}
,
\alpha^{\text{hor}}_\xi
\rangle
-
\langle
\nabla K^{\xi a}
,
\alpha^{\text{hor}}_a
\rangle
\big]
\big\}
\,
\delta^{\lambda u}
\\
&
=\textstyle\sum_{au}\big\{\alpha_{au}\langle\nabla K^{\xi a},
S^{\xi a}(\beta)\rangle
-\beta_{au}\langle\nabla K^{\xi a},
S^{\xi a}(\alpha) \rangle \big\}\,\delta^{\lambda u}
= \sum_{au}\big\{ C^{\xi a}(\beta)\alpha_{au}
-C^{\xi a}(\alpha)\beta_{au}\big\} \,\delta^{\lambda u}.
\end{align*}
So if we define the matrix
$ H_{ia} := \sum_b \big[ C^{ab}(\beta)\alpha_{bi}
- C^{ab}(\alpha)\beta_{bi} \big]$, $i=1,\ldots,D$,
$a=1,\ldots,N$ we have:
$$ R_4 = -\frac{3}{4} \sum_{us}\sum_{\xi\eta}\sum_{\lambda\mu}
H_{u\xi}\,\delta^{\lambda u}\,(\mathbf{ K}^{-1})_{\xi\eta} \,
\delta_{\lambda\mu} \,H_{s\eta}\,\delta^{\mu s}
=-\frac{3}{4} \sum_{u=1}^D \sum_{\xi,\eta=1}^N
H_{u\xi} \, H_{u\eta} \, (\mathbf{ K}^{-1})_{\xi\eta}
= -\frac{3}{4} \|H\|^2_{\mathbf{ K}^{-1}}.
$$
Alternatively, this can be derived from formula \eqref{landmarkderdef}.
\end{proof}
The denominator~\eqref{denom} of sectional curvature for~$\mathcal{L}^N(\mathbb{R}^D)$ is given by the simple formula:
\begin{proposition} 
\label{PrDenLan}
For any pair of cotangent vectors~$\alpha,\beta\in T_q^\ast \mathcal{L}$,
\begin{equation}
\label{den_land}
\|\alpha\|^2_{T^\ast \mathcal{L}}\,
\|\beta\|_{T^\ast \mathcal{L}}^2
-\langle \alpha,\beta \rangle_{T^\ast \mathcal{L}}^2
=\sum_{abcd}
K^{ab}K^{cd}
\big(\langle\alpha_a,\alpha_b\rangle
\langle\beta_c,\beta_d\rangle
-\langle\alpha_a,\beta_b\rangle
\langle\alpha_c,\beta_d\rangle\big).
\end{equation}
\end{proposition}
\begin{proof}
Using double-index notation we may write 
equation~\eqref{denom} as follows:
\begin{align*}
\|\alpha\|^2_{T^\ast \mathcal{L}}\,
\|\beta\|_{T^\ast \mathcal{L}}^2
-\langle \alpha,\beta \rangle_{T^\ast \mathcal{L}}^2
&= \alpha_{au}\,\alpha_{bs}\,\beta_{cr}\,\beta_{dv}
\big(g^{(au)(bs)}g^{(cr)(dv)}-g^{(au)(dv)}g^{(bs)(cr)}\big) \\
&= \textstyle\sum_{abcd} \alpha_{au}\,\alpha_{bs}\,\beta_{cr}\,\beta_{dv}\big(K^{ab}\delta^{us} K^{cd}\delta^{rv} 
- K^{ad}\delta^{uv} K^{bc}\delta^{sr} \big) \\
& =\textstyle\sum_{abcd} \langle \alpha_a,\alpha_b \rangle
\langle \beta_c, \beta_d \rangle K^{ab} K^{cd} -
\textstyle\sum_{abcd}\langle \alpha_a, \beta_d \rangle 
\langle \alpha_b, \beta_d \rangle K^{ad} K^{bc},
\end{align*}
and~\eqref{den_land} follows by relabeling the indices.
\end{proof}

\subsection{The rotationally invariant case}
Finally, suppose the Green's function $K$ is rotationally invariant,
i.e.~that (K4) holds: 
\begin{equation*}
K(x) = \gamma(\|x\|),\,x\in\mathbb{R}^D,
\qquad\mathrm{with}\; \gamma \in C^2
\big([0,\infty)\big).
\end{equation*}
We will use the convenient notation: $\gamma_0:=\gamma(0)$, $\gamma_{ab}:=\gamma(\|q^{a}-q^b\|)$, $\gamma'_{ab}:=\gamma'(\|q^{a}-q^b\|)$,  and
$\gamma''_{ab}:=\gamma''(\|q^{a}-q^b\|)$ for $a,b=1,\ldots,N$. Then we can evaluate the first and second derivatives of $K$:
\begin{lemma}
For rotationally invariant kernels, it is the case that:
\begin{align}\label{nabKr}
\nabla K(x) &= \gamma'(\|x\|)\,\frac{x}{\|x\|}\,, \\
\label{hessKr}
D^2K(x) &= \Big[ \gamma''(\|x\|)-\frac{\gamma'(\|x\|)}{\|x\|} \Big]
\frac{xx^T}{\|x\|^2} + \frac{\gamma'(\|x\|)}{\|x\|}\,\mathbb{I}_D\\
\notag &= \gamma''(\|x\|)\frac{xx^T}{\|x\|^2} + \frac{\gamma'(\|x\|)}{\|x\|}\,\mathrm{Pr}^\perp(x), \end{align}  
where $\mathbb{I}_D$ is the~$D\times D$ identity matrix and $\mathrm{Pr}^\perp(x):=\mathbb{I}_D-\frac{xx^T}{\|x\|^2}$ is projection  to the hyperplane 
of~$\mathbb{R}^D$ that is normal to $x$.
\end{lemma}
\begin{proof}
We have that $\partial_iK(x) =\gamma'(\|x\|)\frac{x^i}{\|x\|}$ and \eqref{nabKr} follows immediately. Also,
\begin{align*}
\partial_j\partial_i
K(x) &= 
\textstyle \frac{x^i}{\|x\|} \frac{\partial}{\partial x^j} \gamma'(\|x\|)
+ \gamma'(\|x\|) \frac{1}{\|x\|} \frac{\partial}{\partial x^j} x^i
+ \gamma'(\|x\|)x^i \frac{\partial}{\partial x^j} \frac{1}{\|x\|} \\
& =
\textstyle \gamma''(\|x\|)\frac{x^ix^j}{\|x\|^2}
+ \frac{\gamma'(\|x\|)}{\|x\|} \delta^{ij}
- \gamma'(\|x\|) \frac{x^ix^j}{\|x\|^3}
= \big[ \gamma''(\|x\|)
-\frac{\gamma'(\|x\|)}{\|x\|} \big] \frac{x^ix^j}{\|x\|^2}
+ \frac{\gamma'(\|x\|)}{\|x\|} \delta^{ij},
\end{align*}
which implies~\eqref{hessKr}.
\end{proof}
Because of \eqref{nabKr}, in the rotationally invariant case, the ``scalar compression'' $C^{ab}(\alpha)$ really does measure a multiple compression of the flow $\alpha^\sharp$ between $q^a$ and $q^b$. We can decompose the vector strain $S^{ab}(\alpha)$ into the part parallel to the vector $q^a-q^b$ and the part perpendicular to this: let
$u^{ab}:=\frac{q^a-q^b}{\|q^a-q^b\|}$ and define
\begin{equation}
\label{S_decomp}
S^{ab}(\alpha)^{\parallel} := \big
\langle S^{ab}(\alpha), u^{ab} \big\rangle, \qquad 
S^{ab}(\alpha)^\perp := S^{ab}(\alpha) - S^{ab}(\alpha)^\parallel\,u^{ab}.
\end{equation}
Note that $S^{ab}(\alpha)^\parallel$ is a \em scalar \em while~$S^{ab}(\alpha)^\perp$ is a {\it vector\/}. In particular it is the case that $C^{ab}(\alpha) = \gamma'_{ab}\cdot S^{ab}(\alpha)^\parallel$. Moreover, formula \eqref{hessKr} allows us to simplify the first term $R_1$ in the curvature formula. Substituting \eqref{hessKr} into \eqref{gen_R1}, we get the rotationally invariant case for $R_1$:
\begin{proposition}
In the rotationally invariant case~{\rm(K4)} we have that
\begin{align} \label {R1rotinv}
R_1 =& \sum_{a \ne b}\Big( \frac{\gamma''_{ab}}{2} \big\langle S^{ab}(\alpha)^\parallel\, \beta_a - 
S^{ab}(\beta)^\parallel\, 
\alpha_a, S^{ab}(\alpha)^\parallel\, \beta_b - S^{ab}(\beta)^\parallel\, \alpha_b \big\rangle  
\\
\notag & 
\qquad\;+ \frac{\gamma'_{ab}}{2\|q^a-q^b\|} \big\langle S^{ab}(\alpha)^\perp
\otimes \beta_a - 
S^{ab}(\beta)^\perp\otimes \alpha_a, S^{ab}(\alpha)^\perp \otimes \beta_b - S^{ab}(\beta)^\perp \otimes\alpha_b \big\rangle \Big).
\end{align}
\end{proposition}
\noindent In the above we use the inner product of tensor products,
$\langle v_1\otimes w_1,v_2\otimes w_2\rangle
:=\langle v_1,v_2\rangle\langle w_1,w_2\rangle$.
\begin{proof}
For any pair of covectors~$\eta$ and~$\mu$ 
in~$T^\ast_q\mathcal{L}$,
by \eqref{hessKr} we have that:
\begin{align*}
S^{ab}(\eta)^TD^2\!K^{ab}\,S^{ab}(\mu)
&=
\gamma''_{ab}\,
S^{ab}(\eta)^Tu^{ab}\,(u^{ab})^TS^{ab}(\mu)
+\frac{\gamma'_{ab}}{\|q^a-q^b\|}
\,
S^{ab}(\eta)^T\,\mathrm{Pr}^\perp(u^{ab})\,S^{ab}(\mu)
\\
&=
\gamma''_{ab}\,
S^{ab}(\eta)^\parallel\,S^{ab}(\mu)^\parallel
+\frac{\gamma'_{ab}}{\|q^a-q^b\|}
\,
\big\langle
S^{ab}(\eta)^\perp,S^{ab}(\mu)^\perp
\big\rangle
.
\end{align*}
Inserting this expressions into~\eqref{gen_R1}
yields the desired result.
\end{proof}
\subsection{One landmark with nonzero momenta}
A simple special case is when only one landmark carries momentum. We now compute the numerator of sectional curvature when both cotangent vectors are nonzero at only \em one \em of the $D$-dimensional landmarks $(q^1,\ldots,q^N)$. We define:
$$(T^\ast_q\mathcal{L})_{1}:=
\big\{\eta\in T^\ast_q\mathcal{L}\,\big|\,\eta_a=0\mbox{ for }a>1\big\}$$
so that the elements of the above set are cotangent vectors of the type
$\eta=(\eta_1,0,\ldots,0)$.
\begin{proposition}
\label{one_momentum}
In $\mathcal{L}^N(\mathbb{R}^D)$, for any pair
$\alpha,\beta\in (T^\ast_q\mathcal{L})_{1}$
the four terms of~$R(\alpha^\sharp,\beta^\sharp,
\beta^\sharp,\alpha^\sharp)$ are given by $R_1=R_2=R_3=0$ and $R_4  = -\frac{3}{4}\sum_{a,b=2}^N\langle H^a,H^b\rangle_{\mathbb{R}^D}
(\mathbf{ K}^{-1})_{ab}$,
where
$$H^a:=(\gamma_{a1}-\gamma_0)\big(\langle\alpha_1,\nabla K^{a1}\rangle\beta_1
-\langle\beta_1,\nabla K^{a1}\rangle\alpha_1\big),
\qquad\mbox{ for }a>1.$$
\end{proposition}
\begin{proof}
The vanishing of~$R_1$ can be checked directly
(note that the sum in~\eqref{gen_R1} is taken over $a\not= b$ since~$S^{cc}(\eta)=0$ for all~$c$ and $\eta$).  
Also, using formula \eqref{Fdef} we see that \em all \em mixed forces $F_a$ are zero, therefore 
$R_2=R_3=0$ by formulae~\eqref{gen_R2} and~\eqref{gen_R3}. Also, by~\eqref{landmarkderdef},
$D^a(\alpha,\beta)
=
(\gamma_{a1}-\gamz)
\big\langle\alpha_1, \nabla K^{a1}\big\rangle\beta_1
$ since~$\alpha,\beta\in(T_q^\ast\mathcal{L})_1$;
a similar expression holds for~$D^a(\beta,\alpha)$,
which concludes the proof by~\eqref{gen_R4}.
\end{proof}
Therefore when~$\alpha,\beta\in(T^\ast_q\mathcal{L})_{1}$ 
the sectional curvature is \em always \em negative; we can understand this by considering the geodesic flow in this case. It follows immediately from Proposition~\ref{prop_ham} 
 that if we start with zero momenta $p_a$ at all $q^a, a>1$, then the momenta at these points stay zero, while the momentum at $q^1$ remains constant. Thus the velocity of~$q^1$ is just given by~$
K(0)p_1$ and this is constant. The point $q^1$ carrying the momentum moves in a straight line at constant speed,
 while the other points~$q^a$ ($a>1$) are carried along 
 by the global flow that the motion of $q^1$ causes and
 move at speeds~$\dot{q}^a=K^{a1}p_1$,
 which are parallel to~$\dot{q}^1$ (but not constant). 
As shown in Figure~\ref{fig_dragging} (the central landmark~$q^1$ is the only one carrying momentum) what happens is that all other landmark points are dragged along by $q^1$, more strongly when close, less when far away. Points directly in front of the path of $q^1$ pile up and points behind space out.
\begin{figure}[t]
\center{
\includegraphics[height=7.2cm]{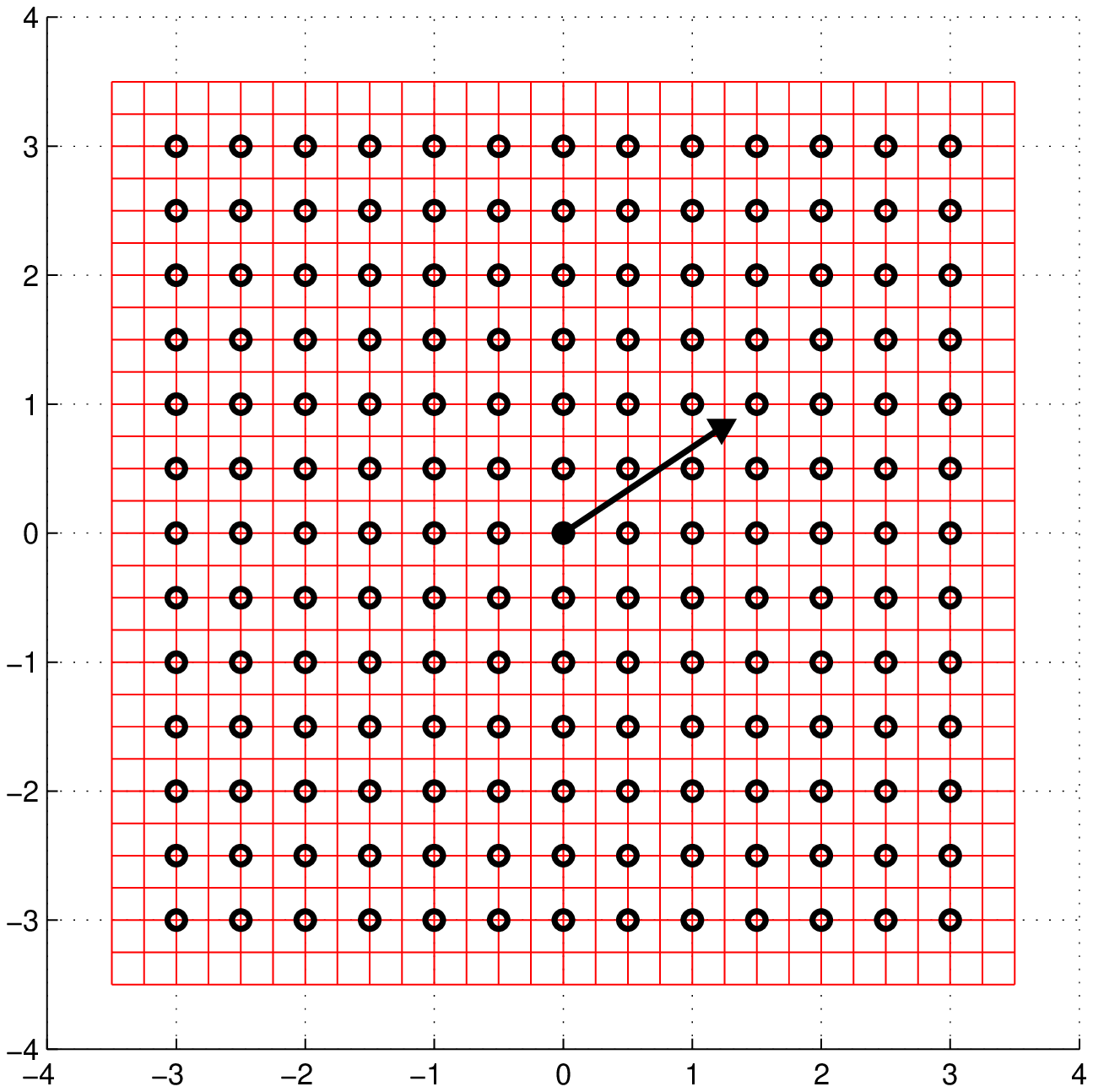}
\includegraphics[height=7.2cm]{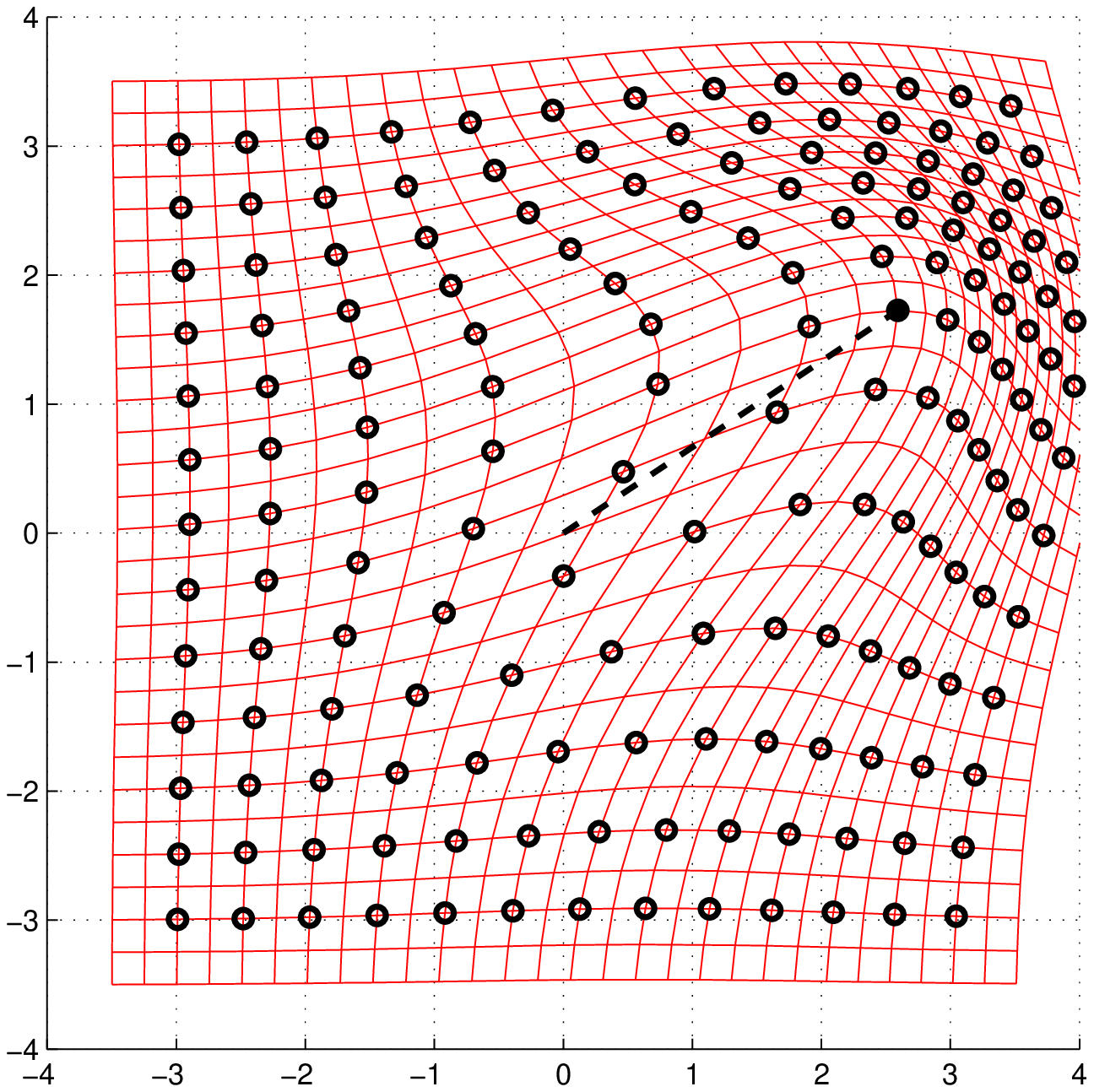}
}
\caption{
Dragging effect of one momentum-carrying landmark~$q^1$
(bullet~$\bullet$) on a grid of landmarks
(circles~$\circ$), with $\gamma(x)=\exp(-\frac{1}{2}\frac{x^2}{\sigma^2})$, $\sigma=1.5$. 
Left: initial configuration, with initial 
momentum $p_1=(2.7,1.8)$ also shown.
Right: configuration after one unit of time, with 
trajectory of~$q^1$ also shown;
the grid represents the diffeomorphism~$\varphi_{01}^v$,
obtained by integrating~$\alpha^{\mathrm{hor}}$ in time.
}\label{fig_dragging}
\end{figure} 
\par
Negative curvature can be seen by the divergence of
geodesics. If you imagine slightly changing the direction
of~$p_1$ in Figure~\ref{fig_dragging}, the final configuration of 
the landmark points (say, after one unit of time) will differ greatly from
the one caused by the original value of~$p_1$.
Also, if you imagine $q^1$ moving along two nearby parallel straight lines, the differential effect on the cloud of other points accumulates so that the final configurations will differ everywhere; thus, even though the initial landmark configurations are close, the final configurations will be far away. In general, the last negative term in the curvature expresses the same effect: the global drag effect of each point results in a kind of turbulent mixing of all the other points (think of a kitchen mixer the motion of whose blades mixes the whole bowl).
\par
Proposition~\ref{one_momentum} simplifies in the case of~$\mathcal{L}=\mathcal{L}^2(\mathbb{R}^D)$ 
(two landmarks only). We shall write:
\begin{align}
\label{ab_decomp}
\alpha_1^\parallel&:=
\langle\alpha_1,u^{12}\rangle,
&
\alpha_1^\perp&:=\alpha_1-\alpha_1^\parallel\, u^{12},
&\beta_1^\parallel&:=
\langle\beta_1,u^{12}\rangle,
&
\beta_1^\perp&:=\beta_1-\beta_1^\parallel\, u^{12}.
\end{align}
\begin{proposition}
In the case of~$\mathcal{L}=\mathcal{L}^2(\mathbb{R}^D)$, when
$\alpha,\beta\in(T^\ast_q\mathcal{L})_{1}$  the 
numerator and denominator of sectional curvature 
are given by, respectively:
\begin{align}
\label{onemom_num}
R(\alpha^\sharp,\beta^\sharp,\beta^\sharp,\alpha^\sharp)=
R_4&=-\frac{3}{4}\,\gamma_0\,
\frac{\gamma_0-\gamma_{12}}{\gamma_0+\gamma_{12}}\,
\big(\gamma'_{12}\big)^2\,
\big\|
\beta_1^\parallel\alpha_1^\perp
-
\alpha_1^\parallel\beta_1^\perp
\big\|^2,
\\
\label{onemom_den}
\|\alpha\|^2_{T^\ast \mathcal{L}}\,
\|\beta\|_{T^\ast \mathcal{L}}^2
-\langle \alpha,\beta \rangle_{T^\ast \mathcal{L}}^2
&=\gamma_0^2
\Big(
\big\|
\beta_1^\parallel\alpha_1^\perp
-
\alpha_1^\parallel\beta_1^\perp
\big\|^2
+
{\textstyle\frac{1}{2}}
\big\|
\beta_1^\perp\otimes\alpha_1^\perp
-
\alpha_1^\perp\otimes\beta_1^\perp
\big\|^2
\Big).
\end{align}
\end{proposition}
\noindent{Again} we have used the inner product of tensor products
$\langle v_1\otimes w_1,v_2\otimes w_2\rangle
:=\langle v_1,v_2\rangle\langle w_1,w_2\rangle$.
\begin{proof}
It is the case that $R_4=-\frac{3}{4}\|H^2\|^2(\mathbf{ K}^{-1})_{22}$
(matrices~$H^a$ were defined in Proposition~\ref{one_momentum}).
But~$(\mathbf{ K}^{-1})_{22}=(\gamz^2-\gam^2)^{-1}\gamz$, whereas
from Proposition~\ref{one_momentum} we have
$$
\|H^2\|^2=(\gamz-\gam)^2
\big(
\langle\alpha_1,\nabla K^{12}\rangle^2
\|\beta_1\|^2
+
\langle\beta_1,\nabla K^{12}\rangle^2
\|\alpha_1\|^2
-2
\langle\alpha_1,\nabla K^{12}\rangle
\langle\beta_1,\nabla K^{12}\rangle
\langle\alpha_1,\beta_1\rangle
\big),
$$
where~$\nabla K^{12}=\gamp u^{12}$
by~\eqref{nabKr}.
Inserting expressions~\eqref{ab_decomp} into the above formula yields~\eqref{onemom_num}.
From Proposition~\ref{PrDenLan} we have that the denominator
is given by~$\gamz^2
(
\|\alpha_1\|^2
\|\beta_1\|^2
-\langle\alpha_1,\beta_1\rangle^2
)
$; again, inserting~\eqref{ab_decomp} into such formula
yields~\eqref{onemom_den}.
\end{proof}
We will generalize the above results in the next section.
\section{Landmark geometry with two nonzero momenta}
The complexity of the formula for curvature reflects a real complexity in the geometry of the landmark space. But there is one case in which the geometry such space can be analyzed quite completely. This is when there are only two nonzero momenta along a geodesic. To put this in context, we first introduce a basic structural relation between landmark spaces.

\subsection{Submersions between landmark spaces}
Instead of labeling the landmarks as $1,2,\cdots,N$, one can use any finite index set $\mathcal A$ and label the landmarks as $q^a$ with $a \in \mathcal A$. And instead of calling the landmark space $\mathcal L^N$, we can call it $\mathcal L^{\mathcal A}$. Now suppose we have a subset $\mathcal B \subset \mathcal A$. Then there is a natural projection $\pi: \mathcal L^{\mathcal A} \rightarrow \mathcal L^{\mathcal B}$ gotten by forgetting about the points with labels in $\mathcal A - \mathcal B$. In the metrics we have been discussing {\it this is a submersion}. In fact, the kernel of $d\pi$, the vertical subspace of $T\mathcal{L}^{\mathcal A}$, is the space of vectors $v^a$ such that $v^a=0$ if $a \in \mathcal B$. Its perpendicular in $T^\ast$ is:
$$
(T^\ast\mathcal{L}^{\mathcal A})_{\mathcal B} := \big\{
p\in T^\ast\mathcal{L}^{\mathcal A}\,\big| p_a=0\mbox{ for }a \in \mathcal A - \mathcal B \big\}
$$
so the orthogonal complement of ker($d\pi$) in $T\mathcal L^{\mathcal A}$ is the space of vectors $p^\sharp$ where $p$ is in $(T^\ast\mathcal{L}^{\mathcal A})_{\mathcal B}$. On this subspace, the norm is just 
$$\sum_{b,b' \in \mathcal B}K(q^b - q^{b'})\langle p_b, p_{b'}\rangle$$
whether $p^\sharp$ is taken to be a tangent vector to $\mathcal A$ or to $\mathcal B$. In other words, the horizontal subspace for the submersion $\pi$ is the subbundle $(T^\ast \mathcal L^{\mathcal A})_{\mathcal B}^\sharp \subset T\mathcal L^{\mathcal A}$ of tangent vectors $p^\sharp$ where $p$ has zero components in $\mathcal A - \mathcal B$ and this has the same metric as the tangent space to $\mathcal L^{\mathcal B}$. In particular, from the general theory of submersions, we know that every geodesic in $\mathcal L^{\mathcal B}$ beginning at some point~$\pi(\{q^a\})$ has a unique lift to a horizontal geodesic in $\mathcal L^{\mathcal A}$ starting at $\{q^a\}$. The picture to have is that all the landmark spaces form a sort of inverse system of spaces whose inverse limit is the group of diffeomorpisms of $\mathbb R^D$. 

We don't want to pursue this is in general, but rather we will study the special case where the cardinality of $\mathcal B$ is two. We might as well, then, go back to our former terminology and consider the map $\pi: \mathcal L^N \rightarrow \mathcal L^2$ gotten by mapping an $N$-tuple $(q^1,q^2,\cdots, q^N)$ to the pair $(q^1,q^2)$. Moreover, we want to consider only the case in which the kernel $K$ is \em rotationally invariant \em as in~(K4).
A basic quantity in all that follows is the distance 
$\rho:=\|q^1-q^2\|$ between the two momentum bearing points. 

\subsection{Two momentum geodesics}
Remarkably, we can describe, more or less explicitly, all the geodesics which arise as horizontal lifts from this map. These are the geodesics with nonzero momenta only at $q^1$ and $q^2$. Moreover, the formula for sectional curvature for the 2-plane spanned by any two horizontal vectors can be analyzed. This analysis was started in the PhD thesis of the first author \cite{micheli.phd} and has been pursued further in \cite{McLMars}.
\par
The metric tensor of~$\mathcal{L}=\mathcal{L}^2(\mathbb R^D)$ in 
coordinates is obtained by inverting the $2\times 2$ matrix {\bf K}: 
\begin{equation}
\label{invK}
\mathbf{ K}=\left[\begin{array}{cc}\gamma_0 & \gamma(\rho) \\
\gamma(\rho) & \gamma_0 \end{array}\right]
\;\;
\Longrightarrow
\;\;
\left\{\begin{array}{l}
(\mathbf{ K}^{-1})_{11}  =  (\mathbf{ K}^{-1})_{22}  =  
(\gamma_0^2-\gamma(\rho)^2)^{-1} \gamma_0
\\
(\mathbf{ K}^{-1})_{12}  =  (\mathbf{ K}^{-1})_{21}  =  
-(\gamma_0^2-\gamma(\rho)^2)^{-1} \gamma(\rho)
\end{array}\right.,
\end{equation}
so that the cometric and metric, 
for all covectors~$\alpha,\beta\in T^\ast_q\mathcal{L}$
and vectors~$v,w\in T_q\mathcal{L}$,
are simply:
\begin{align}
\label{2pt_com_1} 
g^{-1}(\alpha, \beta) &=  
\gamma_0
\big(
\langle
\alpha_1
,
\beta_1
\rangle 
+
\langle
\alpha_2
,
\beta_2
\rangle 
\big)
+ 
\gamma(\rho) \big(
\langle
\alpha_1
,
\beta_2
\rangle 
+
\langle
\alpha_2
,
\beta_1
\rangle 
\big)
, \\
\nonumber
g(v,w) &= 
\frac{1}{\gamma_0^2-\gamma(\rho)^2} 
\Big[
\,
\gamma_0
\big(
\langle
v^1
,
w^1
\rangle 
+
\langle
v^2
,
w^2
\rangle 
\big)
- 
\gamma(\rho) \big(
\langle
v^1
,
w^2
\rangle 
+
\langle
v^2
,
w^1
\rangle 
\big)\Big]
.
\end{align}
\par
The geometry of the two-point space is best understood by changing variables for the landmark coordinates~$(q^1, q^2)$
and the momentum~$(p_1, p_2)$ to their means and semi-differences, that is:
\begin{align*}
\overline q  :=&\; \frac{q^1+q^2}{2} ,
& 
\hspace*{-.5cm} 
\delta q :=&\; \frac{q^1-q^2}{2}, 
&&&
\overline p :=&\; \frac{p_1+p_2}{2} ,
& 
\hspace*{-.5cm} 
\delta p :=&\; \frac{p_1-p_2}{2} ,
\\
\mbox{so that: }\qquad
q^1 =&\; \overline q + \delta q,
& 
\hspace*{-.5cm} 
q^2 =&\; \overline q - \delta q,
&&&
p_1 =&\; \overline p + \delta p,
& 
\hspace*{-.5cm} 
p_2 =&\; \overline p - \delta p.
\end{align*} 
Then the cometric~\eqref{2pt_com_1} becomes:
\begin{align} \label{2ptcometric} 
g^{-1}\big((\overline \alpha, \delta \alpha), (\overline \beta, 
\delta \beta)\big)  &=  
2\big(\gamma_0+\gamma(\rho)\big)
\,
\big\langle
\overline{\alpha}
,
\overline{\beta}
\big\rangle
+ 
2\big(\gamma_0 - \gamma(\rho) \big)
\,
\big\langle
\delta{\alpha}
,
\delta{\beta}
\big\rangle. 
\end{align}
With these coordinates, the two-point landmark space becomes a product $\overline V \times V_\delta$ in which all fibers $\overline V \times \{\delta q_0\}$ are flat Euclidean spaces though with variable scales, all fibers $\{\overline q_0\}\times V_\delta$ are {\it conformally} flat metrics sitting on the manifold $\mathbb R^D - \{0\}$ and the tangent spaces of the two factors are orthogonal. 
\begin{proposition}
In terms of means and semi-differences, the geodesic equations for 
$\mathcal{L}^2(\mathbb R^D)$ are:
\begin{equation}
\label{2ptgeodeq}
\begin{aligned}
\dot{\overline q} &= 
\big(\gamma_0+\gamma(\rho)\big)\,\overline p,
& &\hspace*{2cm}&
\dot{\overline p} &= 0,
\\
\dot{\delta q} &= \big(\gamma_0-\gamma(\rho)\big)\,\delta p,
& & &
\dot{\delta p} &= - 2\frac{\gamma'(\rho)}{\rho}
\big(\|\overline p\|^2 - \|\delta p\|^2\big)\,\delta q.
\end{aligned}
\end{equation}
\end{proposition}
The above result is proven by direct computation. We can solve these equations in {\bf four steps}. 
\par
{\bf 1.} First the {\it linear momentum} $\overline p$ is a constant, so 
``center of mass'' 
$\overline q$ moves in a straight line parallel to this constant:
\begin{equation} \label{solveq} 
\overline{q}(t) =\overline q(0) + \Big(\int_0^t \big(\gamma_0+\gamma(\rho(\tau))\big)d\tau\Big)\, \overline{p}.\end{equation}
\par
{\bf 2.} Secondly, if we treat
vectors~$\delta q$ and~$\delta p$
as 1-forms in~$\mathbb{R}^D$,
equations~\eqref{2ptgeodeq} also show that:
$$ (\delta q \wedge \delta p)^\centerdot = \dot{\delta q}\wedge \delta p + \delta q \wedge \dot{\delta p} = [\text{(scalar)}\,\delta p] \wedge \delta p + \delta q \wedge [\text{(scalar)}\, \delta q] = 0,$$
so the {\it angular momentum} 2-form $\delta q \wedge \delta p
\in \bigwedge^2\mathbb{R}^D$ is {\it constant\/}; we write this as $\omega\, e^1 \wedge e^2$ where $\omega$ is the nonnegative real magnitude of the angular momentum and $(e^1, e^2)$ is an orthonormal pair. Then it follows that:
$$ \delta q(t) = \tfrac12 \rho(t)\big[\cos\big(\theta(t)\big)e^1 + \sin\big(\theta(t)\big)e^2\big],
\quad\text{for some function }\theta(t).$$
\par
{\bf 3.} Thirdly, we can express $\theta(t)$ as an integral:
\begin{align*}
\dot{\delta q} &= \tfrac12 \dot{\rho}\big[
\cos(\theta)e^1+\sin(\theta)e^2\big] + 
\tfrac12 \rho \dot{\theta} \big[
-\sin(\theta)e^1+\cos(\theta)e^2\big], \quad \text{so} \\
\dot{\delta q} \wedge \delta q &= -\tfrac14 \rho^2 \dot{\theta} \, e^1 \wedge e^2, \quad \text{as well as
(from~\eqref{2ptgeodeq}):} \\
\dot{\delta q} \wedge \delta q &= 
\big(\gamma_0-\gamma(\rho)\big) \delta p \wedge \delta q = 
-\omega\,\big(\gamma_0-\gamma(\rho)\big)\, e^1 \wedge e^2 ;
\end{align*}
combining the second and third lines, we find:
\begin{equation} \label{solvetheta} 
\theta(t) = \theta(0) + 4\omega \int_0^t \frac{\gamma_0-\gamma(\rho(\tau))}{\rho(\tau)^2} \,d\tau; \end{equation}
note that by (K1) and (K2) it is the case that  $\gamma_0\geq \gamma(\rho)$ 
for any~$\rho\geq0$, so~$\theta$ is a \em monotone increasing \em function if~$\omega\not=0$, otherwise it is a constant.
\par
{\bf 4.} The last step is to solve for $\rho(t)$. This can be done using conservation of energy~\cite[p.~51]{jost}. Equations~(\ref{2ptgeodeq}) are in fact the cogeodesic  equations for the Hamiltonian~$\mathcal{H}(p,q)$
of section~\ref{ham_form},
which we may rewrite in terms of means and semi-differences
as
$$ \mathcal{H}
=\big(\gamma_0+\gamma(\rho)\big)\|\overline p\|^2 + 
\big(\gamma_0-\gamma(\rho)\big)\|\delta p\|^2$$
by~\eqref{2ptcometric};
hence this function of $\rho$ and $\|\delta p\|$ is a constant
($\overline{p}$ is also a constant). Then we calculate:
\begin{align*}
(\rho^2)^\centerdot &= 4\langle \delta q, \delta q\rangle^\centerdot = 8\langle \dot{\delta q}, \delta q \rangle = 8
\big(\gamma_0-\gamma(\rho)\big)\langle \delta p, \delta q\rangle \quad \Longrightarrow
\quad
\dot{\rho} 
= 4\frac{\gamma_0-\gamma(\rho)}{\rho} 
\langle \delta p, \delta q\rangle.
\end{align*}
But:
\begin{align*}
\quad \langle \delta p, \delta q\rangle^2 + \omega^2
&= \langle \delta p, \delta q\rangle^2 + \|\delta p \wedge \delta q\|^2 = \|\delta p\|^2\cdot \|\delta q\|^2 = \frac{\rho^2}{4} \bigg( \frac{
\mathcal{H}-(\gamma_0+\gamma(\rho))\|\overline{p}\|^2}{\gamma_0-\gamma(\rho)} \bigg),\\
\Longrightarrow
\quad \dot{\rho} &=2\frac{\sqrt{\gamma_0-\gamma(\rho)}}{\rho} \sqrt{ \rho^2 \big[\mathcal{H}
-\big(\gamma_0+\gamma(\rho)\big)\|\overline{p}\|^2\big] - 4 \omega^2\big(\gamma_0-\gamma(\rho)\big)}.
\end{align*}
This means that the function $\rho(t)$ is the solution of:
\begin{align} 
\label{solverho}
t &= \int_{\rho(0)}^{\rho(t)} \frac{x\, dx}{2\sqrt{F(x)}},\quad \mbox{where:}\quad 
F(x) := \mathcal{H}\,x^2\big(\gamma_0-\gamma(x)\big)-\|\overline p\|^2 x^2\big(\gamma_0^2-\gamma(x)^2\big) - 4\omega^2 \big(\gamma_0-\gamma(x)\big)^2.
\end{align}
\begin{summary}
If we fix constants $\mathcal{H}$, $\overline p$, $\omega$, $\rho(0)$, $\theta(0)$, $q^a(0)$ (for all $a$), we can first integrate (\ref{solverho}) to get $\rho(t)$ (the separation of $q^1$ and $q^2$), then integrate (\ref{solvetheta}) to find their relative angle $\theta(t)$, then integrate~(\ref{solveq}) to get their center of mass $\overline q(t)$. This gives the trajectories of $q^1$ and $q^2$. The remaining points are dragged along as solutions of:
$ \tfrac{d}{dt}q^a(t) = \gamma
\big(\|q^a(t)-q^1(t)\|\big)p_1(t) + \gamma\big(\|q^a(t)-q^2(t)\|\big)
p_2(t).$
\end{summary}
As worked out in \cite{McLMars}, one can classify the global behavior of these geodesics into two types. One is the scattering type in which $q^1, q^2$ diverge from each other as time goes to either $\pm \infty$. This occurs if the linear or angular momentum is large enough compared to the energy. In the other case where the energy is large enough compared to both momenta, they come together asymptotically at either $t=+\infty$ or $-\infty$, diverging at the other limit. In both cases, they may spiral around each other an arbitrarily large number of times (see Figure~\ref{fig_2Dtraj}).
\begin{figure}[t]
\center{
\includegraphics[height=6.5cm]{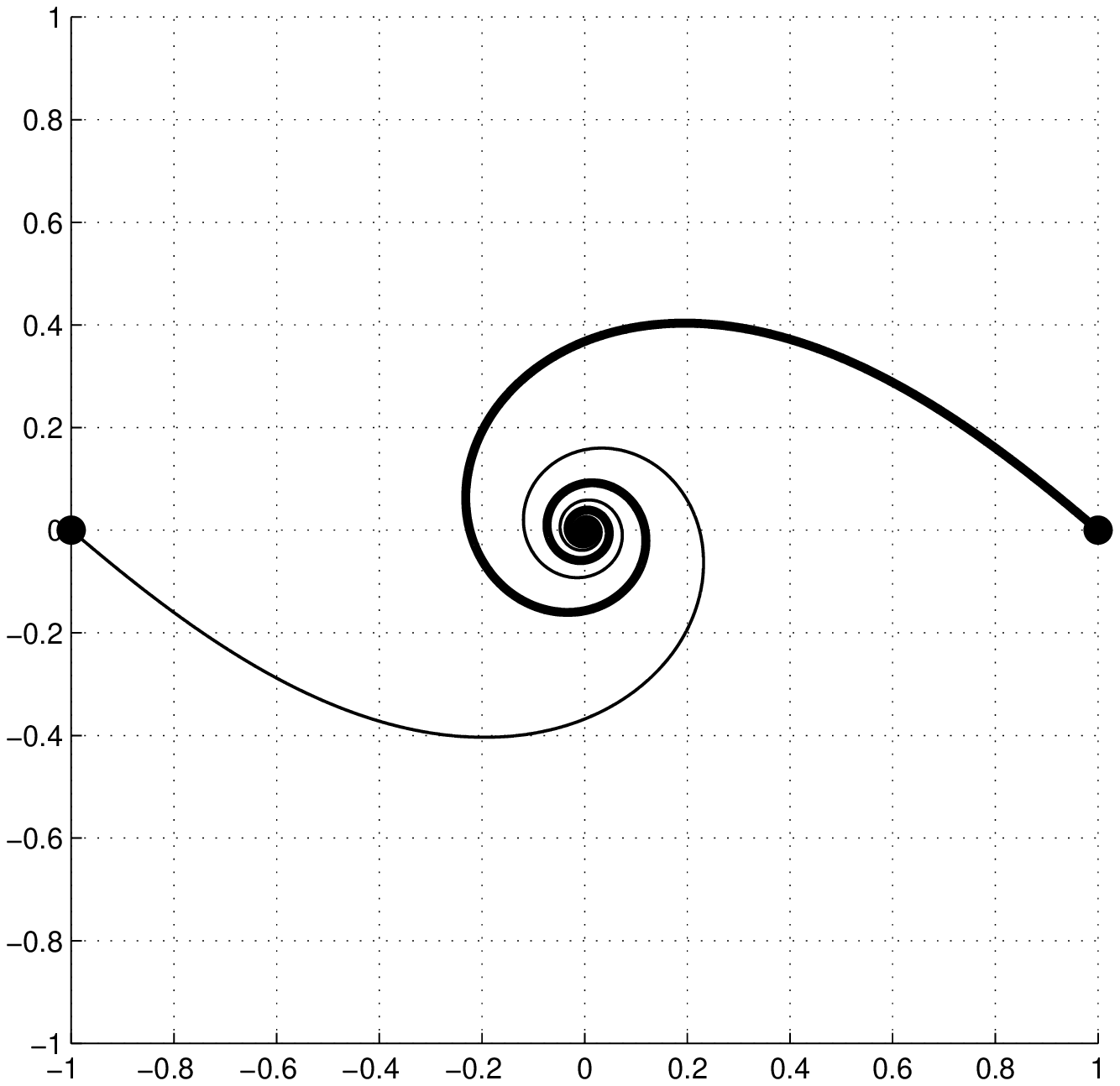}
\includegraphics[height=6.5cm]{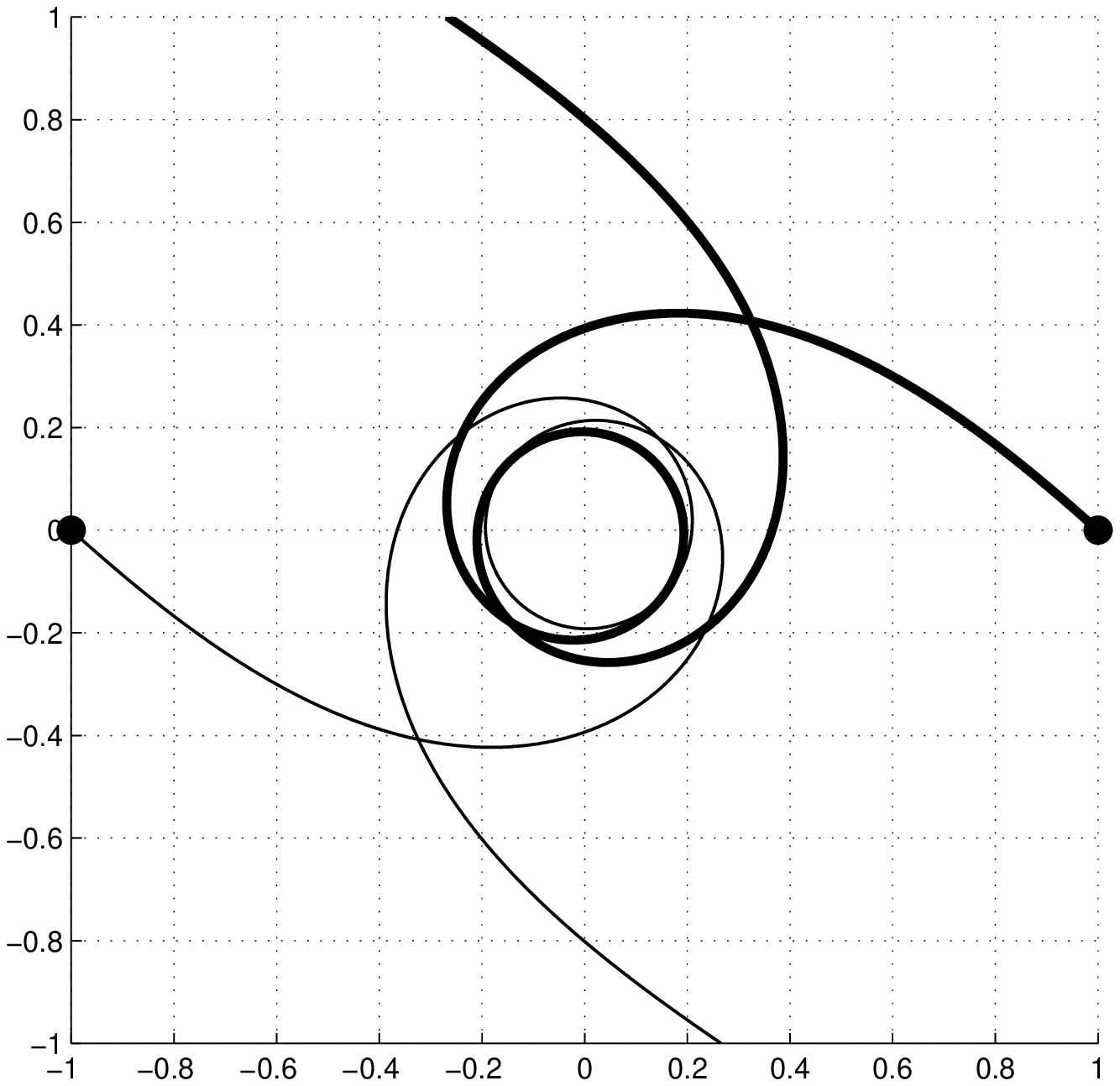}
}
\caption{Converging and diverging trajectories for two landmarks in two dimensions. In these examples $\gamma(x)=
\exp(-\frac{1}{2}x^2)$, $(q^1(0),q^2(0))=((1,0),(-1,0))$, 
$(p_1(0),p_2(0))=((-10, 8.6),(10,-8.6))$ for the graph on the left,
$(p_1(0),p_2(0))=((-10, 9),(10,-9))$ for the graph on the right.
The thick and thin lines are, respectively, the 
trajectories of the first and second landmarks.
}\label{fig_2Dtraj}
\end{figure} 
\subsection{Decomposing curvature}
Next we consider~$\mathcal{L}^N(\mathbb{R}^D)$: we want to compute the sectional curvature $R(\alpha^\sharp,\beta^\sharp,\beta^\sharp,\alpha^\sharp)$ for cotangent vectors that are nonzero at only $(q^1,q^2)$. 
Also, we will use the notation $u:=\frac{q^1-q^2}{\|q^1-q^2\|} $ for the unit vector from $q^2$ to $q^1$ as well as $\rho = \|q^1-q^2\|$ for their distance. Similarly to~\eqref{S_decomp}, we will also want to decompose any vector in $\eta\in\mathbb R^D$ into its parts tangent to $u$ and perpendicular to $u$:
$$\eta^\parallel:=\langle\eta, u\rangle, \quad
\mathrm{and}\quad 
\eta^\bot := \eta - \eta^\parallel\, u.$$
Once again note that~$\eta^\parallel$ is a scalar 
whereas~$\eta^\perp$ is a vector.
Following the notation used to describe geodesics above, for any $\alpha \in (T^\ast_q\mathcal{L})_{1,2}
:=
\big\{\eta\in T^\ast_q\mathcal{L}\,\big|\,\eta_a=0\mbox{ for }a>2\big\}
$, 
we write $\overline \alpha = \tfrac12 (\alpha_1+\alpha_2)$ and $\delta \alpha = \tfrac12 (\alpha_1-\alpha_2)$. 

\begin{proposition}
\label{ThCurvLM2}
In $\mathcal{L}^N(\mathbb{R}^D)$ for any pair~$\alpha,\beta\in (T^\ast_q\mathcal{L})_{1,2}$,
the terms $R_1, R_2$ and $R_3$ in the numerator of sectional 
curvature 
can be written as
\begin{align*}
R_1 &= 4\big(\gamma_0-\gamma(\rho)\big)^2 \gamma''(\rho)
\,\big\langle 
\delta\alpha^\parallel \beta_1 - \delta\beta^\parallel \alpha_1, 
\delta\alpha^\parallel \beta_2 - \delta\beta^\parallel \alpha_2 
\big \rangle \\ 
&\;\;\;\;+ 4\big(\gamma_0-\gamma(\rho)\big)^2\, \frac{\gamma'(\rho)}{\rho}\, \big\langle
\delta\alpha^\bot\otimes \beta_1 - 
\delta\beta^\bot \otimes \alpha_1, 
\delta\alpha^\bot\otimes \beta_2 - 
\delta\beta^\bot \otimes \alpha_2 \big \rangle, \\
R_2 &= -4\big(\gamma_0-\gamma(\rho)\big)\,
\gamma'(\rho)^2 \,\big\langle 
\delta\alpha^\parallel \beta_1 - \delta\beta^\parallel \alpha_1, 
\delta\alpha^\parallel \beta_2 - \delta\beta^\parallel \alpha_2 
\big \rangle, \\
R_3 & = \frac{\gamma_0-\gamma(\rho)}{2}\, \gamma'(\rho)^2
\big[ \big(\langle\alpha_1,\beta_2\rangle + \langle\beta_1,\alpha_2\rangle\big)^2
-4\langle\alpha_1,\alpha_2\rangle\langle\beta_1,\beta_2\rangle\big].
\end{align*}
\end{proposition}
\noindent 
We need the 
following result.
\begin{lemma}
\label{Le:FC}
For any $\alpha\in (T^\ast_q\mathcal{L})_{1,2}$, the discrete strain $S^{12}(\alpha)$ is given by:
\begin{equation} \label{2ptdiscretestrain}
S^{12}(\alpha) = 2\big(\gamma_0-\gamma(\rho)\big)
\delta\alpha.
\end{equation}
For any pair $\alpha,\beta\in (T^\ast_q\mathcal{L})_{1,2}$ it is the case that $F_a(\alpha,\beta)=0$ for $a>2$, whereas
\begin{align}
\label{F12}
F_1(\alpha,\beta)&= -F_2(\alpha,\beta)=
\frac{\gamma'(\rho)}{2}
\big(\langle\alpha_1,\beta_2\rangle + \langle\beta_1,\alpha_2\rangle
\big)u.
\end{align}
Also,  
\begin{equation} D^1(\alpha,\beta)= 
2\big(\gamma_0-\gamma(\rho)\big)\,
\gamma'(\rho)\, \delta\alpha^\parallel\,\beta_2, \qquad
D^2(\alpha,\beta)= 2\big(\gamma_0-\gamma(\rho)\big)\,
\gamma'(\rho)\,  \delta\alpha^\parallel\,
\beta_1.\label{D12}
\end{equation}
\end{lemma}
\begin{remark} We are not interested in~$D^a(\alpha,\beta)$ for~$a>2$ since the terms
in formula~\eqref{gen_R2} where they appear are zero 
(because~$F_a(\alpha,\beta)=0$ for~$a>2$).
\end{remark}
\begin{proof}[Proof of Lemma~\ref{Le:FC}.] The formula for the discrete strain results from: 
\begin{align*}
S^{12}(\alpha) =(\alpha^\sharp)^1 - 
(\alpha^\sharp)^2 =
\textstyle\sum_b (K^{1b}-K^{2b})\alpha_b 
= \gamma_0\alpha_1 + \gamma(\rho)\alpha_2 - \gamma(\rho)\alpha_1 - \gamma_0\alpha_2 = 2\big(\gamma_0-\gamma(\rho)\big)
\delta\alpha.
\end{align*}
The values for $F$ follow immediately from formula \eqref{Fdef}
and~$\nabla K^{12}=\gamma'(\rho)\,u$. 
Note that: 
$$ C^{12}(\alpha)=C^{21}(\alpha)=\big\langle S^{12}(\alpha),
\nabla K^{12}\big\rangle = \big\langle 2\big(\gamma_0-\gamma(\rho)\big)
\delta\alpha,\gamma'(\rho) u \big\rangle
= 2\big(\gamma_0-\gamma(\rho)\big)\gamma'(\rho)\,\delta\alpha^\parallel, 
\mbox{ so}
$$
$D^1(\alpha,\beta)
\!=\!C^{12}(\alpha)\beta_2
\!=\!2(\gamma_0\!-\!\gamma(\rho))\gamma'(\rho) \delta\alpha^\parallel\beta_2$
and 
$D^2(\alpha,\beta)
\!=\!C^{21}(\alpha)\beta_1
\!=\!2(\gamma_0\!-\!\gamma(\rho))\gamma'(\rho) \delta\alpha^\parallel\beta_1$.
\end{proof}
\begin{proof}[Proof of Proposition~\ref{ThCurvLM2}] The $R_1$ expression follows by substituting the expressions in \eqref{2ptdiscretestrain} into formula \eqref{R1rotinv}, noting that the only non-zero terms in the latter are for
$(a,b)=(1,2)$ and
$(a,b)=(2,1)$.
\par
By Theorem~\ref{ThCurvLM} and
the fact that~$F_2=-F_1$ from
Lemma~\ref{Le:FC}, $R_2$ is given by
\begin{align}\nonumber
R_2=&\,\big\langle D^1(\alpha,\alpha)-D^2(\alpha,\alpha),
F_1(\beta,\beta)\big\rangle
+\big\langle D^1(\beta,\beta)-D^2(\beta,\beta),
F_1(\alpha,\alpha)\big\rangle \\
&-\big\langle 
D^1(\alpha,\beta)
-D^2(\alpha,\beta)
+ D^1(\beta,\alpha)
-D^2(\beta,\alpha)
,
F_1(\alpha,\beta)\big\rangle.\label{R2RHS}
\end{align}
Again by Lemma~\ref{Le:FC} we have that
$D^1(\eta,\zeta)-D^2(\eta,\zeta)=
-4\big(\gamma_0-\gamma(\rho)\big)\gamma'(\rho)\, \delta\eta^\parallel\, \delta\zeta$ for any 
pair $\eta, \zeta \in (T^\ast_q\mathcal{L})_{1,2}$,
while 
$
F_1(\eta,\zeta)=\frac{1}{2}\gamma'(\rho)
\big(\langle\eta_1,\zeta_2\big\rangle
+
\langle\eta_2,\zeta_1\rangle\big) u.
$ 
Applying this to all the terms we get the expression for $R_2$ in 
the statement of the proposition. 

As far as $R_3$ is concerned, by Theorem~\ref{ThCurvLM}:
\begin{align*}
R_3 & = \gamma_0\big[\langle F_1(\alpha,\beta),\! F_1(\alpha,\beta)
\rangle\!-\!\langle F_1(\alpha,\alpha),\! F_1(\beta,\beta)\rangle\big]
\!+\! \gamma(\rho)\big[\langle F_1(\alpha,\beta),\! F_2(\alpha,\beta)
\rangle-\!\langle F_1(\alpha,\alpha)\!,\! F_2(\beta,\beta)\rangle\big]
\\
&\;\;\;+\gamma(\rho)\big[\langle F_2(\alpha,\beta),\! F_1(\alpha,\beta) \rangle
\!-\!\langle F_2(\alpha,\alpha),\! F_1(\beta,\beta)\rangle \big]
\!+\! \gamma_0\big[\langle F_2(\alpha,\beta),\! F_2(\alpha,\beta)
\rangle-\! \langle F_2(\alpha,\alpha)\!,\! F_2(\beta,\beta) \rangle
\big]
\\
&= 2(\gamma_0-\gamma(\rho))\big[ \langle F_1(\alpha,\beta), F_1(\alpha,\beta)\rangle -\langle F_1(\alpha,\alpha), F_1(\beta,\beta)\rangle
\big]
\\
&=\textstyle \frac{\gamma_0-\gamma(\rho)}{2} (\gamma'(\rho))^2\big[
(\langle \alpha_1,\beta_2\rangle+\langle\beta_1,\alpha_2\rangle)^2
-4\langle\alpha_1,\alpha_2\rangle\langle\beta_1,\beta_2\rangle\big],
\end{align*}
where we have used the fact that~$F_2=-F_1$,
by equation~\eqref{F12}.
This completes the proof.
\end{proof}
\par
The expressions provided by Proposition~\ref{ThCurvLM2} become much clearer if we go over to means and semi-differences, i.e.~if we 
use the substitutions:
\begin{equation}
\label{subs}
\alpha_1 = \overline\alpha + \delta\alpha,\quad \alpha_2 = \overline\alpha - \delta\alpha,\quad \beta_1 = \overline\beta + \delta\beta,\quad \beta_2 = \overline\beta - \delta\beta.
\end{equation}
\begin{corollary} For any
$\alpha,\beta\in(T^\ast_q \mathcal{L})_{1,2}$,
with~$\mathcal{L}=\mathcal{L}^N(\mathbb{R}^D)$, 
it is the case that:
\label{2ptR123}
\begin{align*}
R_1 &= 4\big(\gamma_0-\gamma(\rho)\big)^2\gamma''(\rho)\, \big(
\|\delta\beta^\parallel \overline\alpha  
- \delta\alpha^\parallel \overline\beta\|^2 
- \|\delta\beta^\parallel \delta\alpha 
- \delta\alpha^\parallel \delta\beta\|^2 \big) \\
&\;\;\;\;+4\big(\gamma_0-\gamma(\rho)\big)^2 
\frac{\gamma'(\rho)}{\rho} \big(
\|\delta\beta^\bot \otimes \overline\alpha  
- \delta\alpha^\bot \otimes \overline\beta\|^2 
- \|\delta\beta^\bot \otimes \delta\alpha 
- \delta\alpha^\bot \otimes \delta\beta\|^2 \big), \\
R_2 &= -4\big(\gamma_0-\gamma(\rho)\big)\gamma'(\rho)^2\, \big(
\|\delta\beta^\parallel \overline\alpha  
- \delta\alpha^\parallel \overline\beta\|^2 
- \|\delta\beta^\parallel \delta\alpha 
- \delta\alpha^\parallel \delta\beta\|^2 \big),  \\
R_3 &= \big(\gamma_0-\gamma(\rho)\big)\gamma'(\rho)^2 \,\big(
2\|
\delta\beta \otimes \overline\alpha
-
\delta\alpha \otimes \overline\beta
\|^2
-
\|
\overline\beta \otimes \overline\alpha
- 
\overline\alpha \otimes \overline\beta 
\|^2
-
\|
\delta\beta \otimes \delta\alpha
- 
\delta\alpha \otimes \delta\beta 
\|^2 
\big).
\end{align*}
\end{corollary}
\begin{proof}
By insertion of formulae~\eqref{subs}
it is easily seen that
\begin{align*}
\big\langle 
\delta\alpha^\parallel \beta_1 - \delta\beta^\parallel \alpha_1, 
\delta\alpha^\parallel \beta_2 - \delta\beta^\parallel \alpha_2 
\big \rangle 
&=
\|\delta\beta^\parallel \overline\alpha  
- \delta\alpha^\parallel \overline\beta\|^2 
- \|\delta\beta^\parallel \delta\alpha 
- \delta\alpha^\parallel \delta\beta\|^2,
\\
\big\langle
\delta\alpha^\bot\otimes \beta_1 - 
\delta\beta^\bot \otimes \alpha_1, 
\delta\alpha^\bot\otimes \beta_2 - 
\delta\beta^\bot \otimes \alpha_2 \big\rangle
&=
\|\delta\beta^\bot \otimes \overline\alpha  
- \delta\alpha^\bot \otimes \overline\beta\|^2 
\!- \|\delta\beta^\bot \otimes \delta\alpha 
- \delta\alpha^\bot \otimes \delta\beta\|^2\!\!,
\end{align*}
so the new expressions for~$R_1$ and~$R_2$ 
follow immediately. 
Also:
\begin{align*}
\big[\langle \alpha_1,
&
\beta_2\rangle+\langle\beta_1,
\alpha_2\rangle\big]^2
-4\langle\alpha_1,\alpha_2\rangle\langle\beta_1,\beta_2\rangle
\\
&=
\big[
2(
\langle
\overline\alpha,\overline\beta
\rangle
-
\langle
\delta\alpha,\delta\beta
\rangle
)
\big]^2
-4
\big(
\langle
\overline\alpha,\overline\alpha
\rangle
-
\langle
\delta\alpha,\delta\alpha
\rangle
\big)
\big(
\langle
\overline\beta,\overline\beta
\rangle
-
\langle
\delta\beta,\delta\beta
\rangle
\big)
\\
&=
-2
\big[
2(
\langle
\overline\alpha,\overline\alpha
\rangle
\langle
\overline\beta,\overline\beta
\rangle
-
\langle
\overline\alpha,\overline\beta
\rangle^2
)
\big]
-
2
\big[
2(
\langle
\delta\alpha,\delta\alpha
\rangle
\langle
\delta\beta,\delta\beta
\rangle
-
\langle
\delta\alpha,\delta\beta
\rangle^2
)
\big]
\\
&\quad
+4
\big[
\langle
\overline\alpha,\overline\alpha
\rangle
\langle
\delta\beta,\delta\beta
\rangle
+
\langle
\overline\beta,\overline\beta
\rangle
\langle
\delta\alpha,\delta\alpha
\rangle
-
2
\langle
\overline\alpha,\overline\beta
\rangle
\langle
\delta\alpha,\delta\beta
\rangle
\big]
\\
&=
-2\|
\overline\beta \otimes \overline\alpha
- 
\overline\alpha \otimes \overline\beta 
\|^2 
-2\|
\delta\beta \otimes \delta\alpha
- 
\delta\alpha \otimes \delta\beta 
\|^2
+4\|
\delta\beta \otimes \overline\alpha
-
\delta\alpha \otimes \overline\beta
\|^2
.\qedhere
\end{align*}
\end{proof}
The fourth term $R_4$ is the only one which involves the other points 
$q^a, a>2$. But one has an inequality for this term involving the same expressions in $\alpha$ and $\beta$: 
\begin{proposition} \label{L2RD} 
Any pair $\alpha,\beta\in (T^\ast\mathcal{L}^N)_{1,2}$ are constant 1-forms on $\mathcal L^N$ which are pull-backs via the submersion $\mathcal L^N \rightarrow \mathcal L^2$ of constant 1-forms on $\mathcal L^2$. We can therefore consider the curvature term $R_4({\mathcal L^N}) = -\tfrac34 \| [\alpha^\sharp, \beta^\sharp]_{\mathcal L^N}\|^2$ on $\mathcal L^N$ and the corresponding term $R_4({\mathcal L^2}) = -\tfrac34 \| [\alpha^\sharp, \beta^\sharp]_{\mathcal L^2}\|^2$ on $\mathcal L^2$. Then we have the inequality:
\begin{equation*} R_4(\mathcal L^N) \le R_4(\mathcal L^2) = 
-6\gamma'(\rho)^2 \bigg[
\frac{\big(\gamma_0-\gamma(\rho)\big)^2}{\gamma_0+\gamma(\rho)}
\|
\delta\beta^\parallel \overline\alpha
-
\delta\alpha^\parallel \overline\beta 
\|^2 
+
\big(\gamma_0-\gamma(\rho)\big)\, 
\|
\delta\beta^\parallel \delta\alpha
-
\delta\alpha^\parallel\delta\beta 
\|^2\bigg].
\end{equation*}
\end{proposition}
\begin{proof}
Firstly, note that $[\alpha^\sharp, \beta^\sharp]_{\mathcal L^N}$ breaks into perpendicular parts: a vertical part in the kernel of $d\pi$ and a horizontal part which is simply the horizontal lift of 
$[\alpha^\sharp, \beta^\sharp]_{\mathcal L^2}$. This explains the inequality assertion in Proposition~\ref{L2RD}. 
To calculate $R_4(\mathcal L^2)$, we use the last expression in \eqref{gen_R4}, i.e.
\begin{align*}
R_4(\mathcal{L}^2)
&=
\textstyle
-\frac{3}{4}
\sum_{a,b=1}^2
\big\langle
D^a(\alpha,\beta)-D^a(\beta,\alpha),
D^b(\alpha,\beta)-D^b(\beta,\alpha)
\big\rangle\,
(\mathbf{ K}^{-1})_{ab}
\\
&=
-3\, \frac{\gamma_0-\gamma(\rho)}{\gamma_0+\gamma(\rho)}
\,
\gamma'(\rho)^2
\Big\{\gamma_0
\Big[
\|
\delta\alpha^\parallel \beta_1 
-
\delta\beta^\parallel \alpha_1
\|^2 
+
\|
\delta\alpha^\parallel \beta_2 
-
\delta\beta^\parallel \alpha_2
\|^2 \Big]
\\&\qquad\qquad\qquad\qquad\qquad\;\;
-2\gamma(\rho)
\big\langle
\delta\alpha^\parallel \beta_1 
-
\delta\beta^\parallel \alpha_1
,
\delta\alpha^\parallel \beta_2 
-
\delta\beta^\parallel \alpha_2
\big\rangle
\Big\},
\end{align*}
where we have used~\eqref{invK} and~\eqref{D12}. 
The final result follows after
inserting~\eqref{subs} into the above expression
and performing  
some algebra.
%
%
\end{proof}
Note that all terms in Corollary \ref{2ptR123} and Proposition \ref{L2RD} are very similar. In fact, they are all ``components'' of the norm 
$\|\alpha \wedge \beta\|^2$ of the 2-form whose sectional curvature is being computed. First note that we can decompose $T^*\mathcal L^2$ into the direct sum of three pieces, namely:
\begin{align*}
\delta^\parallel T^*\mathcal L^2
&:=
\big\{
(au,-au)\,\big|\,a\in\mathbb{R}
\big\},
&
&
\dim \big(
\delta^\parallel T^*\mathcal L^2\big)
=1,
\\
\delta^\perp T^*\mathcal L^2
&:=
\big\{
(p,-p)
\,\big|\,
p\in\mathbb{R}^D,p\perp u
\big\},
&
&
\dim \big(
\delta^\perp T^*\mathcal L^2
\big)
=D-1,
\\
\overline T^* \mathcal L^2&:=
\big\{
(p,p)
\,\big|\,
p\in\mathbb{R}^D
\big\},
&
&
\dim\big( 
\overline T^* \mathcal L^2
\big)=D,
\end{align*}
where as usual~$u:=\frac{q^1-q^2}{\|q^1-q^2\|}$
(see Figure~\ref{fig3spaces}).
Note that these three subspaces  are orthogonal
with respect to the cometric by virtue of~\eqref{2pt_com_1}.
An arbitrary covector~$\alpha=(\alpha_1,\alpha_2)\in 
T^\ast_q\mathcal{L}^2$ can be uniquely 
decomposed into the summation
$\alpha=\alpha_{(1)}+\alpha_{(2)}+\alpha_{(3)}$, with:
\begin{align}
\label{decomp}
\alpha_{(1)}
&
:=
(
\delta\alpha^\parallel u
,
-
\delta\alpha^\parallel u
)
\in\delta^\parallel T^\ast \mathcal{L}^2,
& 
\alpha_{(2)}
&
:=
(\delta\alpha^\perp,-\delta\alpha^\perp)
\in\delta^\perp T^\ast \mathcal{L}^2,
&
\alpha_{(3)}
&
:=
(\overline\alpha,\overline\alpha)
\in \overline{T}^\ast \mathcal{L}^2.
\end{align}
So it is the case that: (i)
$\alpha\in \delta^\parallel T^\ast \mathcal{L}^2
 \Leftrightarrow \delta^\perp\alpha=0$ and $\overline{\alpha} =0$;
(ii)
$\alpha\in \delta^\perp T^\ast \mathcal{L}^2
 \Leftrightarrow \delta^\parallel\alpha=0$  and $\overline{\alpha} =0$;
(iii)
$\alpha\in  \overline{T}^\ast \mathcal{L}^2
 \Leftrightarrow \delta\alpha^\parallel=0$ and $\delta\alpha^\perp=0$.
%
\begin{figure}[t]
\begin{center}
\begin{picture}(120,60)
\setlength{\unitlength}{1pt}
\qbezier[80](0,20)(60,20)(120,20)
\put( 40,20){\makebox(0,0){\Large$\bullet$}}
\put( 80,20){\makebox(0,0){\Large$\bullet$}}
\put(44,15){\makebox(0,0)[t]{\small$q^1$}}
\put(75,15){\makebox(0,0)[t]{\small$q^2$}}
\put(20,25){\makebox(0,0)[b]{\small$\alpha_1$}}
\put(105,25){\makebox(0,0)[b]{\small$\alpha_2$}}
\put(60,65){\makebox(0,0){
$\boxed{
\;\delta^\parallel T^\ast_q\mathcal{L}
}$
}}
\thicklines
\put(40,20){\vector(-1,0){28}} %
\put(80,20){\vector(1,0){28}} %
\end{picture}
\,\,\,\,
\,\,\,\,
\,\,
\begin{picture}(120,60)
\setlength{\unitlength}{1pt}
\qbezier[80](0,20)(60,20)(120,20)
\put( 40,20){\makebox(0,0){\Large$\bullet$}}
\put( 80,20){\makebox(0,0){\Large$\bullet$}}
\put(41,26){\makebox(0,0)[b]{\small$q^1$}}
\put(80,15){\makebox(0,0)[t]{\small$q^2$}}
\put(37,-5){\makebox(0,0)[r]{\small$\alpha_1$}}
\put(84,45){\makebox(0,0)[l]{\small$\alpha_2$}}
\put(60,65){\makebox(0,0){
$\boxed{
\;\delta^\perp T^\ast_q\mathcal{L}
}$
}}
\thicklines
\put(40,20){\vector(0,-1){28}} %
\put(80,20){\vector(0,1){28}} %
\end{picture}
\,\,
\,\,\,\,
\,\,\,\,
\begin{picture}(120,60)
\setlength{\unitlength}{1pt}
\qbezier[80](0,20)(60,20)(120,20)
\put( 40,20){\makebox(0,0){\Large$\bullet$}}
\put( 80,20){\makebox(0,0){\Large$\bullet$}}
\put(40,15){\makebox(0,0)[t]{\small$q^1$}}
\put(80,15){\makebox(0,0)[t]{\small$q^2$}}
\put(20,39){\makebox(0,0){\small$\alpha_1$}}
\put(60,39){\makebox(0,0){\small$\alpha_2$}}
\put(60,65){\makebox(0,0){
$\boxed{
\;\overline{T}^\ast_q\mathcal{L}
}$
}}
\thicklines
\put(40,20){\vector(-2,1){28}} %
\put(80,20){\vector(-2,1){28}} %
\end{picture}
\end{center}
\caption{Typical covectors~$\alpha=(\alpha_1,\alpha_2)$
in spaces~$\delta^\parallel T^\ast_q\mathcal{L}$,
$\delta^\perp T^\ast_q\mathcal{L}$,
and
$\overline{T}^\ast_q\mathcal{L}$,
for $\mathcal{L}=\mathcal{L}^2(\mathbb{R}^2)$.}
\label{fig3spaces}
\end{figure} 
\par
Consequently the space of 2-forms $\bigwedge^2 T^*\mathcal L^2$ decomposes into the direct sum of five pieces:
\begin{align*}
\bigwedge^2 T^*\mathcal L^2 &\;=\; \bigoplus_{i=1}^5V_i,
\mbox{ with: }
&
V_1
&\;:=\;
\delta^\parallel T^*\mathcal L^2 
\;\wedge\; 
\overline T^*\mathcal L^2,
\\
V_2
&\;:=\;
\delta^\bot T^*\mathcal L^2 
\;\wedge\;  
\overline T^*\mathcal L^2,
&
V_3
&\;:=\;
\delta^\parallel T^*\mathcal L^2 
\;\wedge\; 
\delta^\bot T^*\mathcal L^2,
\\
V_4
&\;:=\;
\bigwedge^2\big( \delta^\bot T^*\mathcal L^2\big),
&
V_5
&\;:=\;
\bigwedge^2\big( \overline T^*\mathcal L^2\big).
\end{align*}
(Since $\delta^\parallel T^*\mathcal L^2$ is one-dimensional it creates no 2-forms.)
Once again, note that the spaces~$V_1,\ldots,V_5$ are pairwise 
\em  orthogonal \em with respect to the inner product
\begin{equation}
\label{2formprod}
\big\langle
\alpha\wedge\beta
,
\xi\wedge\eta
\big\rangle_{\bigwedge^2 T^*\mathcal L^2}
:=
\big
\langle
\alpha
,
\xi
\big\rangle_{T^*\mathcal L^2}
\big
\langle
\beta
,
\eta
\big\rangle_{T^*\mathcal L^2}
-
\big
\langle
\alpha
,
\eta
\big\rangle_{T^*\mathcal L^2}
\big
\langle
\beta
,
\xi
\big\rangle_{T^*\mathcal L^2},
\quad
\alpha,\beta,\xi,\eta\in T^\ast\mathcal{L}^2
\end{equation}
by the orthogonality of~$\delta^\parallel T^\ast_q\mathcal{L}$,
$\delta^\perp T^\ast_q\mathcal{L}$,
and
$\overline{T}^\ast_q\mathcal{L}$.
Any 2-form $\alpha\wedge \beta$ then decomposes into the sum of its five projections onto these subspaces and its norm squared is the sum of the norm squared of these components. Let us first give the five pieces of its norm names:
\begin{align*}
& & 
T_1 &:= \|\delta\beta^\parallel u\otimes\overline\alpha 
-\delta\alpha^\parallel u \otimes\overline\beta\|^2 ,
\\
T_2 &:= \|\delta\beta^\bot \otimes \overline\alpha 
-\delta\alpha^\bot \otimes \overline\beta\|^2 ,
&
T_3 &:= \|\delta\beta^\parallel  u\otimes \delta\alpha^\bot 
-\delta\alpha^\parallel u \otimes \delta\beta^\bot\|^2 , 
\\
T_4 &:= \|\delta\beta^\bot \otimes \delta\alpha^\bot 
-\delta\alpha^\bot \otimes \delta\beta^\bot\|^2 ,
&
T_5 &:= \|\overline\beta \otimes \overline\alpha 
-\overline\alpha \otimes \overline\beta\|^2.
\end{align*}
In the above definitions~$\|\;\;\;\|$ indicates the Euclidean norm.
We have to be careful here: we have been using Euclidean norms in $\mathbb R^D$ in all our formulas above and now we are dealing with norms in $T^*\mathcal L^2$; these essentially differ only by a factor, by \eqref{2ptcometric}. More precisely, the following result holds:
%
\begin{proposition}
The denominator of the sectional curvature~\eqref{seccurv} for~$\mathcal{L}^2(\mathbb{R}^2)$
can be written as:
\begin{equation}
\label{curvdenom} 
\|\alpha\wedge\beta\|_{\bigwedge^2 T^*\mathcal L^2}^2 = 
4\big(\gamma_0^2-\gamma(\rho)^2\big) (T_1+T_2) 
+2\big(\gamma_0-\gamma(\rho)\big)^2(2T_3+T_4)
+2\big(\gamma_0+\gamma(\rho)\big)^2 T_5. 
\end{equation}
\end{proposition}
\begin{proof}
We may apply decomposition~\eqref{decomp}
to both~$\alpha=\sum_{i=1}^3\alpha_{(i)}$
and~$\beta=\sum_{i=1}^3\beta_{(i)}$, and write
\begin{align*}
\alpha\wedge\beta 
=
&
\big(
\alpha_{(1)}\wedge\beta_{(3)}
-
\beta_{(1)}\wedge\alpha_{(3)}
\big)
+
\big(
\alpha_{(2)}\wedge\beta_{(3)}
-
\beta_{(2)}\wedge\alpha_{(3)}
\big)
\\
&+
\big(
\alpha_{(1)}\wedge\beta_{(2)}
-
\beta_{(1)}\wedge\alpha_{(2)}
\big)
+
\alpha_{(2)}\wedge\beta_{(2)}
+
\alpha_{(3)}\wedge\beta_{(3)},
\end{align*}
where the five summands on the right-hand side belong to
$V_1,\ldots,V_5$ respectively.
We have
\begin{align*}
\|&\alpha_{(1)}\wedge\beta_{(3)}
-
\beta_{(1)}\wedge\alpha_{(3)}
\|^2_{\bigwedge^2 T^*\mathcal L^2}\,\,=
\\
&
=
\big\|\alpha_{(1)}\wedge\beta_{(3)}
\|^2_{\bigwedge^2 T^*\mathcal L^2}
+
\big\|
\beta_{(1)}\wedge\alpha_{(3)}
\|^2_{\bigwedge^2 T^*\mathcal L^2}
-2
\langle
\alpha_{(1)}\wedge\beta_{(3)}
,
\beta_{(1)}\wedge\alpha_{(3)}
\rangle_{\bigwedge^2 T^*\mathcal L^2}
\\
&\stackrel{(\ast)}{=}
4\big(\gamma_0^2-\gamma(\rho)^2\big)
\big[
(\delta\alpha^\parallel)^2\|\overline\beta\|^2
+
(\delta\beta^\parallel)^2\|\overline\alpha\|^2
-2\,
\delta\alpha^\parallel 
\delta\beta^\parallel 
\langle\alpha,\beta\rangle
\big]
=
4\big(\gamma_0^2-\gamma(\rho)^2\big)\,T_1,
\end{align*}
where we have used \eqref{2formprod}
and \eqref{2ptcometric} in step~($\ast$).
The square norm of the remaining four terms is computed similarly.
Orthogonality of $V_1,\ldots,V_5$ finally yields~\eqref{curvdenom}.
\end{proof}
\par
To express the formulas for the \em numerator \em of sectional curvature succinctly, let us also introduce abbreviations for the coefficients involving $\gamma$:
\begin{equation}
\label{def_ks}
\begin{aligned}
k_1(\rho) &:= \big(\gamma_0-\gamma(\rho)\big)^2 \gamma''(\rho),
&\qquad
k_2(\rho) &:= \big(\gamma_0-\gamma(\rho)\big)^2 \frac{\gamma'(\rho)}{\rho}, \\
k_3(\rho) &:= \big(\gamma_0-\gamma(\rho)\big)\gamma'(\rho)^2,  
&
k_4(\rho) &:= \frac{\big(\gamma_0-\gamma(\rho)\big)^2}{\gamma_0+\gamma(\rho)}
\gamma'(\rho)^2 .
\end{aligned}
\end{equation}
%
Note that $k_1$, $k_2$, $k_3$ and $k_4$ are all homogeneous of degree 3 in $\gamma$ and degree $-2$ in the distance $\rho$ or $d\rho$ on $\mathcal L^N$. Moreover $k_2$ is negative, $k_3$ and $k_4$ are positive, while $k_1$ may be positive or negative. For all $\gamma$ of interest, $\gamma'$ is everywhere negative, starting at 0 decreasing to a minimum at some $\rho_0$, then increasing back to 0 at $\infty$. Then $k_1$ is negative for $\rho < \rho_0$ and positive for $\rho > \rho_0$. 
\par
The following equalities are proven by direct computation:
\begin{align*}
\|
\delta\beta \otimes \overline\alpha 
-
\delta\alpha \otimes \overline\beta
\|^2
&=
T_1+T_2,
\\
\|
\delta\beta^\perp \otimes \delta\alpha 
-
\delta\alpha^\perp \otimes \delta\beta
\|^2
&=
T_3+T_4,
\\
\|
\delta\beta \otimes \delta\alpha 
-
\delta\alpha \otimes \delta\beta
\|^2
&=
2T_3+T_4.
\end{align*}
Inserting notation~\eqref{def_ks} and the above equalities into
Propositions~\ref{ThCurvLM2} and~\ref{L2RD}
immediately yields:
%
\begin{figure}[t]
\epsfig{width=3in,height=1.8in,file=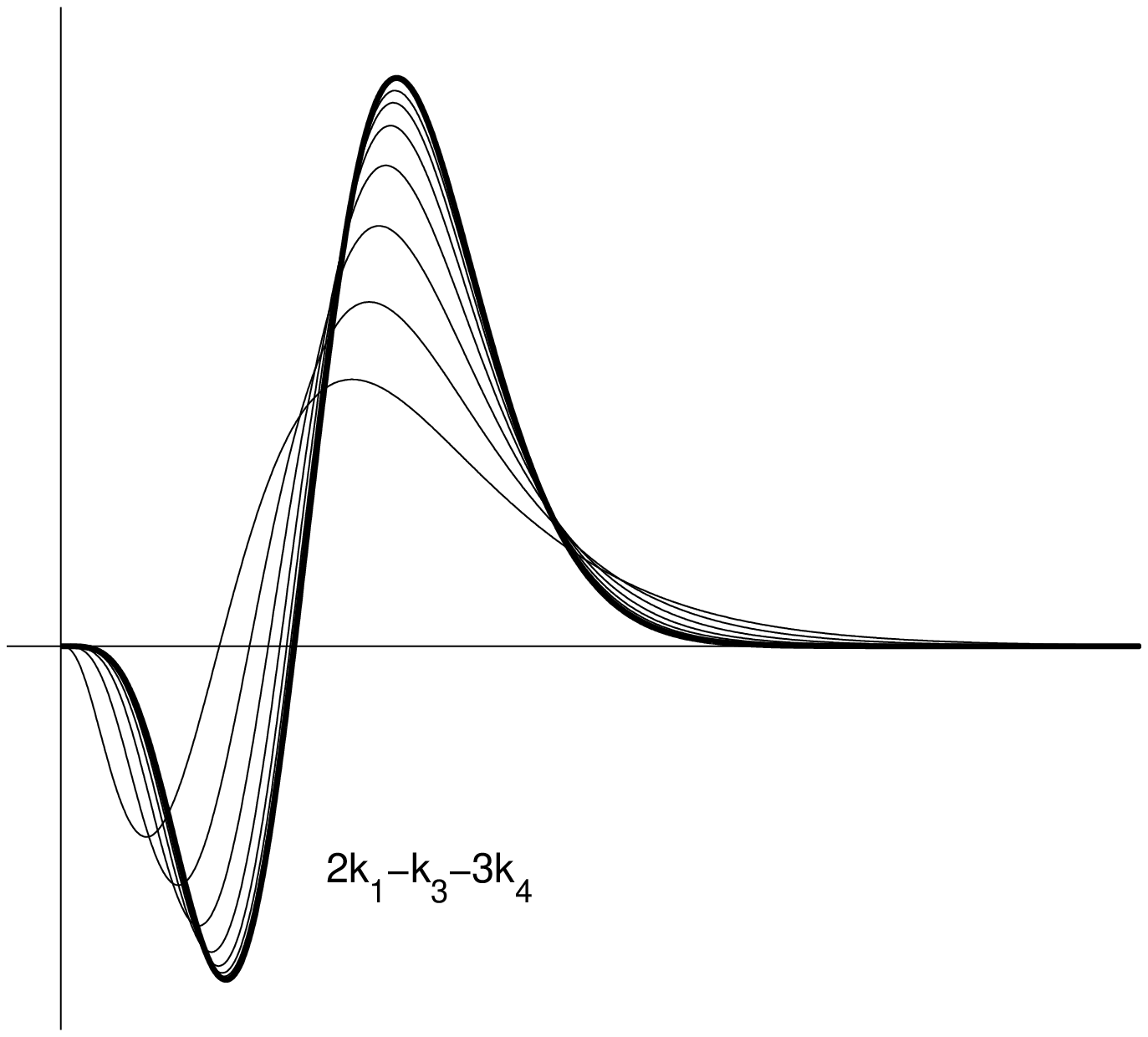}
\epsfig{width=2.9in,height=1.8in,file=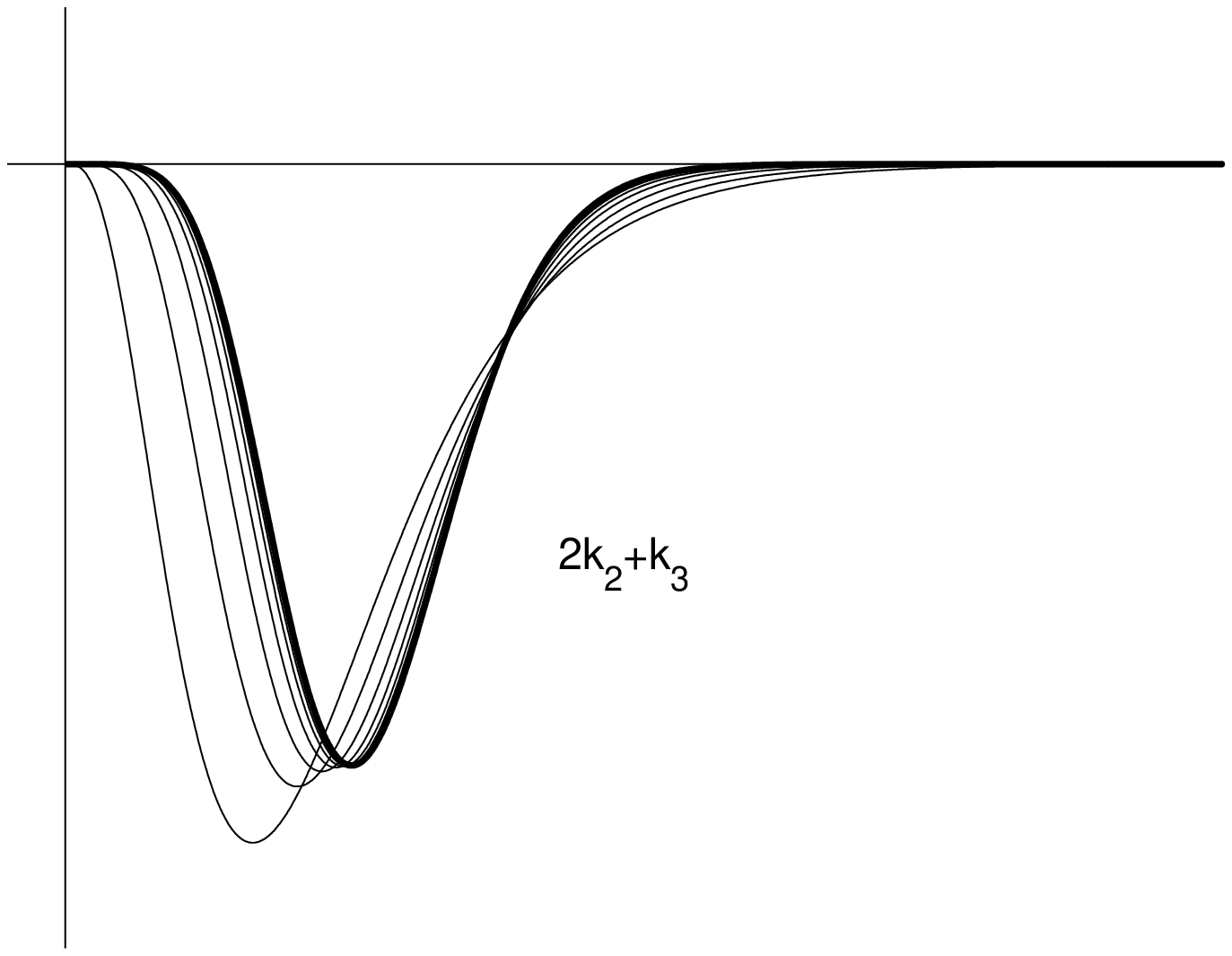}
\\ 
\epsfig{width=3in,height=1.8in,file=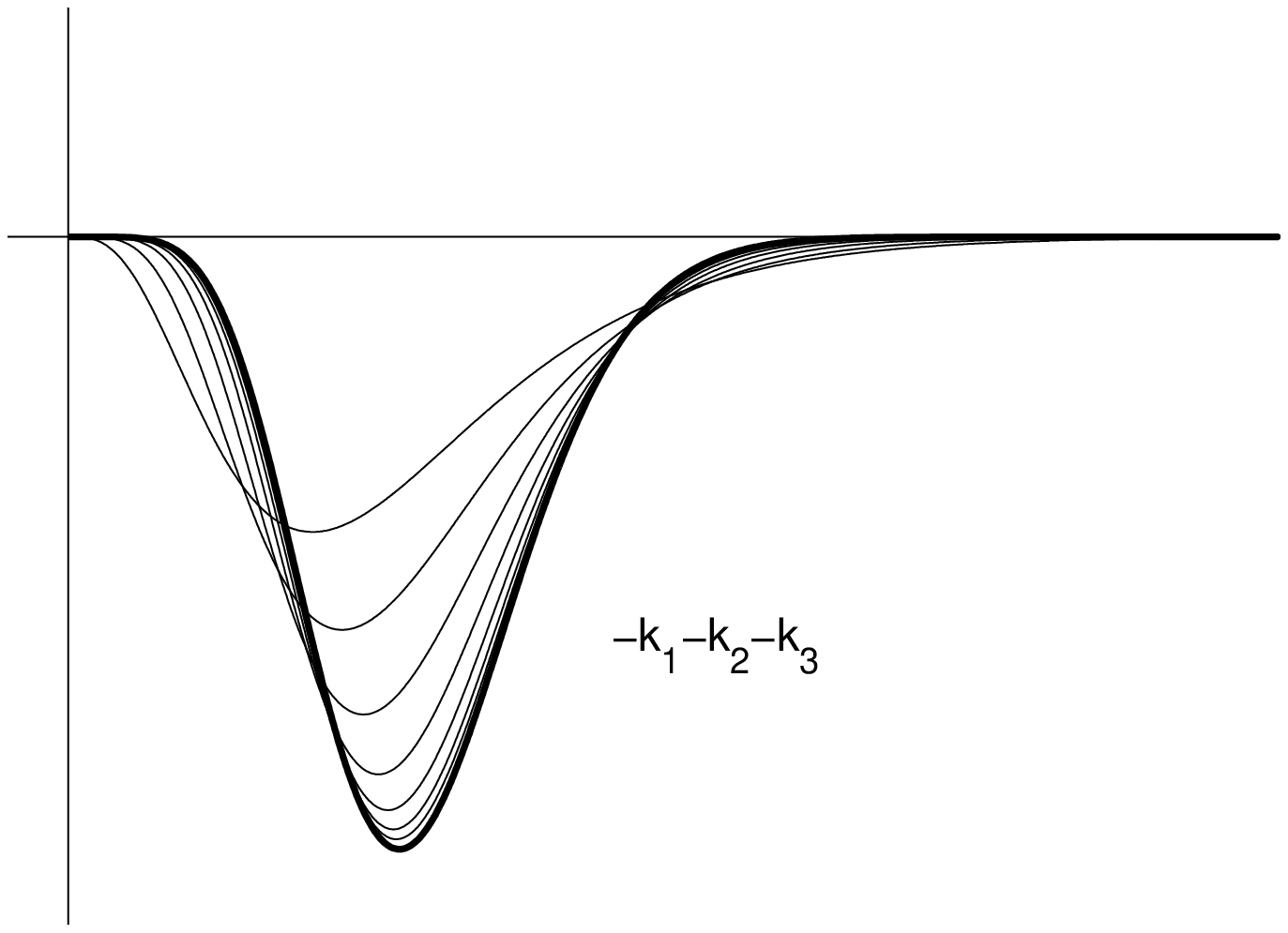}
\epsfig{width=2.9in,height=2in,file=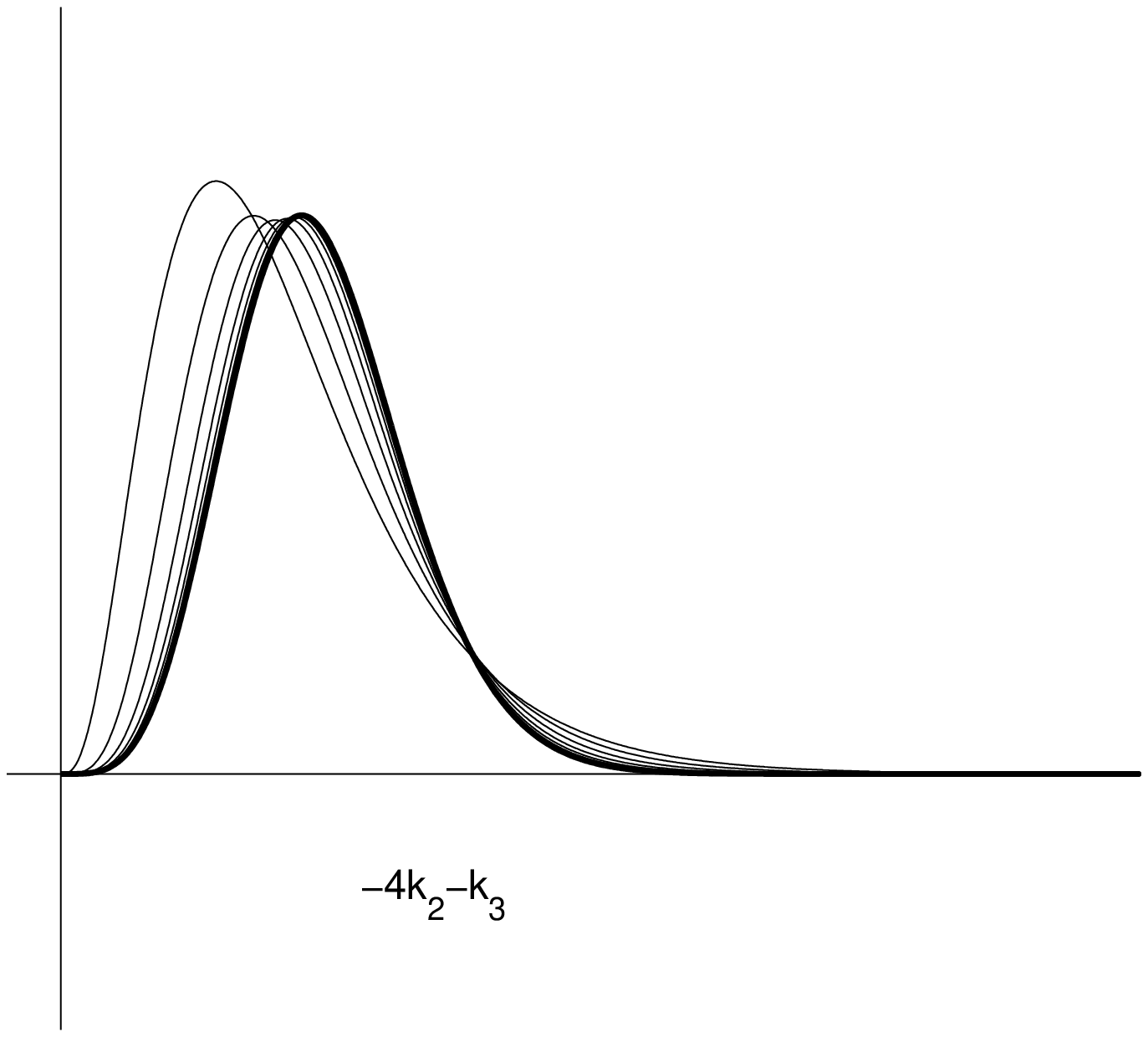}
\caption{The coefficients of $T_1$ (top left), $T_2$ (top right), $T_3$ (bottom left) and $T_4$ (bottom right) for the Bessel kernels $\gamma$ (shown with thin lines) and the Gaussian kernel (shown with the thick line). The kernels are scaled to normalize $\gamma(0)$ and $\gamma''(0)$.}
\label{fourcoeff}
\end{figure}
\begin{proposition}
We can write the terms in the numerator of sectional curvature for $\mathcal L^2(\mathbb{R}^D)$ as:
\begin{equation}
\label{R_final}
\begin{aligned}
R_1 &= 4k_1(T_1-T_3) + 4k_2(T_2-T_3-T_4), &\qquad
R_2 &= -4k_3 (T_1-T_3),\\
R_3 &= k_3(2(T_1+T_2)-2T_3-T_4-T_5), &
R_4 
&=-6(k_3 T_3 + k_4 T_1),
\end{aligned}
\end{equation}
%
hence $R=R(\alpha^\sharp,\beta^\sharp,\beta^\sharp,\alpha^\sharp)
=\sum_{i=1}^4R_i$
may be expressed as:
\begin{equation} \label{2ptcurvature}
R 
=
2\big(2k_1-k_3-3k_4\big)T_1 + 2\big(2k_2+k_3\big)T_2
 +4\big(-k_1-k_2-k_3\big)T_3 +\big(-4k_2-k_3\big)T_4 -k_3T_5.\end{equation}
\end{proposition}
\par
By virtue of Proposition~\ref{L2RD}  the above proposition still holds
in the case of~$\mathcal{L}^N(\mathbb{R}^D)$
as long as~$\alpha,\beta\in(T_q^\ast\mathcal{L})_{1,2}$ and the equality signs for~$R_4$ in~\eqref{R_final}
and~$R$ in~\eqref{2ptcurvature} are substituted  by ``$\leq$''.
The coefficients in~\eqref{2ptcurvature} may have all sorts of signs for peculiar kernels. However, the kernels~$\gamma$ of interest are the Bessel kernels \eqref{Besselkernel} and the Gaussian kernel, which is their asymptotic limit as their order goes to infinity. The coefficients for these kernels are shown in Figure \ref{fourcoeff}. We see that the coefficients of $T_2$ and $T_3$ are negative while those of $T_4$ are positive. {\it Henceforth, we assume we have a kernel for which this is true.}

\subsection{Sectional curvature of~$\mathcal{L}^2(\mathbb{R}^1)$}
\label{sectionL21}
\begin{figure}[t]
\center{\includegraphics[height=7.2cm]{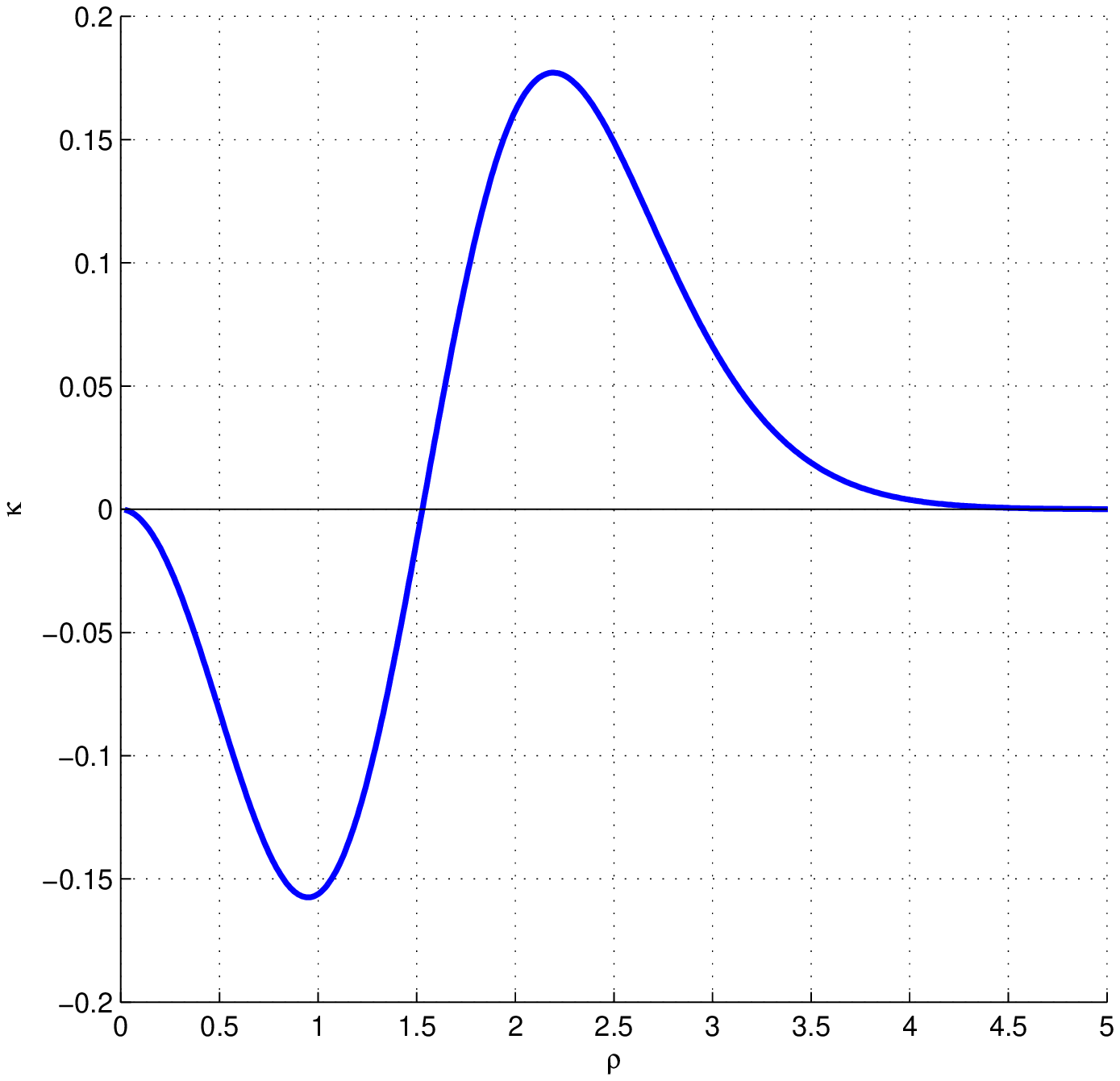}
\;
\includegraphics[height=7.2cm]{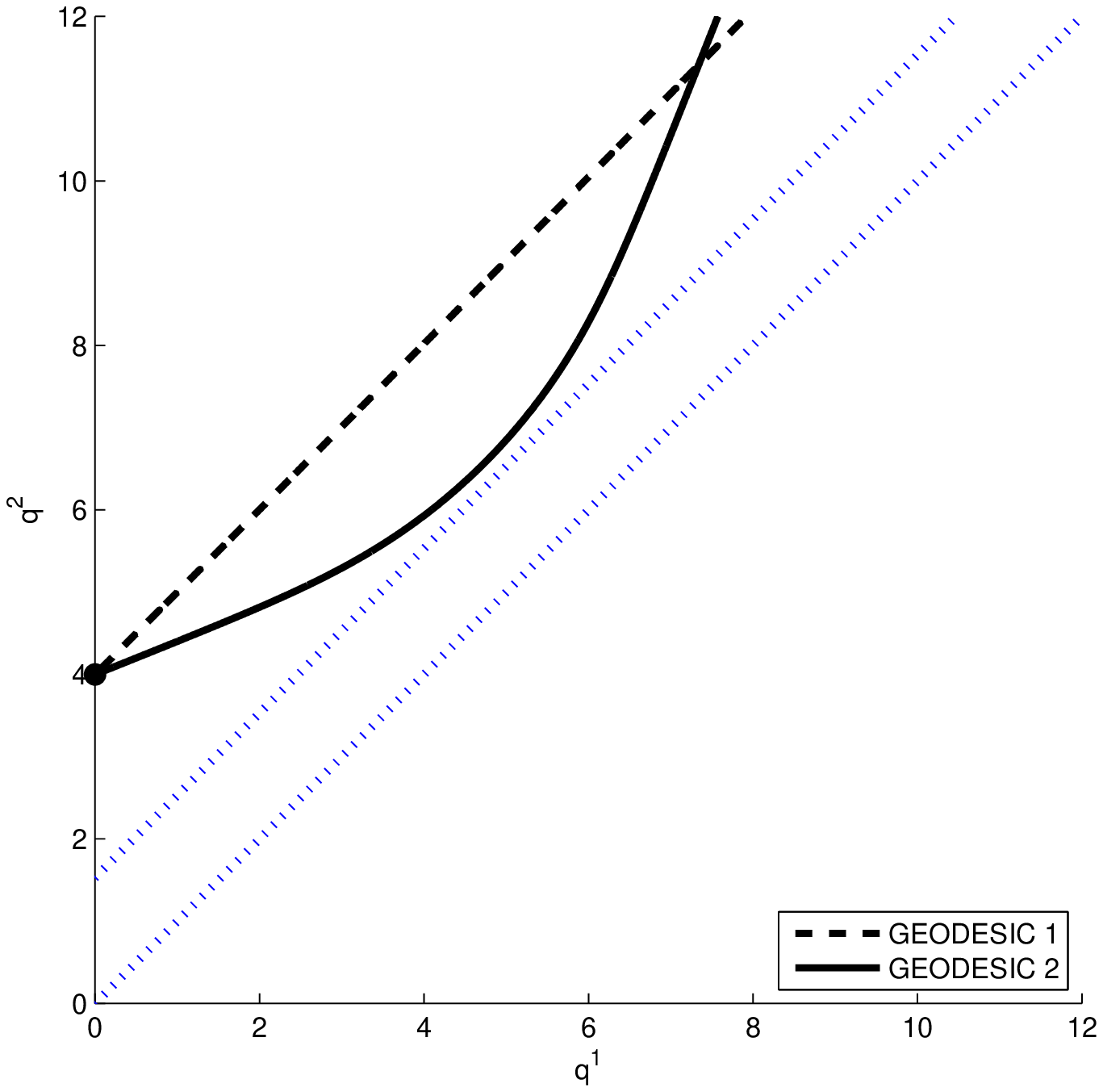}
}
\caption{
Left: sectional curvature~$\mathcal{K}$ 
for~~$\mathcal{L}^2(\mathbb{R}^1)$
(from Proposition~\ref{scurv_L2R1}),
as a function of~$\rho=|q^1-q^2|$;
here $\gamma(x)=
\exp(-\frac{1}{2}x^2)$.
Right: two trajectories in~$\mathcal{L}^2(\mathbb{R}^1)$
shown in the~$(q^1,q^2)$ plane (under the assumption that~$q^1<q^2$).
Both geodesics originate at $(q^1,q^2)=(0,4)$, and 
lie in the region where~$\mathcal{K}>0$ (above the upper dotted
 line, that indicates the zero of~$\mathcal{K}$ 
 at~$|q^1-q^2|\simeq1.53$); they have different 
initial momenta
$(p_1,p_2)=(1,1)$ and $(p_1,p_2)=(1,0.4)$, and exhibit conjugate points.
}\label{fig_L2R1}
\end{figure} 
Finally, we will now explore the important example of two landmarks on the real line. In this particular case the manifold is two dimensional, so 
sectional curvature~$\mathcal{K}$ will turn out to be independent of cotangent vectors~$\alpha$ and~$\beta$. In fact, given the translation invariance of the metric tensor, it will only depend on the 
distance~$\rho=|q^1-q^2|$ between the two landmarks.

The spaces $\delta^\parallel T^\ast \mathcal L^2(\mathbb R^1)$ and $\overline T^\ast \mathcal L^2(\mathbb R^1)$ are one-dimensional while $\delta^\bot T^\ast \mathcal L^2(\mathbb R^1) = \{0\}$. Thus
 $$
 \bigwedge^2 T^\ast \mathcal L^2(\mathbb R^1) 
 \;= \;
 \delta^\parallel T^\ast \mathcal L^2(\mathbb R^1)\; \wedge\; \overline T^\ast \mathcal L^2(\mathbb R^1)
 $$
and the only non-zero term in \eqref{2ptcurvature} is $T_1$. Therefore combining formulas \eqref{curvdenom} and \eqref{2ptcurvature} we get:

\begin{proposition}
\label{scurv_L2R1}
The sectional curvature of~$\mathcal{L}^2(\mathbb{R}^1)$
is given by
$$\mathcal{K}= \frac{2k_1-k_3-3k_4}{2(\gamma_0^2-\gamma(\rho)^2)} = 
\frac{\gamma_0-\gamma(\rho)}{\gamma_0+\gamma(\rho)}\,\gamma''_{12}
-\frac{2\gamma_0-\gamma(\rho)}{(\gamma_0+\gamma(\rho))^2}
\big(\gamma'(\rho)\big)^2.$$ 
\end{proposition}
\par
The above function~$\mathcal{K}$ is shown on the left-hand side of Figure~\ref{fig_L2R1} as a function of~$\rho$, for the Gaussian kernel. The coefficient of the term $T_1$ in~\eqref{2ptcurvature} is negative for $\rho$ small and positive for $\rho$ large. The ``cause'' of the positive curvature has been analyzed in \cite{micheli.phd}. Roughly speaking, suppose two points both want to move a fixed distance to the right. Then if they are far enough away, they can just move more or less independently (we shall refer to this as Geodesic~1). {\it Or} (i) the one in back can speed up while the one in front slows down, then (ii) when the pair are close, they move in tandem using less energy because they are close and finally (iii) the back one slows down, the front one speeds up when they near their destinations 
(Geodesic~2). This gives explicit conjugate points (in the sense that two points are joined by distinct geodesics) and is illustrated on the right-hand side of figure Figure~\ref{fig_L2R1} 
(where Geodesics~1 and~2 are represented, respectively, by the
dashed and thick curves).

\subsection{Sources of positive curvature;  obstacle avoidance}
There is another source of positive curvature in $\mathcal L^2$ in higher dimensions. It is clear from equation~\eqref{2ptcurvature} 
and Figure~\ref{fourcoeff} that any positive curvature must come from the term with $T_1$ or the term with $T_4$. As the five terms are orthogonal, we can make all of them but one zero. 
\par
For example, if we choose~$\alpha=(\delta\alpha^\parallel u,
-\delta\alpha^\parallel u)\in\delta^\parallel T^\ast\!\mathcal{L}$
and~$\beta=(\overline{\beta},
\overline{\beta})\in\overline{T}^\ast\!\mathcal{L}$,
then it is the case that~$T_1=(\delta\alpha^\parallel)^2\|\overline{\beta}\|^2$
and it is the only non-zero term.
Then, if $\rho$ is sufficiently large, the sectional curvature for this 2-plane is positive as discussed in the last section.
Figure~\ref{fig_L2R2_conj} illustrates an instance of the existence
of conjugate points for two geodesics  
in~$\mathcal{L}^2(\mathbb{R}^2)$; the momenta~$(p_1,p_2)$ of each of the two trajectories belong at all times to~$\delta^\parallel T^\ast\! \mathcal{L}^2\oplus\overline{T}^\ast\! \mathcal{L}^2$.
\begin{figure}[t]
\center{\includegraphics[height=7.2cm]{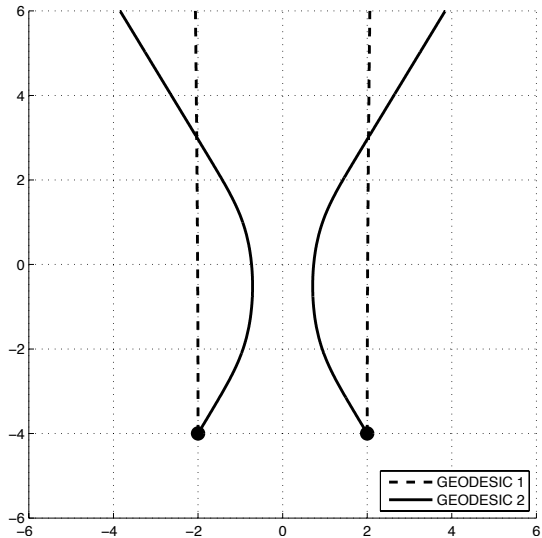}
}
\caption{
Existence of conjugate points in~$\mathcal{L}^2(\mathbb{R}^2)$,
with~$\gamma(x)=
\exp(-\frac{1}{2}x^2)$. Both geodesics 
originate at landmark 
set~$(q^1,q^2)=\big((-2,-4),(2,4)\big)$;
the first one (dashed) has initial 
momentum $(p_1,p_2)=\big((0,10),(0,10)\big)\in \overline{T}^\ast \mathcal{L}^2$
while the second one (continuous) has initial 
momentum $(p_1,p_2)=\big((6,10),(-6,10)\big)
\in \delta^\parallel T^\ast \mathcal{L}^2\oplus\overline{T}^\ast \mathcal{L}^2$.
The geodesic trajectories  exhibit conjugate points.}\label{fig_L2R2_conj}
\end{figure} 
\par
The other possibility is that $T_4$ is the non-zero term, which happens when~$\alpha=(\delta\alpha^\perp,
-\delta\alpha^\perp)\in\delta^\perp T^\ast\!\mathcal{L}$
and~$\beta=(\delta\beta^\perp,
-\delta\beta^\perp)\in\delta^\perp T^\ast\!\mathcal{L}$.
We have~$T_4=2(\| \delta\alpha^\perp \|^2
\| \delta\beta^\perp \|^2-\langle\delta\alpha^\perp,\delta\beta^\perp\rangle^2)$,
and for it to be nonzero it is required
that $D \ge 3$ because $T_4$ is the norm of a 2-form in $
\bigwedge^2\big(\delta^\bot T^\ast \mathcal L\big)$, which has dimension $(D-1)(D-2)/2$. 
The positive curvature of this section is readily seen by considering the geodesics which these vectors generate. The simplest example is the following:
\begin{proposition}
The circular periodic orbit of radius~$r$:
\begin{equation}
\label{circular}
 q^1(t) = (r\cos t, r\sin t), 
 \qquad 
 q^2(t)=-q^1(t),
\end{equation}
$t\in \mathbb{R}$, is a geodesic in~$\mathcal{L}^2(\mathbb{R}^2)$
if and only if  $r$ is the solution of the equation $\gamma_0-\gamma(2r) + r\gamma'(2r)=0$. 
\end{proposition}
\begin{proof}
For orbit~\eqref{circular}
it is the case that~$\overline q\equiv0$,
$\delta q = q^1$ and $\rho\equiv2r$; also
$p_1=\big(\gamma_0-\gamma(\rho)\big)^{-1}\dot{q}^1$ and
$p_2=-p_1$, so that $\overline{p}=0$ and
$\delta p = p_1$. The first three equations 
of~\eqref{2ptgeodeq} can easily be checked;
the fourth one holds if and only if~$\gamma_0-\gamma(2r) + r\gamma'(2r)=0$.
\end{proof}
(The above result was also proven by Fran\c{c}ois-Xavier Vialard of Imperial College, London.)
Orbit~\eqref{circular} has the property that at time $\pi$, $q^1$ and $q^2$ interchange their positions: it is a geodesic from the set of landmark points~$\big((r,0), (-r,0)\big)\in \mathcal{L}^2(\mathbb{R}^2)$ to the set $\big((-r,0), (r,0)\big)
\in \mathcal{L}^2(\mathbb{R}^2)$. But if these points live in~$\mathbb{R}^3$, they can move around each other in any plane containing the points. Thus we have a circle of geodesics in $\mathcal{L}^2(\mathbb{R}^3)$:
$$ q^1(t) = (r\cos t, r\cos\theta \sin t , r\sin \theta \sin t ),
\qquad 
q^2(t) = -q^1(t) $$
all connecting $\big((r,0,0), (-r,0,0)\big)$ 
to $\big((-r,0,0), (r,0,0)\big)$,
for any~$\theta\in[0,2\pi)$. This is exactly like all the lines of fixed longitude connecting the north and south pole on the 2-sphere and means that one set of landmark points is a conjugate point  of the other in $\mathcal{L}^2(\mathbb{R}^3)$. This is the simplest example of how geodesics between landmark points must avoid collisions and so make a choice between different possible detours, leading to conjugate points and thus positive curvature.
\section{Conclusions}
We believe that $\mathcal{L}^N(\mathbb{R}^D)$, the Riemannian manifold of~$N$ landmark points in~$D$ dimensions, is a fundamental object for differential geometry and that we have only scratched the surface in its study. We started with a basic formula which computes sectional curvature of a Riemannian manifold in terms of the cometric, its partial derivatives, and the metric itself (but not its derivatives). This is particularly adapted to computing curvature for manifolds which arise as submersive quotients of other manifolds and gives O'Neill's formula as a corollary. We then applied this to derive a formula for sectional curvature of the space of landmarks. This formula is not simple but, like Arnold's formula for curvature of Lie groups under left- (or right)-invariant metrics, splits into a sum of four terms. The four terms involve interesting intermediate expressions in the two vectors (or co-vectors) which define the section and which have relatively simple geometric interpretations. We called these the {\it mixed force}, the {\it discrete vector strain}, the {\it scalar compression} and the {\it landmark derivative}. The geodesic equation in its Hamiltonian form is quite simple and involves the force as expected. We also gave several concrete examples to illustrate the nature of these geodesics.
\par
Finally, we have examined in detail the case of geodesics in which only one or two landmark points have non-zero momenta, and computed the curvature in sections spanned by such geodesics. We found that in this case there are essentially two sources of positive curvature. One can understand them through the non-uniqueness of geodesics joining two $N$-tuples: the first sort of non-uniqueness is caused by the two points with non-zero momentum choosing between converging in the middle of the geodesic (``car-pooling'') or moving independently and not converging; the second occurs only when~$D\geq 3$ and arises when the same two points need to get around each other and must choose on which side to pass (if $D=2$, this sort non-uniqueness also occurs but comes from non-trivial topology, not curvature). 
\par
One of the most important questions left open is explore how prevalent positive curvature is in general, i.e.\ for geodesics in which all points carry momentum. Answering this question is central to applications of landmark space in which geodesics are actually computed. One might hope that the picture for two momenta is true in general but this is far from clear. It seems interesting to explore whether there is some sort of ``index'' for curvature forms --- a numerical measure of how much positive vs.\ negative curvature is present. Another important question is to explore the shape of the coefficients in~\eqref{2ptcurvature} for different kernels. More generally, what is the impact of different kernel types (Bessel, Gaussian, Cauchy) on the corresponding geodesics? Finally, note that all kernels have a length constant built into their definition so the geometry of the space of landmarks is far from scale invariant. Thus one should analyze what happens asymptotically when the points are very close relative to this constant or are very far from each other.
\bibliography{MMM1-bib}
\bibliographystyle{abbrv}
\end{document}